\documentclass[11pt,a4paper]{article}
\usepackage{amssymb,amsmath,mathrsfs,enumerate}

\usepackage[colorlinks=true, pdfstartview=FitV, linkcolor=blue, citecolor=blue, urlcolor=blue,pagebackref=false]{hyperref}

\usepackage{tikz}
\usepackage{float}
\usepackage[font=footnotesize]{caption}

\usepackage{times,theorem,latexsym,color,comment}

\newcommand{\ul}[1]{\underline{#1}}

\newcommand{\BOX}{\ensuremath\Box}

\newtheorem{theorem}{Theorem }[section]

\newtheorem{corollary}[theorem]{Corollary}
\newtheorem{definition}[theorem]{Definition}
{\theorembodyfont{\rmfamily}}
{\theorembodyfont{\rmfamily}}
{\theorembodyfont{\rmfamily}}
\newtheorem{lemma}[theorem]{Lemma}
\newtheorem{proposition}[theorem]{Proposition}
{\theorembodyfont{\rmfamily}\newtheorem{remark}[theorem]{Remark}}
{\theorembodyfont{\rmfamily}}

\newcommand{\C}{\mathbb{C}}
\newcommand{\R}{\mathbb{R}}
\newcommand{\Z}{\mathbb{Z}}
\newcommand{\dd}{\,{\rm d}}
\newcommand{\ep}{\varepsilon}
\newcommand{\ii}{\mathrm i}

\newcommand{\B}{\mathbb{B}}

\newcommand{\F}{\mathcal{F}}
\newcommand{\mcG}{\mathcal{G}}

\newcommand{\mcW}{\mathcal{W}}
\newcommand{\V}{\mathcal{V}}
\newcommand{\Q}{\mathcal{Q}}

\DeclareMathOperator{\spn}{span}

\def\Xint#1{\mathchoice
	{\XXint\displaystyle\textstyle{#1}}%
	{\XXint\textstyle\scriptstyle{#1}}%
	{\XXint\scriptstyle\scriptscriptstyle{#1}}%
	{\XXint\scriptscriptstyle\scriptscriptstyle{#1}}%
	\!\int}
\def\XXint#1#2#3{{\setbox0=\hbox{$#1{#2#3}{\int}$}
		\vcenter{\hbox{$#2#3$}}\kern-.5\wd0}}

\def\dashint{\Xint-}

\newenvironment{proof}{{\vskip\baselineskip\noindent\textbf{Proof:}}}%
{\hspace*{.1pt}\hspace*{\fill}\BOX\vskip\baselineskip}

\newenvironment{proofx}[1]%
{\vskip\baselineskip\noindent\textbf{Proof of {#1}:}}%
{\hspace*{.1pt}\hspace*{\fill}\BOX\vskip\baselineskip}
{\vskip\baselineskip\noindent\textbf{Proof of Theorem \protect\ref{#1}:}}%
{\hspace*{.1pt}\hspace*{\fill}\BOX\vskip\baselineskip}
{\vskip\baselineskip\noindent\textbf{Proof of Theorems \protect\ref{#1} --
		\protect\ref{#2}:}}%
{\hspace*{.1pt}\hspace*{\fill}\BOX\vskip\baselineskip}

\begin{document}

\title{Higher-order boundary layers and regularity \\ for Stokes systems over rough boundaries}

\author{Mitsuo Higaki, Jinping Zhuge}	
\date{}

\maketitle

\noindent {\bf Abstract.}\ In this paper, we study the large-scale boundary regularity for the Stokes system in periodically oscillating John domains. Our main contribution is the construction of boundary layer correctors of arbitrary order. This is a significant generalization of the known results restricted to the first and second orders. As an application, we prove the large-scale regularity estimate, as well as a Liouville theorem, of arbitrary order for the Stokes system. Our results are also related to higher-order boundary layer tails and wall laws in viscous fluids over rough boundaries.

\medskip


\medskip


\medskip

\section{Introduction}
\subsection{Motivations}
The effect of rough boundaries is an important and challenging topic in the field of hydrodynamics. It has been extensively studied experimentally, numerically and theoretically with a vast literature during the past half century. A remarkable progress has been made in deriving the wall law for viscous fluids, which is a method to replace rough boundaries artificially by flat ones and to impose the effective conditions on them. The reader is referred to the preceding works \cite{HP20,HPZ} and references therein for more detailed backgrounds, including approximation of flows via the wall laws and the investigation of turbulence.

In this paper, we focus on the local behavior of the stationary Stokes flows over periodically oscillating boundaries. There are two critical parameters for the geometry 
of a rough boundary --- thickness and structure. In applications, the thickness of rough boundaries is relatively small, compared to the characteristic scales of fluid flows. The periodicity is a mathematical simplification of realistic boundaries having self-similar structures, such as rusted or sanded metal plates, natural riverbeds, fish skins, etc.  In our situation, since the Stokes system is linear, the thickness and period of rough boundaries can be enlarged and rescaled respectively to 1 and $2\pi$ by choosing the period suitably (e.g., a multiple of the smallest period). Precisely, we assume that the domain $\Omega \subset \R^d$ as well as its boundary $\Gamma := \partial \Omega$ satisfies the following two basic assumptions: 
\begin{itemize}
    \item Thickness:
    \begin{equation}\label{cdn.thick}
        \R^d_+: = \{ (x,y)\in \R^d~|~ y >0  \} \subset \Omega \subset  \{ (x,y)\in \R^d~|~ y >-1  \};
    \end{equation}

    \item Periodicity:
    \begin{equation}\label{cdn.periodic}
        \Omega + (2\pi z,0) = \Omega \quad \text{for any } z\in \Z^{d-1},
    \end{equation}
    where $\Omega + (\tau,0) := \{ (x+\tau,y)~|~ (x,y) \in \Omega \} $ for fixed $\tau\in\R^{d-1}$.
\end{itemize}
Throughout this paper, the point in $\R^d$ will always be written as $(x,y)$ with $x = (x_1, x_2,\cdots, $ $ x_{d-1})\in \R^{d-1}$ and $y\in \R$. 
Calligraphy letters $\mathcal{X}, \mathcal{Y}, \cdots$ will also be used only when it makes the description more concise.
The condition \eqref{cdn.thick} says that the boundary thickness of $\Gamma$ is at most 1. The condition \eqref{cdn.periodic} says that the domain $\Omega$ 
is $(2\pi)$-periodic in all $(d-1)$-horizontal directions. Another essential technical geometric assumption for $\Omega$ 
will be given later.

Let $R>1$ and let $Q_R(0)\subset\R^d$ denote the cube centered at the origin with side length $2R$, i.e., $Q_R(0)=(-R,R)^d$. Set $B_{R,+} = Q_R(0)\cap \Omega$ and $\Gamma_R = Q_R(0) \cap \Gamma$. Then we consider the following Stokes system subjected to the no-slip boundary condition:
\begin{equation}\tag{S}\label{intro.S}
\left\{
\begin{array}{ll}
-\Delta u+\nabla p=0&\mbox{in}\ B_{R,+}\\
\nabla\cdot u=0&\mbox{in}\ B_{R,+}\\
u=0&\mbox{on}\ \Gamma_R,
\end{array}
\right.
\end{equation}
where $u = (u_1,\cdots, u_d)$ denotes the velocity field and $p$ denotes the pressure field. Our interest is the mesoscopic structures or large-scale behaviors of an arbitrary weak solution $(u,p)$ in $B_{r,+}$ for $1<r<R$, typically when $R$ is large. If the boundary is flat, i.e., if $\Omega$ and $\Gamma$ are respectively $\R^d_+$ and $\partial\R^d_+ := \{(x,y)\in \R^d~|~y=0\}$, the solution $(u,p)$ of \eqref{intro.S} is smooth up to the boundary $\Gamma_{R/2}$. Particularly, we have the Taylor expansion for $u$
\begin{equation}\label{eq.TaylorExp}
    u(x,y) = \sum_{\substack{|\alpha|+l \le m \\ l\ge 1} } \frac{1}{\alpha! l!} \partial_x^\alpha \partial_y^l u(0,0) x^\alpha y^l + O(|(x,y)|^{m+1}). 
\end{equation}
In the above summation, $l$ is not equal to zero by the boundary condition. Moreover, it can be verified that the polynomial on the right-hand side of \eqref{eq.TaylorExp} solves the Stokes system in $\R_+^d$. In other words, if $\Gamma$ is flat, for any $m\ge 1$, the weak solution $u$ is close to a polynomial solution of degree $m$ in $B_{r,+}$ with a higher-order remainder. A similar consideration can also be made for the pressure $p$. These properties demonstrate the local structures and $C^{m,1}$ regularity of the solutions to \eqref{intro.S} when the boundaries are flat. Now, it is natural to ask if there is an analog of \eqref{eq.TaylorExp} if $\Omega$ has a periodically oscillating boundary.

We point out that similar questions have been well-understood in elliptic homogenization with oscillating coefficients when there is no boundary. Roughly it says that the local solutions can be approximated by heterogeneous polynomial solutions involving higher-order correctors; see e.g., \cite{AL89,AKM16,AKS20} or \cite{AKM19}. Turning to our problem, however, the oscillation is imposed on the boundary instead of coefficients, which leads to a completely different situation at least at a technical level. Recently, the local large-scale regularity (lower-order) to PDEs with oscillating boundaries has been investigated in several works. The uniform Lipschitz estimate was first obtained in \cite{KP15, KP18} for elliptic equations by the compactness method in the case when the boundary is a Lipschitz graph; see also \cite{GZ20} for related work. The Lipschitz graph assumption then was removed in \cite{Z21} so that the boundaries can be arbitrary at small scales. Similar results, up to $C^{2,\mu}$ estimates, have been extended to stationary Navier-Stokes system in \cite{HP20,HPZ}, which will be discussed below.

The key to deriving the higher-order analog of \eqref{intro.S} is to introduce the so-called boundary layers, which play similar roles as the correctors in elliptic homogenization with oscillating coefficients; see e.g., \cite{AL89,AKM16,KL16,Gu17,BG19,AKS20,AKS21}. The first-order boundary layers are classical, given by the bounded solutions of the following cell problem: for $i\in\{1,2,\cdots, d\}$, 
\begin{equation}\label{intro.eq.1st.BL}
\left\{
\begin{array}{ll}
-\Delta v_{(i)}+\nabla q_{(i)}=0&\mbox{in}\ \Omega \\
\nabla\cdot v_{(i)}=0&\mbox{in}\ \Omega \\
v_{(i)}(x,y) + y\mathbf{e}_i = 0 &\mbox{on}\ \Gamma. 
\end{array}
\right.
\end{equation}
Historically, the first-order boundary layer was studied to derive the Navier wall law rigorously, which had been known to be effective empirically. The mathematical analysis was pioneered by \cite{AMPV97,APV98,JM01,JM03}. The extensions have been given to the case of more general boundaries \cite{BG08,M09,G09,GM10,DG11,BDG12} and to the nonstationary case \cite{MNN13,H16}. On the other hand, the connection to the regularity theory for the stationary Navier-Stokes system was investigated by \cite{HP20}. Indeed, the local large-scale $C^{1,\mu}$ regularity has been obtained in \cite{HP20} and \cite{HPZ}, namely that for a Navier-Stokes flow $u$ in $B_{r,+}$, there exists a constant $c_i$ such that
\begin{equation}\label{est.1st}
    \Big\| \nabla u 
    - \nabla \sum_{i=1}^{d-1} c_i (y\mathbf{e}_i +v_{(i)}) \Big\|_{\ul{L}^2(B_{r,+})} 
    \le C r^{\mu}, \quad r\in(1,R). 
\end{equation}
The above estimate is called a large-scale $C^{1,\mu}$ estimate since the radius is restricted to $r>1$. It in general does not hold for $r\le 1$ if the boundary has no structure such as smoothness at scales less than $1$. This is a natural phenomenon in homogenization theory.

More recently in \cite{HPZ}, $C^{2,\mu}$ regularity was firstly proved for the stationary Navier-Stokes system with the help of the second-order boundary layers. The proof relies on the periodicity assumption on the boundary. We emphasize that, in contrast, the first-order boundary layers can be constructed without periodicity. A typical second-order boundary layer is a weak solution to the Stokes system in $\Omega$ subjected to the boundary condition
\begin{equation*}
    v(x,y) + x_j y \mathbf{e}_i = 0 \quad \text{on } \Gamma, 
\end{equation*}
where $j\in\{1,2,\cdots,d-1\}$. The difficulty in solving this cell problem is that the boundary value has linear growth when $x_j\to \infty$. As a result, unlike the first-order case, one cannot rely on the Dirichlet-Neumman map \cite{GM10, DP14, DG17, KP18, HP20} or the Green function \cite{HPZ}, which are useful for the solutions with finite Dirichlet energy. The key observation in \cite{HPZ} is that 
\begin{equation*}
    v(x,y) - x_j v_{(i)}(x,y) = 0 \quad \text{on } \Gamma 
\end{equation*}
holds, where $v_{(i)}$ is the first-order boundary layer given by \eqref{intro.eq.1st.BL}. Hence if we introduce a new unknown function $V := v - x_j v_{(i)}$, then it satisfies the no-slip condition on $\Gamma$ and a Stokes system with periodic ``force terms'' involving $v_{(i)}$. This transformation allows one to find a solution $V$ on a periodic cell, which is an infinite cylinder in our setting. When solving the problem, we need to take advantage of a good quantitative convergence of $v_{(i)}$ when $y\to\infty$, which essentially requires a periodic structure of the boundary.

As indicated, the boundary layers are correctors that adjust to the oscillating boundaries and are crucial to understanding the local behavior of the solutions to \eqref{intro.S}. In this paper, we will construct the boundary layers of arbitrary order by recognizing their algebraic structures and recurrence relations. As far as the authors know, this is a new result even in smooth domains. The boundary layers will be used to derive a higher-order version of \eqref{est.1st} analogous to the classical estimate \eqref{eq.TaylorExp}, which directly leads to the Liouville theorem for the homogeneous Stokes system on $\Omega$. We also introduce a notion of the effective Stoke polynomials and provide the higher-order wall laws in a local sense. This new class of boundary conditions contains the Navier wall law in \cite{JM01,JM03,G09,GM10} as the first-order case.

\subsection{Assumptions and statements of results}
In the analysis of \eqref{intro.S}, other than the assumptions \eqref{cdn.thick} and \eqref{cdn.periodic}, we still need to assume a technical geometric condition on $\Omega$ as in \cite{HPZ}. This is recalled in Definition \ref{def.John2} below.

%
\begin{definition}[\cite{Jo61,MS79}]\label{def.John}
{\rm 
Let $D \subset \R^d$ be an open bounded set and $\tilde{\mathcal{X}} \in D$. We say that $D$ is a (bounded) {\it John domain} with respect to $\tilde{\mathcal{X}}$ and with constant $L$ if for any $\mathcal{Y}\in D$, there exists a Lipschitz mapping $\rho:[0,|\mathcal{Y}-\tilde{\mathcal{X}}|] \to D$ with Lipschitz constant $L\in (0,\infty)$, such that $\rho(0) = \mathcal{Y}, \rho(|\mathcal{Y}-\tilde{\mathcal{X}}|) = \tilde{\mathcal{X}}$ and ${\rm dist}\,(\rho(t),\partial D) \ge t/L$ for all $t\in [0, |\mathcal{Y}-\tilde{\mathcal{X}}|]$.
}
\end{definition}
%

The boundary of a John domains allows for fractals and inward cusps. Lipschitz domains, NTA domains, domains with inward cusps or certain fractals such as Koch's snowflake are John domains. Domains with outward cusps are not.

%
\begin{definition}[\cite{HPZ}]\label{def.John2}
{\rm 
Let $\Omega \subset \R^d$ be an open set satisfying \eqref{cdn.thick}. 

\begin{enumerate}[(i)]

\item We say that $\Omega$ is a {\it bumpy John domain} with constants $(L,K)$ if for any $\mathcal{X} \in \{(x,0) \in \R^d \}$ and any $R\ge 1$, there exists a bounded John domain $\Omega_R(\mathcal{X})$ with respect to $\mathcal{X}_R = \mathcal{X}+R \mathbf{e}_d$ and with constant $L\in(0,\infty)$ according to Definition \ref{def.John} such that
\begin{equation}\label{e.exisinterdom}
B_{R,+}(\mathcal{X}) \subset \Omega_R(\mathcal{X}) \subset B_{2R,+}(\mathcal{X}).
\end{equation}

\item We say that $\Omega$ is a {\it periodic bumpy John domain} with constants $(L,K)$ if $\Omega$ is a bumpy John domain with constants $(L,K)$ and satisfies the condition \eqref{cdn.periodic}.
\end{enumerate}
}
\end{definition}
%

It is easily seen that the constants of bumpy John domains are rescaling- and translation-invariant. We will say that a constant $C$ depends on $\Omega$ if it depends on the constant $(L,K)$ in Definition \ref{def.John2} as well as the dimension $d$. A typical example of the periodic bumpy John domains is one with a porous boundary (see Figure \ref{fig.1}).

%
\begin{figure}[h]
	\begin{center}
		\includegraphics[scale =0.18]{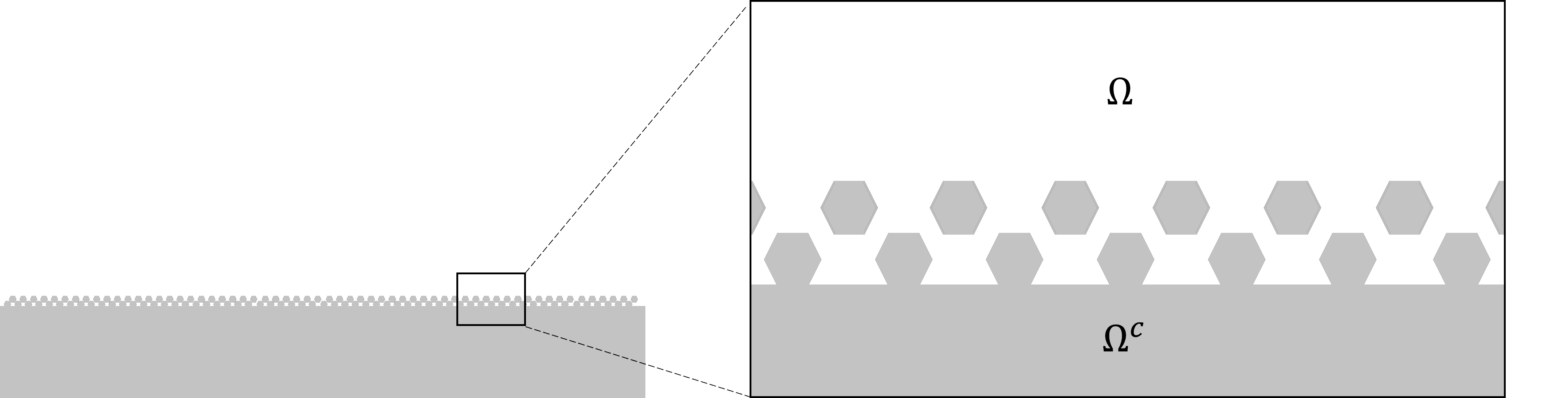}
	\end{center}
	\caption{An example of periodic bumpy John domains with a porous boundary}\label{fig.1}
\end{figure}
%

The John domain assumption can be thought of as the minimum boundary requirement when dealing with the stationary Stokes system. Indeed, it is proved in \cite{ADM06} that, under suitable conditions, the (bounded) domain in which a right-inverses of the divergence operator exists must be a John domain. The right-inverse of the divergence, also known as the Bogovskii operator, plays an important role for estimating the pressure. Let us mention that Definition \ref{def.John2} was first introduced in \cite{HPZ}, where the first-order boundary layers were constructed over non-periodic boundaries. The proof uses the Green function and its properties following from the a priori large-scale Lipschitz regularity for \eqref{intro.S}. These enable one to establish a local energy estimate of the pressure in bumpy John domains. As is discussed in \cite{HPZ}, it is unclear whether the method in \cite{GM10, DP14, DG17, KP18, HP20} applies in this case. A major source of difficulty is that the boundary is not necessarily represented by a graph.

Our main result of this paper is the existence and the corresponding estimates of the boundary layers of arbitrary order. Let $\mathcal{P}_m(\R^d)$ be the set of all real-coefficient polynomials of at most degree $m$ and set 
$$\mathcal{P}_{m,0}(\R^d_+) = \{ P\in \mathcal{P}_m(\R^d)~|~ P(x,0) = 0 \}\subset \mathcal{P}_m(\R^d).$$

\begin{theorem}\label{intro.thm.bl}
Let $\Omega$ be a periodic bumpy John domain. For all $m\in\Z_{+}$, there exists a linear map $\mathscr{S}[\cdot]: \mathcal{P}_{m,0}(\R^d_+)^d \rightarrow H^1_{\rm loc}(\overline{\Omega})^d\times L^2_{\rm loc}(\overline{\Omega})$ which maps $P \in \mathcal{P}_{m,0}(\R^d_+)^d$ to the pair $\mathscr{S}[P]=(v,q)$ which is a weak solution to  
\begin{equation}\label{eq1.intro.thm.bl}
\left\{
\begin{array}{ll}
-\Delta v + \nabla q=0&\mbox{in}\ \Omega \\
\nabla\cdot v=0&\mbox{in}\ \Omega \\
v + P = 0 &\mbox{on}\ \Gamma. 
\end{array}
\right.
\end{equation}
Moreover, the growth rate of $v$ at spatial infinity is strictly less than that of $P$.
\end{theorem}
%

The above theorem is a straightforward corollary of Proposition \ref{thm.bl.alpha} which also contains several basic estimates (details are omitted here for simplicity). Theorem \ref{intro.thm.bl} recovers the classical first-order boundary layers if $m=1$ and the second-order boundary layers in \cite{HPZ} if $m=2$. The structure of $\mathscr{S}[P]$, nontrivially mixing polynomials and periodic functions, can be seen in Subsection \ref{subsec.intro.key}; see also Proposition \ref{prop.S.formula} for an intrinsic formula for $\mathscr{S}[P]$.

As an important application, using the boundary layer correctors in Theorem \ref{intro.thm.bl}, one can show the higher-order regularity for the weak solutions of \eqref{intro.S}. To state the result, we introduce some function spaces in a convenient manner. The precise definitions and their properties are given in Section \ref{sec.stokespoly} and Subsection \ref{subsec.het.stokespoly}. Let $\mathcal{S}_m(\R^d_+)$ be the space of Stokes polynomials $(P,Q)$ in $\R^d_+$, that is the set of polynomial solutions of degree at most $m$ to the Stokes equations imposed on $\R_+^d$. Moreover, let $\mathcal{S}_m(\Omega)$ be the space of heterogeneous Stoke polynomials in $\Omega$: 
\begin{equation}
\mathcal{S}_m(\Omega) 
= \{ (w,\pi) = (P,Q) + \mathscr{S}[P]~|~ (P,Q) \in\mathcal{S}_m(\R^d_+) \}.
\end{equation}
Notice that the boundary layer $\mathscr{S}[P]$ corrects the trace of the polynomial $P$ on $\Gamma$. Hence $(w,\pi)$ is a solution of the Stokes system in $\Omega$ having the same growth rate as $(P,Q)$.

Now we can state the optimal higher-order regularity for \eqref{intro.S} at large scales. 
%
\begin{theorem}\label{thm.main}
Let $\Omega$ be a periodic bumpy John domain. For all $R>1$ and $m\in\Z_{+}$, if $(u, p) \in H^1(B_{R,+})^d \times L^2(B_{R,+})$ is a weak solution of \eqref{intro.S}, then there exist  $(w,\pi)\in\mathcal{S}_m(\Omega)$ and a constant $c_p$ such that, for all $r\in (1, R/2)$, 
\begin{equation}\label{est.mainRegularity}
\begin{aligned}
\|\nabla u- \nabla w\|_{\ul{L}^2(B_{r,+})}
+ \|p-\pi-c_p\|_{\ul{L}^2(B_{r,+})}
\le 
C\Big(\frac{r}{R}\Big)^{m}
\|\nabla u\|_{\ul{L}^2(B_{R,+})}. 
\end{aligned}
\end{equation}
The constant $C$ is independent of $R$ and $r$, while depends on $\Omega$ and $m$.
\end{theorem}
%

%
\begin{remark}\label{rem.thm.main}
One can obtain the estimate of $u$ itself if the Caccioppoli inequality holds for \eqref{intro.S} in $B_{R,+}$. If it does, from \eqref{est.mainRegularity} combined with the Poincar\'{e} inequality, we find
\begin{equation*}
\begin{aligned}
\|u- w\|_{\ul{L}^2(B_{r,+})}
\le 
C\Big(\frac{r}{R}\Big)^{m+1}
\|u\|_{\ul{L}^2(B_{R,+})}, \quad r\in (1, R/4). 
\end{aligned}
\end{equation*}
However, the validity of the usual Caccioppoli inequality in bumpy John domains seems to be a delicate question of independent interest. We refer to the Appendix A of \cite{HPZ}. 
\end{remark}
%

%
\begin{remark}
With the estimates of boundary layers in Proposition \ref{thm.bl.alpha}, the proof for Theorem \ref{thm.main} is more or less standard via a quantitative Campanato iteration originating from \cite{AS16} (also see \cite{GNO20,Shen17} or recent monographs \cite{Shen18,AKM19}). The application of this method to Stokes systems or linear elasticity involving the pressure can be found in \cite{GZ19, GZ20, HPZ}.
\end{remark}
%

Theorem \ref{thm.main} is essentially a large-scale version of \eqref{eq.TaylorExp} with the Stokes polynomial replaced by the heterogeneous Stokes polynomial from $\mathcal{S}_m(\Omega)$. This result can be compared to the large-scale $C^{m,1}$ regularity in elliptic homogenization with oscillating coefficients (e.g., \cite{AKM16, AKS20}).
Theorem \ref{thm.main} also implies the Liouville theorem of arbitrary order for the homogeneous Stokes system imposed on the whole domain $\Omega$. This basically says that $\mathcal{S}_m(\Omega)$ contains all the solutions in $\Omega$ whose velocity gradients grow at most like $o(|(x,y)|^{m})$ and has the same dimension as $\mathcal{S}_m(\R^d_+)$; see Subsections \ref{subsec.het.stokespoly} and \ref{subsec.liouville} for details.

Note that Theorem \ref{thm.main} captures the structure of the solutions near the boundary. Then it is natural to ask what happens to a local solution away from the boundary. In fact, the singularity only appears near the boundary and homogenization should be expected to take place far away from the boundary. Actually, it is well-known that the first-order boundary layers converge exponentially to constant values as $y\to \infty$. These limits at $y=\infty$ are called the boundary layer tails. This motivates us to define the boundary layer tails for higher-order boundary layers, and therefore to define the effective Stokes polynomials in this paper. Indeed, we can see from Proposition \ref{thm.bl.alpha} that the boundary layer $\mathscr{S}[P]$ given in Theorem \ref{intro.thm.bl} converges exponentially to a Stokes polynomial as $y\to \infty$. Consequently, the element in $\mathcal{S}_m(\Omega)$, which is a special solution, also converges exponentially to a Stokes polynomial. Thus, we may define the space of effective Stokes polynomials as  
\begin{equation*}
\begin{aligned}
\mathcal{S}_m^{\rm eff}(\Omega) 
= & \big\{(w_{\rm poly},\pi_{\rm poly}) \in \mathcal{P}_m(\R^d)^d \times \mathcal{P}_{m-1}(\R^d)~|~\\
& \qquad \lim_{y\to +\infty} ((w,\pi)-(w_{\rm poly},\pi_{\rm poly}))=0, \ \ 
(w,\pi)\in \mathcal{S}_m(\Omega) \big\}.
\end{aligned}
\end{equation*}

The following theorem shows that the behavior of the local solutions, slightly away from the boundary, may be pointwise captured by the effective Stokes polynomials in $\mathcal{S}_m^{\rm eff}(\Omega)$, instead of the heterogeneous Stokes polynomials in $\mathcal{S}_m(\Omega)$.

%
\begin{theorem}\label{thm.pointwise}
Under the same assumptions as in Theorem \ref{thm.main}, if $R$ is sufficiently large depending on $m$, then there exist $(w_{\rm poly}, \pi_{\rm poly}) \in \mathcal{S}_m^{\rm eff}(\Omega)$ and a constant $c_p$ such that for all $(x,y) \in B_{R/2,+}$ with $y\ge 4$,
\begin{equation}\label{est.AporByEff}
\begin{aligned}
& |\nabla u(x,y) - \nabla w_{\rm poly}(x,y)| + |p(x,y) - \pi_{\rm poly}(x,y) - c_p| \\ & \le C \Big\{  \Big(\frac{r}{R}\Big)^m +  (1+|x|)^{m-1} e^{-y/2}  \Big\} \| \nabla u \|_{\ul{L}^2(B_{R,+})},
\end{aligned}
\end{equation}
and
\begin{equation}\label{est.AporByEff.1}
|u(x,y) - w_{\rm poly}(x,y)| \le C \Big\{  y \Big(\frac{r}{R}\Big)^m +  (1+|x|)^{m-1} e^{-y/2}  \Big\} \| \nabla u \|_{\ul{L}^2(B_{R,+})},
\end{equation}
where $r = |(x,y)|$ and $C$ depends only on $m$ and $\Omega$.
\end{theorem}
%

The estimate \eqref{est.AporByEff} shows that if $y\ge 2m\ln(R)$, then the solution of \eqref{intro.S} can be approximated equally well by the effective Stokes polynomials as the approximation in \eqref{est.mainRegularity} by the heterogeneous Stokes polynomials. Moreover, from the proof, which is a combination of Theorem \ref{thm.main} and Proposition \ref{thm.bl.alpha}, one can see that the effective Stokes polynomial $(w_{\rm poly}, \pi_{\rm poly})$ in Theorem \ref{thm.pointwise} is exactly the limiting polynomial of $(w,\pi)$ in Theorem \ref{thm.main}.

Observe that the effective Stokes polynomials in $\mathcal{S}_m^{\rm eff}(\Omega)$ do not satisfy the no-slip condition either on $\partial \Omega$ or $\partial \R^d_+$. However, in view of Theorem \ref{thm.pointwise}, the condition on $\partial \R^d_+$ satisfied by $w_{\rm poly}$ can be regarded as the effective boundary condition for the solution $(u,p)$. We would like to understand it as the wall law for $(u,p)$ in a local setting. This interpretation is indeed compatible with the literature, in the sense that the condition satisfied by the elements from the first-order space $\mathcal{S}_1^{\rm eff}(\Omega)$ is the classical Navier wall law: 
\begin{equation*}
    w_{\rm poly}(x,0) = M \partial_y w_{\rm poly}(x,0)
    \quad \text{for any }
    (w_{\rm poly}, \pi_{\rm poly}) \in \mathcal{S}_1^{\rm eff}(\Omega),
\end{equation*}
where $M$ is an intrinsic constant ($d\times d$)-matrix determined only by $\Omega$ whose last column is $0$. This fact was addressed by \cite{HP20} in a rescaled setting and is dubbed the local Navier wall law. We point out that the Navier wall law plays a role as a corrected boundary condition (versus the no-slip boundary condition) on the flattened boundary that provides more accurate and effective numerical and theoretical calculations \cite{APV98,JM01,JM03,G09,GM10,MNN13,H16}. In Subsection \ref{subsec.wall.law}, we identify the local wall laws of arbitrary order. As a simple and concrete example, we compute the second-order wall law in the 2D case in Subsection \ref{subsec.Example}.

\subsection{Key ideas of the proof of Theorem \ref{intro.thm.bl}}\label{subsec.intro.key}
Let us explain the key ideas of the proof of Theorem \ref{intro.thm.bl}, which is the core of this paper. The goal is to build a weak solution $(v^\alpha, q^\alpha) = (v^{\alpha,l}_{(i)}, q^{\alpha,l}_{(i)})$ to the Stokes system 
\begin{equation}\tag{BL$^{\alpha}$}\label{intro.eq.bl.alpha}
\left\{
\begin{array}{ll}
-\Delta v^{\alpha} + \nabla q^{\alpha}=0 &\text{in } \Omega \\
\nabla\cdot v^{\alpha}=0 &\text{in } \Omega \\
v^{\alpha}(x,y)+x^\alpha y^l{\bf e}_i=0 &\text{on } \Gamma 
\end{array}
\right.
\end{equation}
for a given triplet $(\alpha,l,i)\in\Z_{\ge0}^{d-1}\times \Z_{+} \times \{1,\cdots,d\}$. Theorem \ref{intro.thm.bl} then follows by linearity of the equations. Notice that \eqref{intro.eq.bl.alpha} coincides with \eqref{intro.eq.1st.BL} if $(\alpha,l)=(0,1)$ and is a system for the second-order boundary layers if $|\alpha|+l=2$. Again, the difficulty in solving the problem \eqref{intro.eq.bl.alpha} is that the boundary value has growth when $|x|\to \infty$. The method in \cite{HPZ} is no longer applicable since it is restricted to cases when $|\alpha|+l=2$.

The new idea is to introduce an ansatz 
\begin{equation}\label{intro.ansatz}
\begin{aligned}
v^{\alpha}(x,y) = 
\sum_{\beta\le\alpha} 
\dbinom{\alpha}{\beta} x^{\alpha-\beta} V^{\beta}(x,y), \qquad
q^{\alpha}(x,y) = 
\sum_{\beta\le\alpha} 
\dbinom{\alpha}{\beta} x^{\alpha-\beta} Q^{\beta}(x,y), 
\end{aligned}
\end{equation}
where the family of new unknown functions $\{(V^{\beta},Q^{\beta})\}$ is assumed to be periodic in $x$. This ansatz is adopted in view of the Taylor expansion of $(v^{\alpha},q^{\alpha})$ at the origin but with variable coefficients. It systematically generalizes the heuristic method in \cite{HPZ}. In order to recover the boundary condition on $v^\alpha$, one needs compatibility conditions on $\Gamma$ for $\{(V^{\beta},Q^{\beta})\}$. Here we assume that $(V^0,Q^0)$ is the solution bounded at infinity of \eqref{intro.eq.bl.alpha} with $\alpha=0$ and that $V^{\beta}=0$ on $\Gamma$ for all $\beta\neq0$. The unique existence of $(V^0,Q^0)$ can be proved easily thanks to the small vertical extent of $\Omega\setminus\R^d_+$; see Lemma \ref{prop.bl.zero}.

Inserting \eqref{intro.ansatz} to \eqref{intro.eq.bl.alpha}, we see that $(V^{\beta},Q^{\beta})$ must solve the recursive system  
\begin{equation}\label{intro.eq.bl.beta}
\begin{aligned}
-\Delta V^{\beta} + \nabla Q^{\beta}
&= \sum_{i=1}^{d-1}
\Big(
2\dbinom{\beta_i}{2} 
V^{\beta-2{\bf e}_i}
+ 2\dbinom{\beta_i}{1} 
\partial_i V^{\beta-{\bf e}_i}
- \dbinom{\beta_i}{1} 
Q^{\beta-{\bf e}_i} {\bf e}_i
\Big), \\
\nabla\cdot V^{\beta}
&=
-\sum_{i=1}^{d-1}
\dbinom{\beta_i}{1} 
V^{\beta-{\bf e}_i}_i.
\end{aligned}
\end{equation}
Here all the notations with multi-indices belonging to $\Z^d\setminus\Z_{\ge0}^d$ are understood to be zero. Despite its linearity, the analysis of the system \eqref{intro.eq.bl.beta} is not straightforward. In fact, the solutions of \eqref{intro.eq.bl.beta} can grow at spatial infinity, which prevents us to apply the Lax-Milgram lemma. To illustrate it, let us put $\beta={\bf e}_1$ in \eqref{intro.eq.bl.beta}. Then the second line reads
\begin{equation*}
\begin{aligned}
\nabla\cdot V^{{\bf e}_1}
&=
-V^{0}_1.
\end{aligned}
\end{equation*}
It can be shown that the right-hand side converges to a negative constant if $(l,i)=(1,1)$ and $\Omega\neq\R^d_+$, namely if $V^{0}$ is a nontrivial first-order boundary layer. We refer to \cite[Proposition 3]{HP20} for the proof when $d=3$. Consequently, the solutions of \eqref{intro.eq.bl.beta} periodic in $x$ may grow when $y\to\infty$, otherwise the equation can be inconsistent. In addition, we need to consider the solutions with polynomial growth in $y$ as the recurrence proceeds.

The insight for this obstacle is to split $(V^{\beta},Q^{\beta})$ into 
\begin{equation*}
V^{\beta}(x,y)=V^{\beta}_{\rm poly}(y)+V^{\beta}_{\rm per}(x,y), \qquad 
Q^{\beta}(x,y)=Q^{\beta}_{\rm poly}(y)+Q^{\beta}_{\rm per}(x,y),
\end{equation*}
and study the pairs $(V^{\beta}_{\rm poly},Q^{\beta}_{\rm poly})$ and $(V^{\beta}_{\rm per},Q^{\beta}_{\rm per})$ in the two different manners:
\begin{itemize}
\item $V^{\beta}_{\rm poly}$ and $Q^{\beta}_{\rm poly}$ are polynomials in $y$ of degree at most $|\beta|$. Roughly speaking, both are obtained by the integration of $V^{\beta-{\bf e}_i}_{\rm poly}, V^{\beta-2{\bf e}_i}_{\rm poly}$ and $Q^{\beta-{\bf e}_i}_{\rm poly}$. However, a simple computation does not work due to the vectorial structure of the problem \eqref{intro.eq.bl.beta}.

\item $V^{\beta}_{\rm per}$ and $Q^{\beta}_{\rm per}$ are periodic functions in $x$ decaying exponentially when $y\to\infty$. Although they are represented explicitly on flat domains $\{y=L\}$ for $L>1$, we need to be careful in deriving the quantitative bounds when $|\beta|$ is large.
\end{itemize}

The procedure outlined above is carried out in Section \ref{sec.bl}. We rely on induction on the length $|\beta|$ of $\beta\in\Z_{\ge0}^{d-1}$ in the proof, but the details are complex in nature, as at each level we need to estimate the decomposed $(V^{\beta},Q^{\beta})$. Eventually, the solutions of \eqref{intro.eq.bl.alpha} are constructed based on the ansatz \eqref{intro.ansatz}, as stated in Proposition \ref{thm.bl.alpha}.

\subsection{Notations and definitions}

The following is a summary of the notations and definitions used in this paper.

\noindent \underline{\it Domains.} For $R>1$ and $(x,y)\in\R^d$, we define $Q_R = Q_R(0)= (-R,R)^d$ and
\begin{equation*}
    Q_R(x,y) = Q_R(0) + (x,y),
    \qquad
    B_{R,+}(x,0) = Q_R(x,0)\cap \Omega.
\end{equation*}
For simplicity, we write $B_{R,+} = B_{R,+}(0,0)$.

\noindent \underline{\it Weak solutions.}
A pair $(u,p)\in H^1(B_{R,+})^d\times L^2(B_{R,+})$ is said to be a weak solution of \eqref{intro.S} if it satisfies: (i) $\nabla\cdot u=0$ in the sense of distributions, (ii) $\psi u\in H^1_0(Q_R(0))^d$ for any cut-off function $\psi\in C_0^\infty(Q_R(0))$, and (iii) the weak formulation
\begin{align*}
\int_{B_{R,+}} \nabla u \cdot \nabla \varphi
- \int_{B_{R,+}} p (\nabla\cdot\varphi)
= 0 \quad \text{for any } \varphi\in C^{\infty}_{0}(B_{R,+})^d.
\end{align*}
We emphasize that $u$ has natural zero-extension $\tilde{u}$ to $Q_R$ so that $\tilde{u}\in H^1(Q_R)^d$. We do not distinguish between $u$ and $\tilde{u}$. This fact will be used without mentioning in Section \ref{sec.largescale}.

Moreover, a pair $(v,q)\in H^1_{{\rm loc}}(\overline{\Omega})^d\times L^2_{{\rm loc}}(\overline{\Omega})$ is said to be a weak solution of \eqref{eq1.intro.thm.bl} if it satisfies: (i) $\nabla\cdot v=0$ in the sense of distributions, (ii) $\chi(v+P)\in H^1_0(\Omega)^d$ for any $\chi\in C^\infty_0(\R^d)$, and (iii) the weak formulation:
\begin{align*}
	\int_\Omega \nabla v \cdot\nabla \phi
	- \int_\Omega q (\nabla\cdot \phi)
	=0, \quad \text{for any } \phi\in C^\infty_0(\Omega)^d.
\end{align*}
In addition, the weak solution of the system \eqref{intro.eq.bl.beta} is defined in the similar manner as above.

\noindent \underline{\it Derivatives.} The derivative with respect $x_i$ and $y$ will be denoted respectively by $\partial_i = \partial/\partial x_i$ and $\partial_y = \partial/\partial y$. Moreover, we use the following derivative operators: 
\begin{equation*}
\begin{aligned}
    &\nabla' = (\partial_1,\cdots, \partial_{d-1}), &\qquad
    &\nabla = (\nabla', \partial_y),\\
    &\Delta' = \nabla'\cdot \nabla', &\qquad 
    &\Delta = \nabla\cdot \nabla. 
\end{aligned}
\end{equation*}

\noindent \underline{\it Multi-index.} An element of $\Z_{\ge 0}^{d-1}$ is called a multi-index. For $\alpha = (\alpha_1,\cdots, \alpha_{d-1}) \in \Z_{\ge 0}^{d-1}$, we define the length of $\alpha$ by $|\alpha|=\alpha_1+\cdots +\alpha_{d-1}$. We also define $\alpha! = \alpha_1!\cdots \alpha_{d-1}!$.
For $\alpha, \beta\in \Z_{\ge 0}^{d-1}$, we say that $\beta \le \alpha$ if $\alpha_i\le\beta_i$ holds for all $i=1,\cdots d-1$. Moreover, using $\alpha\in \Z_{\ge 0}^{d-1}$, we define the monomial and the higher-order derivative respectively by
\begin{equation*}
    x^\alpha = x_1^{\alpha_1} \cdots x_{d-1}^{\alpha_{d-1}}, \qquad
    \partial_x^\alpha = \partial_1^{\alpha_1} \cdots \partial_{d-1}^{\alpha_{d-1}}. 
\end{equation*}

\noindent \underline{\it Binomial and multinomial coefficients.} For $0\le k\le n$, the binomial coefficients is 
\begin{equation*}
    \binom{n}{k} = \frac{n!}{k!(n-k)!}. 
\end{equation*}
For $\alpha, \beta\in \Z_{\ge 0}^{d-1}$ with $\beta \le \alpha$, the (generalized) binomial coefficient is defined by 
\begin{equation*}
    \binom{\alpha}{\beta} = \frac{\alpha!}{\beta!(\alpha - \beta)!} = \binom{\alpha_1}{\beta_1} \cdots \binom{\alpha_{d-1}}{\beta_{d-1}}.
\end{equation*}
Moreover, for $\alpha\in \Z_{\ge 0}^{d-1}$ with $|\alpha|=m$, the multinomial binomial coefficient is defined by 
\begin{equation*}
    \binom{m}{\alpha} = \frac{m!}{\alpha!}. 
\end{equation*}

\noindent \underline{\it Vector and tensor calculus.} The standard basis of $\R^d$ is denoted by $\{{\bf e}_1,\cdots,{\bf e}_d\}$. 
For $\mathcal{X}=(\mathcal{X}_1,\cdots,\mathcal{X}_d),\mathcal{Y}=(\mathcal{Y}_1,\cdots,\mathcal{Y}_d)\in\C^d$, we set $\mathcal{X}\cdot \mathcal{Y}=\mathcal{X}_1 \mathcal{Y}_1+\cdots+\mathcal{X}_d \mathcal{Y}_d$. This coincides with the standard inner product if $\mathcal{X},\mathcal{Y}\in\R^d$. Moreover, depending on the context, we will write $(x,y)=(x_1,\cdots,x_{d-1},y)\in\C^d$ as the column vector $\begin{bmatrix} x \\ y \\ \end{bmatrix}$, which is distinguished from the binomial coefficient above.
For $m\in\Z_{\ge 0}$ and $\mathcal{X}\in\R^d$, the symmetric tensor $\mathcal{X}^{\otimes m}=\{(\mathcal{X}^{\otimes m})_{\alpha}\}_{\alpha}$ is defined by 
\begin{equation*}
	(\mathcal{X}^{\otimes m})_{\alpha} = \mathcal{X}^\alpha, 
	\quad \alpha\in\Z^d, 
	\quad |\alpha|=m.
\end{equation*}
For symmetric tensors $S=\{S_\alpha\}_{\alpha}, T=\{T_\alpha\}_{\alpha}$ of order $m$, we define the contraction by 
\begin{equation*}
	S:T = \sum_{|\alpha| = m} \dbinom{m}{\alpha} S_\alpha T_\alpha, \quad \alpha\in\Z^d. 
\end{equation*}
These notations are useful in briefly describing Taylor expansions; see Subsection \ref{subsec.velocityest}.

\noindent \underline{\it Average integral.} For a bounded open set $E\subset \R^d$ and a measurable function $f$ on $E$, we set 
\begin{equation*}
    \dashint_E f = \frac{1}{|E|} \int_E f, \qquad 
    \| f\|_{\ul{L}^2(E)} = \bigg( \dashint_E |f|^2 \bigg)^{1/2}.
\end{equation*}
%

\subsection{Organization of the paper}
This paper is organized as follows. In Section \ref{sec.stokespoly}, we study the properties of the polynomial solutions to the Stokes system on $\R^d_+$. The definition of the space of Stokes polynomials $\mathcal{S}_m(\R^d_+)$ will be given at the end of the section. Section \ref{sec.bl} is devoted to the proof of Theorem \ref{intro.thm.bl} with quantitative estimates. 
In addition, the space of heterogeneous Stoke polynomials $\mathcal{S}_m(\Omega)$ is defined in Subsection \ref{subsec.het.stokespoly}.  In Section \ref{sec.largescale}, using the boundary layers, we prove Theorem \ref{thm.main}, as well as a Liouville theorem of arbitrary order. Finally, in Section \ref{sec.effective}, we define the space of effective Stokes polynomials $\mathcal{S}_m^{\rm eff}(\Omega)$. As applications, we prove Theorem \ref{thm.pointwise} in Subsection \ref{subsec.effective} and derive the local wall laws in Subsection \ref{subsec.wall.law}. 

\subsection*{Acknowledgements}
M.H. acknowledges the discussion with Prof. Saiei-Jaeyeong Matsubara-Heo in the early stage of this work. M.H. is partially supported by JSPS KAKENHI Grant Number JP 20K14345. 

\section{Stokes polynomials}\label{sec.stokespoly}
In this section, we consider the Stokes equations imposed on the half-space: 
\begin{equation}\tag{SH}\label{eq.stokes.halfspace}
\left\{
\begin{array}{ll}
-\Delta u + \nabla p=0 &\text{in } \R^{d}_+ \\
\nabla\cdot u=0 &\text{in } \R^{d}_+ \\
u=0 &\text{on } \partial\R^{d}_+.
\end{array}
\right.
\end{equation}
Our aim is to study the structure of the vector space $\mathcal{S}_m(\R^{d}_+)$ of all the polynomial solutions of \eqref{eq.stokes.halfspace} up to degree $m$. We are aware that the main theorems can be proved independently of this section. However, it could serve as a useful reference when the results of this paper are applied in more practical settings.

Let $\dot{\mathcal{P}}_m(\R^{d})$ be the space of homogeneous polynomials of degree $m\in\Z_{\ge0}$. Define the space of homogeneous Stokes polynomials by 
\begin{equation*}
\begin{aligned}
\dot{\mathcal{S}}_m(\R^{d}_+) 
=\{(u,p)\in\dot{\mathcal{P}}_m(\R^{d})^{d}\times \dot{\mathcal{P}}_{m-1}(\R^{d})~|~(u,p) \ \, \text{is a solution of} \ \, \eqref{eq.stokes.halfspace}\}. 
\end{aligned}
\end{equation*}
When $m=1$, it is easily seen that
\begin{equation*}
\begin{aligned}
\dot{\mathcal{S}}_1(\R^{d}_+) 
&=\Big\{( \sum_{i=1}^{d-1} A_i y{\bf e}_i,L)~|~A_i ,L\in\R\Big\} \\
&=\spn \{(y{\bf e}_1,0), \cdots, (y{\bf e}_{d-1},0), (0,1)\}. 
\end{aligned}
\end{equation*}
Thus,
\begin{equation}\label{eq.dim.S1}
    \dim \dot{\mathcal{S}}_1(\R^{d}_+) = d.
\end{equation}

To investigate the structure of $\dot{\mathcal{S}}_m(\R^{d}_+)$ when $m\ge2$, we focus on the pressure term. The solenoidal condition implies that $p=p(x,y)$ solves the following Neumann problem
\begin{equation*}
\left\{
\begin{array}{ll}
\Delta p=0 &\text{in } \R^{d}_+ \\
\partial_y p=\Delta u_{d} &\text{on } \partial\R^{d}_+. 
\end{array}
\right.
\end{equation*}
We recall that $u_{d}$ denotes the $d$-th component of $u=(u_1,\cdots,u_d)$. Introducing 
\begin{equation*}
\begin{aligned}
\dot{\mathcal{A}}_m(\R^{d}) 
&:=\{p\in \dot{\mathcal{P}}_m(\R^{d})~|~\Delta p =0\}, 
\end{aligned}
\end{equation*}
we see that the definition of $\dot{\mathcal{S}}_m(\R^{d}_+)$ can be rewritten as 
\begin{equation*}
\begin{aligned}
\dot{\mathcal{S}}_m(\R^{d}_+) 
&=\{(u,p)\in\dot{\mathcal{P}}_m(\R^{d})^{d}\times \dot{\mathcal{A}}_{m-1}(\R^{d})~|~(u,p) \ \, \text{is a solution of} \ \, \eqref{eq.stokes.halfspace}\}. 
\end{aligned}
\end{equation*}

Thus we begin by considering harmonic polynomials in half-spaces. The following lemma provides a relation between an element of $\dot{\mathcal{A}}_{m}(\R^{d})$ and its trace on $\partial\R^{d}_+$. We omit the proof as it is a simple computation.

\begin{lemma}\label{prop.harmonic}
Let $m\ge 1$ and $q\in\dot{\mathcal{A}}_{m}(\R^{d})$. Then there exists $q_1(x) \in \dot{\mathcal{P}}_m(\R^{d-1})$ and $q_2(x) \in \dot{\mathcal{P}}_{m-1}(\R^{d-1})$ such that
\begin{equation*}
    q(x,y) = \sum_{j=0}^{\infty} \frac{(-1)^j}{(2j)!} ((\Delta')^j q_1(x)) y^{2j} + \sum_{j=0}^{\infty} \frac{(-1)^j}{(2j+1)!} ((\Delta')^j q_2(x) ) y^{2j+1}.
\end{equation*}
In particular, if $q(x,0) = 0$, then $q_1 = 0$. 
\end{lemma}

Next let us define an operator $\Delta_{\rm D}^{-1}$ on the set $\{x^\alpha y^l~|~|\alpha|+l=m\}$ by 
\begin{equation*}
\begin{aligned}
\Delta_{\rm D}^{-1} (x^\alpha y^l) 
&= \sum_{j=0}^{\infty}  \frac{(-1)^j l!}{(l+2j+2)!} ((\Delta')^j x^\alpha) y^{l+2j+2} \\
&= \frac{1}{(l+2)(l+1)} x^\alpha y^{l+2} 
- \frac{1}{(l+4)(l+3)(l+2)(l+1)} (\Delta' x^\alpha) y^{l+4}
+ \cdots. 
\end{aligned}
\end{equation*}
Then $\Delta_{\rm D}^{-1}$ extends to an $\R$-linear map $\Delta_{\rm D}^{-1}:\dot{\mathcal{P}}_m(\R^{d}) \to \dot{\mathcal{P}}_{m+2}(\R^{d})$.

%
\begin{lemma}\label{prop.inv.dirichlet.lap.}
Let $m\ge1$. Then the following statements hold.

\begin{enumerate}[(i)]
\item Let $f\in\dot{\mathcal{P}}_m(\R^{d})$. Then $u:=\Delta_{\rm D}^{-1} f$ satisfies 
\begin{equation}\label{eq.Delta-1}
\left\{
\begin{array}{ll}
\Delta u=f &\text{in } \R^{d}_+ \\
u=0 &\text{on } \partial\R^{d}_+.
\end{array}
\right.
\end{equation}

\item Let $f\in\dot{\mathcal{P}}_m(\R^{d})$. For $i=1,2,\cdots,d-1$, we have
\begin{equation}\label{eq.commutative.del.i}
\begin{aligned}
\partial_i \Delta_{\rm D}^{-1} f
=\Delta_{\rm D}^{-1} \partial_i f, 
\end{aligned}
\end{equation}
and
\begin{equation}\label{eq.dyDelta-1}
\begin{aligned}
\partial_y \Delta_{\rm D}^{-1} f
=\Delta_{\rm D}^{-1} \partial_y f + \partial_y \Delta_D^{-1} (f(x,0)), 
\end{aligned}
\end{equation}
In particular, $\partial_y$ does not commute with $\Delta_{\rm D}^{-1}$ on the space spanned by $\{x^\alpha\}_{|\alpha|=m}$. 
\end{enumerate}
\end{lemma}

\begin{proof}
Note that \eqref{eq.Delta-1} and \eqref{eq.commutative.del.i} are obvious by the definition of $\Delta_D^{-1}$. To see \eqref{eq.dyDelta-1}, first observe that $\partial_y \Delta_D^{-1} (x^\alpha y^l) =  \Delta_D^{-1} \partial_y(x^\alpha y^l)$, if $l\ge 1$. By linearity, this implies
\begin{equation*}
    \partial_y \Delta_{\rm D}^{-1} ( f(x,y) - f(x,0)) =  \Delta_{\rm D}^{-1} \partial_y( f(x,y) - f(x,0)) = \Delta_{\rm D}^{-1} \partial_y f,
\end{equation*}
which gives the desired result.
\end{proof}
%

Now we give a solution formula for \eqref{eq.stokes.halfspace}, which regards the pressure as a source term.
%
\begin{proposition}\label{prop.solformula}
Let $m\ge2$. The following statements hold.

\begin{enumerate}[(i)]
\item Let $p\in\dot{\mathcal{A}}_{m-1}(\R^{d})$. Set
$$U=\Delta_{\rm D}^{-1}\nabla p \in \dot{\mathcal{P}}_m(\R^{d})^{d}.$$
Then the pair $(U,p)$ satisfies 
\begin{equation}\label{eq1.prop.solformula}
\left\{
\begin{array}{ll}
-\Delta U + \nabla p=0 &\text{in } \R^{d}_+ \\
\nabla\cdot U= \partial_y \Delta_{\rm D}^{-1}(\partial_y p(x,0)) &\text{in } \R^{d}_+ \\
U=0 &\text{on } \partial\R^{d}_+.
\end{array}
\right.
\end{equation}

\item Let $p\in\dot{\mathcal{A}}_{m-1}(\R^{d})$. Set
$$V= (V_1,0,\cdots,0)\in \dot{\mathcal{P}}_m(\R^{d})^{d}$$
with $V_1$ defined by 
\begin{equation*}
\begin{aligned}
V_1(x,y) & = \sum_{j=0}^\infty \frac{(-1)^j}{(2j+1)!} ((\Delta')^j c_{m-1}(x)) y^{2j+1} \\
& = c_{m-1}(x) y - \frac{1}{3\cdot2} (\Delta'c_{m-1}(x)) y^3 + \cdots,\\
c_{m-1}(x) &= -\int_{0}^{x_1} \partial_y p(z,x_2,\cdots,x_{d-1},0) \dd z.
\end{aligned}
\end{equation*}
Then $V$ satisfies 
\begin{equation}\label{eq2.prop.solformula}
\left\{
\begin{array}{ll}
-\Delta V=0 &\text{in } \R^{d}_+ \\
\nabla\cdot V=-\partial_y \Delta_{\rm D}^{-1} (\partial_y p(x,0)) &\text{in } \R^{d}_+ \\
V=0 &\text{on } \partial\R^{d}_+.
\end{array}
\right.
\end{equation}

\item Define the injective linear map $S:\dot{\mathcal{A}}_{m-1}(\R^{d})\to\dot{\mathcal{P}}_{m}(\R^{d}_+)^{d}$ by 
$$S[p]=U+V.$$
Then the pair $(S[p],p)$ is a solution of \eqref{eq.stokes.halfspace}, namely,  $(S[p],p)\in\dot{\mathcal{S}}_{m}(\R^{d}_+)$. 
\end{enumerate}
\end{proposition}
%
%
\begin{proof} (i) Lemma \ref{prop.inv.dirichlet.lap.} (i) implies that the first and third lines of \eqref{eq1.prop.solformula} are satisfied by $(U,p)$. It follows from \eqref{eq.commutative.del.i} and \eqref{eq.dyDelta-1} that
\begin{equation}
    \begin{aligned}
       \nabla \cdot U & = \nabla'\cdot \Delta_D^{-1} \nabla' p + \partial_y  \Delta_D^{-1} \partial_y p \\
       & = \Delta_D^{-1} \Delta' p + \Delta_D^{-1} \partial_y^2 p +  \partial_y \Delta_D^{-1} (\partial_y p(x,0)) \\
       & = \Delta_D^{-1} \Delta p + \partial_y \Delta_D^{-1} (\partial_y p(x,0)) \\
       & = \partial_y \Delta_D^{-1} (\partial_y p(x,0)),
    \end{aligned}
\end{equation}
where the last line follows from the fact $p\in\dot{\mathcal{A}}_{m-1}(\R^{d})$. Hence $(U,p)$ satisfies the second line of \eqref{eq1.prop.solformula} as well.

(ii) It is easy to check that $V$ satisfies the first and third lines of \eqref{eq2.prop.solformula}. A direct computation shows
\begin{equation*}
\begin{aligned}
\nabla\cdot V(x,y)
&= \partial_1 V_1(x,y) \\
& = \sum_{j=0}^\infty \frac{(-1)^j}{(2j+1)!} ((\Delta')^j \partial_1 c_{m-1}(x)) y^{2j+1}
\\
& = 
- \sum_{j=0}^\infty \frac{(-1)^j}{(2j+1)!} ((\Delta')^j \partial_y p(x,0) ) y^{2j+1}
\\
&= -\partial_y \Delta_{\rm D}^{-1} (\partial_y p(x,0)), 
\end{aligned}
\end{equation*}
where the definition of $\Delta_D^{-1}$ has been used in the last equation.
This gives the second line of \eqref{eq2.prop.solformula}.

(iii) It follows from (i) and (ii) that $(S[p],p)$ is a solution of \eqref{eq.stokes.halfspace}. To show the injectivity, let us assume that $S[p]=0$. Then one has $\nabla p=0$ from the equations, which implies $p=0$ because $p\in\dot{\mathcal{A}}_{m-1}(\R^{d})$ and $m\ge2$. This completes the proof. 
\end{proof}
%

Now we provide a structure property of $\dot{\mathcal{S}}_m(\R^{d}_+)$.

%
\begin{proposition}\label{thm.decom.stokespoly}
Let $m\ge2$ and let $S$ be the operator in Proposition \ref{prop.solformula} (iii). Then we have
\begin{equation*}
\begin{aligned}
\dot{\mathcal{S}}_m(\R^{d}_+)
=V_{m}^{(1)}\oplus V_{m}^{(2)}, 
\end{aligned}
\end{equation*}
where 
\begin{equation*}
\begin{aligned}
V_{m}^{(1)}
&=\{(u,p)\in\dot{\mathcal{S}}_m(\R^{d}_+)~|~p\in \dot{\mathcal{A}}_{m-1}(\R^{d}), u=S[p] \}, \\
V_{m}^{(2)}
&=\{(u,p)\in\dot{\mathcal{S}}_m(\R^{d}_+)~|~p=0\}. 
\end{aligned}
\end{equation*}
\end{proposition}
%

%
\begin{remark}\label{rem.thm.decom.stokespoly}
By this proposition and the injectivity of $S$, for given pair $(u,p)\in\dot{\mathcal{S}}_m(\R^{d}_+)$, 
one can find a unique pressure $p\in \dot{\mathcal{A}}_{m-1}(\R^{d})$ from the velocity $u\in \dot{\mathcal{P}}_m(\R^{d})^{d}$. 
\end{remark}
%

%
\begin{proof}
It suffices to show that $\dot{\mathcal{S}}_m(\R^{d}_+)\subset V_{N}^{(1)}\oplus V_{N}^{(2)}$. Let $(u,p)\in \dot{\mathcal{S}}_m(\R^{d}_+)$. If $p = 0$, then $(u,p)\in V_{N}^{(2)}$. If $p \in \dot{\mathcal{A}}_{m-1}(\R^{d}) \setminus\{0\}$, then $(S[p],p)\in V_{N}^{(1)}$ and $(v,q):=(u,p)-(S[p],p)\in V_{N}^{(2)}$. Hence we have $(u,p)=(S[p],p)+(v,q)\in V_{N}^{(1)}\oplus V_{N}^{(2)}$, which concludes the proof. 
\end{proof}
%

One can characterize the element of $V_{m}^{(2)}$ as follows.

%
\begin{lemma}\label{lem.zeropresure}
Let $w=(w_1,\cdots,w_{d-1},w_{d})\in V_{m}^{(2)}$. Then we have $w_{d} = 0$ and 
\begin{equation}\label{eq1.lem.zeropresure}
\begin{aligned}
w_i(x,y) & = \sum_{j=0}^{\infty} \frac{(-1)^j}{(2j+1)!} ((\Delta')^j a^{(i)}_{m-1}(x) ) y^{2j+1}, \qquad i=1,\cdots,d-1
\end{aligned}
\end{equation}
for some $a^{(1)}_{m-1},\cdots,a^{(d-1)}_{m-1}\in\dot{\mathcal{P}}_{m-1}(\R^{d-1})$ with constraint
\begin{equation}\label{eq2.lem.zeropresure}
\begin{aligned}
\partial_1 a^{(1)}_{m-1}(x) +\cdots+ \partial_{d-1} a^{(d-1)}_{m-1}(x)=0. 
\end{aligned}
\end{equation}
Conversely, any $w\in \dot{\mathcal{P}}_{m}(\R^{d})^d$ satisfying $w_d=0$, \eqref{eq1.lem.zeropresure} and \eqref{eq2.lem.zeropresure} belongs to $V_{m}^{(2)}$. 
\end{lemma}
%
%
\begin{proof}
Let $w\in V_{m}^{(2)}$. Then $w$ satisfies
\begin{equation*}
\left\{
\begin{array}{ll}
\Delta w=0 &\text{in } \R^{d}_+ \\
\nabla\cdot w=0 &\text{in } \R^{d}_+ \\
w=0 &\text{on } \partial\R^{d}_+.
\end{array}
\right.
\end{equation*}
From the first and third lines, and Lemma \ref{prop.harmonic}, there are $a^{(1)}_{m-1},\cdots,a^{(d)}_{m-1}\in\dot{\mathcal{P}}_{m-1}(\R^{d-1})$ such that \eqref{eq1.lem.zeropresure} holds.
On the other hand, $\nabla\cdot w = 0$ leads to
\begin{equation*}
\begin{aligned}
0 
= a^{(d)}_{m-1}(x) + (\partial_1 a^{(1)}_{m-1}(x) +\cdots+ \partial_{d-1} a^{(d-1)}_{m-1}(x)) y + \cdots.
\end{aligned}
\end{equation*}
This implies both that $w_d=0$ and that $\partial_1 a^{(1)}_{m-1}(x) +\cdots+ \partial_{d-1} a^{(d-1)}_{m-1}(x)=0$. It is easy to see that the converse holds true by the argument so far. The proof is complete.   
\end{proof}
%

Proposition \ref{thm.decom.stokespoly} and Lemma \ref{lem.zeropresure} imply the following corollary.

%
\begin{corollary}\label{cor.dim.stokespoly}
Let $d\ge2$ and $m\ge 1$. Then we have 
\begin{equation}\label{eq.dimDotSm}
    \dim \dot{\mathcal{S}}_m(\R^{d}_+) = d \dbinom{m+d-3}{d-2}.
\end{equation}
\end{corollary}
%
%
\begin{proof}
By Proposition \ref{thm.decom.stokespoly}, one has 
\begin{equation}\label{eq1.proof.cor.dim.stokespoly}
\begin{aligned}
\dim \dot{\mathcal{S}}_m(\R^{d}_+) 
=\dim V_{m}^{(1)} + \dim V_{m}^{(2)}. 
\end{aligned}
\end{equation}
The injectivity of $S:\dot{\mathcal{A}}_{m-1}(\R^{d})\to\dot{\mathcal{P}}_{m}(\R^{d}_+)^{d}$ implies that 
\begin{equation}\label{eq2.proof.cor.dim.stokespoly}
\begin{aligned}
\dim V_{m}^{(1)} 
&=\dim \dot{\mathcal{A}}_{m-1}(\R^{d}). 
\end{aligned}
\end{equation}
By Lemma \ref{lem.zeropresure}, there is an isomorphism
\begin{equation*}
\begin{aligned}
V_{m}^{(2)} 
&\cong
\{a=(a^{(1)}_{m-1}(x),\cdots,a^{(d-1)}_{m-1}(x))\in\mathcal{P}_{m-1}(\R^{d-1}_+)^{d-1}~|~\nabla'\cdot a=0\} \\
&= \ker\nabla'\cdot. 
\end{aligned}
\end{equation*}
Here $\nabla'\cdot$ is the linear map $\nabla'\cdot:\dot{\mathcal{P}}_{m-1}(\R^{d-1})^{d-1} \to \dot{\mathcal{P}}_{m-2}(\R^{d-1})$. Since it is obviously surjective, the homomorphism theorem yields that 
\begin{equation}\label{eq3.proof.cor.dim.stokespoly}
\begin{aligned}
\dim V_{m}^{(2)}
=\dim \ker\nabla'\cdot
=(d-1) \dim \dot{\mathcal{P}}_{m-1}(\R^{d-1}) - \dim \dot{\mathcal{P}}_{m-2}(\R^{d-1}). 
\end{aligned}
\end{equation}
From \eqref{eq1.proof.cor.dim.stokespoly}, \eqref{eq2.proof.cor.dim.stokespoly} and \eqref{eq3.proof.cor.dim.stokespoly}, and using the well-known facts
\begin{equation*}
\begin{aligned}
\dim \dot{\mathcal{P}}_{m}(\R^{d})
&= \dbinom{m+d-1}{d-1}, \\
\dim \dot{\mathcal{A}}_{m}(\R^{d})
&= \dbinom{m+d-1}{d-1} - \dbinom{m+d-3}{d-1}, \\
\dbinom{j}{k} + \dbinom{j}{k-1} &= \dbinom{j+1}{k} \quad \text{for any } 1\le k\le j,
\end{aligned}
\end{equation*}
we have
\begin{equation*}
\begin{aligned}
\dim \dot{\mathcal{S}}_m(\R^{d}_+)
&=\dbinom{m+d-2}{d-1} - \dbinom{m+d-4}{d-1} \\
&\quad
+ (d-1) \dbinom{m+d-3}{d-2} - \dbinom{m+d-4}{d-2} \\
& = d \dbinom{m+d-3}{d-2},
\end{aligned}
\end{equation*}
at least for $m\ge 2$. Finally note that the result for $m=1$ in \eqref{eq.dim.S1} is consistent with \eqref{eq.dimDotSm}. This completes the proof.
\end{proof}
%


Let us recall that $\mathcal{P}_m(\R^{d})$ is the set of polynomials up to degree $m$. We denote by $\mathcal{A}_m(\R^d)$ the set of harmonic polynomials up to degree $m$. Now we define
\begin{equation}\label{def.stokespoly}
\mathcal{S}_m(\R^{d}_+) = \dot{\mathcal{S}}_1(\R^{d}_+)\oplus \cdots \oplus \dot{\mathcal{S}}_m(\R^{d}_+) 
\subset \mathcal{P}_{m}(\R^{d})^d \times \mathcal{A}_{m-1}(\R^{d}). 
\end{equation}
Then $\mathcal{S}_m(\R^{d}_+)$ consists of all the polynomial solutions of \eqref{eq.stokes.halfspace} up to degree $m$. Corollary \ref{cor.dim.stokespoly} and a simple calculation show that
\begin{equation*}
    \dim \mathcal{S}_m(\R^{d}_+) = d \dbinom{m+d-2}{d-1}.
\end{equation*}
We call an element of $\mathcal{S}_m(\R^{d}_+)$  a Stokes polynomial in $\R^d_+$. Note that, by Remark \ref{rem.thm.decom.stokespoly}, for every $(u,p)\in\mathcal{S}_m(\R^{d}_+)$, we can recover the pressure part $p$ uniquely up to a constant from the velocity part $u$.

As seen in the introduction, the space $\mathcal{S}_m(\R^{d}_+)$ works well when studying the regularity for Stokes system in half-spaces. However, this is not the case in bumpy John domains since the polynomials in $\mathcal{S}_m(\R^{d}_+)$ have non-zero traces on the non-flat boundary $\Gamma = \partial \Omega$. In the next section,  We will construct the boundary layer correctors in order to eliminate the trace of these polynomials on $\Gamma$.

%
\section{Higher-order boundary layers}\label{sec.bl}
%
In this section, we construct a boundary layer corrector $(v^{\alpha}, q^{\alpha}) = (v^{\alpha,l}_{(i)}, q^{\alpha,l}_{(i)})$ solving
\begin{equation}\tag{BL$^{\alpha}$}\label{eq.bl.alpha}
\left\{
\begin{array}{ll}
-\Delta v^{\alpha} + \nabla q^{\alpha}=0 &\text{in } \Omega \\
\nabla\cdot v^{\alpha}=0 &\text{in } \Omega \\
v^{\alpha}(x,y)+x^\alpha y^l{\bf e}_i=0 &\text{on } \Gamma.
\end{array}
\right.
\end{equation}
Here all dependencies on $l$ and $i$ are omitted in $(v^{\alpha}, q^{\alpha})$ and \eqref{eq.bl.alpha} for simplicity. Recall that $\Omega$ is assumed to be a periodic bumpy John domain in Definition \ref{def.John2}.

The main result of this section is the following Proposition \ref{thm.bl.alpha}. As already notified in the introduction, Theorem \ref{intro.thm.bl} is a direct consequence of it thanks to linearity.

%
\begin{proposition}\label{thm.bl.alpha}
Let $\alpha\in\Z_{\ge0}^{d-1}$, $l\in\Z_{+}$ and $i\in\{1,\cdots,d\}$. There exists a weak solution $(v^{\alpha},q^{\alpha})\in H^1_{{\rm loc}}(\overline{\Omega})^{d}\times L^2_{{\rm loc}}(\overline{\Omega})$ of \eqref{eq.bl.alpha} satisfying
\begin{equation}\label{est1.thm.bl.alpha}
\begin{aligned}
\|v^{\alpha}\|_{\ul{L}^2(B_{r,+})}
&\le C r^{|\alpha|}, \quad r\ge1, \\
\|\nabla v^{\alpha}\|_{\ul{L}^2(B_{r,+})}
+ \|q^{\alpha}\|_{\ul{L}^2(B_{r,+})}
&\le C r^{|\alpha|-1/2}, \quad r\ge1.
\end{aligned}
\end{equation}
Moreover, $(v^{\alpha},q^{\alpha})$ can be decomposed into 
\begin{equation}\label{est2.thm.bl.alpha}
\begin{aligned}
v^{\alpha}(x,y)&= v^{\alpha}_{\rm poly}(x,y) + v^{\alpha}_{\rm het}(x,y), \\
q^{\alpha}(x,y)&= q^{\alpha}_{\rm poly}(x,y) + q^{\alpha}_{\rm het}(x,y). 
\end{aligned}
\end{equation}
The terms on the right-hand sides are described as follows: 
\begin{enumerate}
\item $v^{\alpha}_{\rm poly}(x,y)$ and $q^{\alpha}_{\rm poly}(x,y)$ are polynomials in $x$ and $y$, such that 
\begin{equation}\label{est3.thm.bl.alpha}
\begin{aligned}
|v^{\alpha}_{\rm poly}(x,y)|
&\le C(1+|x|+|y|)^{|\alpha|}, \\
|\nabla v^{\alpha}_{\rm poly}(x,y)|
+ |q^{\alpha}_{\rm poly}(x,y)|
&\le C(1+|x|+|y|)^{|\alpha|-1}. 
\end{aligned}
\end{equation}

\item $v^{\alpha}_{\rm het}(x,y)$ and $q^{\alpha}_{\rm het}(x,y)$ are heterogeneous polynomials, namely, linear combinations of a product of a monomial in $x$ and a function periodic in $x$ decaying exponentially as $y\to\infty$, such that 
\begin{equation}\label{est4.thm.bl.alpha}
\begin{aligned}
\|\nabla^j v^{\alpha}_{\rm het}\|_{\ul{L}^2(B_{r,+})}
+ \|q^{\alpha}_{\rm het}\|_{\ul{L}^2(B_{r,+})}
\le C r^{|\alpha|-1/2}, \quad r\ge1, 
\end{aligned}
\end{equation}
for $j=0,1$. Besides far away from the boundary, we have 
\begin{equation}\label{est5.thm.bl.alpha}
\begin{aligned}
|\nabla^j v^{\alpha}_{\rm het}(x,y)|
+ |q^{\alpha}_{\rm het}(x,y)|
\le C (1+|x|)^{|\alpha|} e^{-y/2}, \quad y\ge4, 
\end{aligned}
\end{equation}
for $j=0,1$. 
\end{enumerate}
All the constants $C$ above depend on $\Omega$, $\alpha$ and $l$. 
\end{proposition}
%

We will prove Proposition \ref{thm.bl.alpha} in Subsection \ref{subsec.proof.thm.bl}. The function spaces and the technical lemmas needed for the proof are collected in Subsection \ref{subsec.prelims}. In Subsection \ref{subsec.het.stokespoly}, we define the heterogeneous Stokes polynomials based on Proposition \ref{thm.bl.alpha}.

%
\subsection{Preliminaries}\label{subsec.prelims}
%
\subsubsection*{Function spaces}
Similarly to \cite{HPZ}, we will work in the function spaces on the fundamental domain
\begin{equation*}
\Omega_{{\rm p}}=\Omega\cap\big((-\pi,\pi]^{d-1}\times(-1,\infty)\big).
\end{equation*}
It is regarded as an open and connected submanifold of $\mathbb{T}^{d-1}\times\R$ where $\mathbb{T}=\R/(2\pi\Z)$ is the torus. Then the subset $\Omega_{{\rm p},<2}:=\Omega_{{\rm p}}\cap \{y<2\}$ is automatically diffeomorphic to a bounded John domain in $\R^d$, and therefore, we have the Bogovskii operator on $\Omega_{{\rm p},<2}$ by Theorem \ref{thm.bog.} below. This fact will be used in the proof of Propositions \ref{prop.bl.zero} and \ref{prop.bl.beta}.

Let us introduce the function spaces needed in this section. 
The spaces $L^2(\Omega_{\rm p})$ and $\widehat{H}^1_0(\Omega_{\rm p})$ are the completion of $C_0^\infty(\Omega_{\rm p})$ under the norms
\begin{equation*}
\|f\|_{L^2(\Omega_{\rm p})} 
:= \bigg( \int_{\Omega_{\rm p}} |f|^2 \bigg)^{1/2}, 
\qquad  \| f \|_{\widehat{H}^1(\Omega_{\rm p})} 
:= \bigg( \int_{\Omega_{\rm p}} |\nabla f|^2 \bigg)^{1/2}.
\end{equation*}
One can easily check that $\widehat{H}^1_0(\Omega_{\rm p})$ is a Hilbert space equipped with the inner product $\langle \nabla f,\nabla g\rangle_{\Omega_{{\rm p}}}$. Here and in what follows, $\langle \cdot, \cdot\rangle_{\Omega_{{\rm p}}}$ denotes 
$$\langle \varphi,\psi\rangle_{\Omega_{{\rm p}}}
:=\int_{\Omega_{{\rm p}}} \varphi\cdot \overline{\psi},$$
where $\overline{\psi}$ denotes the complex conjugate of $\psi$. 
Finally, let $\widehat{H}^1_{0,\sigma}(\Omega_{{\rm p}}):=\{ f\in \widehat{H}^1_{0}(\Omega_{{\rm p}})^d ~|~ \nabla\cdot f = 0 \}$, which is obviously a closed subspace of $\widehat{H}^1_{0}(\Omega_{{\rm p}})^d$ and thus again a Hilbert space.

We also use the cut-off function $\eta_-$ and truncation function $\eta_+$ respectively defined as
\begin{equation}\label{def.eta.m}
\begin{aligned}
&\text{$\eta_-(t)$ is non-negative,} \\
&\text{$\eta_-(t)=1$ if $t<\frac14$ and $\eta_-(t)=0$ if $t>\frac12$}, 
\end{aligned}
\end{equation}
and
\begin{equation}\label{def.eta.p}
\begin{aligned}
&\text{$\eta_+(t)$ is non-negative,} \\
&\text{$\eta_+(t)=0$ if $t<\frac14$ and $\eta_+(t)=1$ if $t>\frac12$}.
\end{aligned}
\end{equation}

\subsubsection*{Useful lemmas}
We collect the useful lemmas when proving Proposition \ref{thm.bl.alpha}. Firstly, we recall the Bogovskii operators in John domains. For a bounded open set $D\subset\R^{d}$ and $q\in(1,\infty)$, let
\begin{align*}
L^q_0(D)=\bigg\{f\in L^q(D)~\bigg|~\dashint_D f=0\bigg\}.
\end{align*}
%

%
\begin{theorem}[\cite{ADM06}]\label{thm.bog.}
Let $\Omega\subset\R^{d}$ be a bounded John domain according to Definition \ref{def.John}. Then there exists an operator $\B:L^q_0(\Omega)\to W^{1,q}_0(\Omega)^{d}$ satisfying 
\begin{align*}
\nabla\cdot \B[f]&=f \quad \mbox{in}\ \Omega, \\
\|\B[f]\|_{W^{1,q}(\Omega)} &\le C\|f\|_{L^q(\Omega)},
\end{align*}
with $C$ depending on $\Omega$ and $q$.
\end{theorem}
%

Next we consider the Stokes problem with source
\begin{equation}\tag{BL$_F$}\label{eq.bl.F}
\left\{
\begin{array}{ll}
-\Delta V + \nabla Q=F &\text{in } \Omega \\
\nabla\cdot V=0 &\text{in } \Omega \\
V=0 &\text{on } \partial\Omega.
\end{array}
\right.
\end{equation}

The following lemma gives the solution formula of $(V,Q)$ away from the boundary.

%
\begin{lemma}\label{prop.bl.F}
Let $L\ge1$. Suppose that there exists a unique weak solution $(V,Q)\in \widehat{H}^1_{0,\sigma}(\Omega_{{\rm p}})\times L^2(\Omega_{\rm p})$ of \eqref{eq.bl.F}. Assume that $F$ has the Fourier series expansion
\begin{equation}\label{assump1.prop.bl.F}
\begin{aligned}
F(x,y)
&=
\sum_{k\in\Z^{d-1}\setminus\{0\}} 
\F_k(y-L) e^{-|k|(y-L)} e^{\ii k\cdot x}, \quad y>L, 
\end{aligned}
\end{equation}
where $\F_k=\F_k(z)$ is a polynomial in $z$ with coefficients depending on $k$, such that
\begin{equation}\label{assump2.prop.bl.F}
\begin{aligned}
|\F_k(z)| + |k|^{-1} |\partial_z\F_k(z)| 
\le a_k (1+|k||z|)^{n},  
\end{aligned}
\end{equation}
for some $n\in\Z_{+}$, with the sequence $\{a_k\}_{k\in\Z^{d-1}\setminus\{0\}}$ satisfying  
\begin{equation}\label{assump3.prop.bl.F}
\begin{aligned}
\sum_{k\in\Z^{d-1}\setminus\{0\}}
|k|^{-3} |a_k|^2 
<\infty. 
\end{aligned}
\end{equation}
Then $(V,Q)$ can be decomposed into 
\begin{equation}\label{est1.prop.bl.F}
\begin{aligned}
V(x,y) &= V_{\rm poly} + V_{\rm per}(x,y), \\
Q(x,y) &= Q_{\rm per}(x,y). 
\end{aligned}
\end{equation}
The terms on the right-hand sides are described as follows: 
\begin{enumerate}[(i)]

\item $V_{\rm poly}$ is a constant vector field, such that
\begin{equation}\label{est2.prop.bl.F}
\begin{aligned}
|V_{\rm poly}| \le C \|\nabla V\|_{L^2(\Omega_{\rm p})}. 
\end{aligned}
\end{equation}

\item $V_{\rm per}(x,y)$ and $Q_{\rm per}(x,y)$ are functions of $L^2(\Omega_{\rm p})$ decaying exponentially as $y\to\infty$. These can be expanded in the Fourier series as
\begin{equation}\label{est3.prop.bl.F}
\begin{aligned}
V_{\rm per}(x,y)
&= \sum_{k\in\Z^{d-1}} 
V_k(y) e^{\ii k\cdot x}, \quad y>0, \\
Q_{\rm per}(x,y)
&= \sum_{k\in\Z^{d-1}} 
Q_k(y) e^{\ii k\cdot x}, \quad y>0, 
\end{aligned}
\end{equation}
with $V_{0}(y)$ and $Q_{0}(y)$ vanishing on $[L,\infty)$. Besides far away from the boundary, 
\begin{equation}\label{est4.prop.bl.F}
\begin{aligned}
V_{\rm per}(x,y)
&= \sum_{k\in\Z^{d-1}\setminus\{0\}} 
\V_k(y-L) e^{-|k|(y-L)} e^{\ii k\cdot x}, \quad y>L, \\
Q_{\rm per}(x,y)
&= \sum_{k\in\Z^{d-1}\setminus\{0\}} 
\Q_k(y-L) e^{-|k|(y-L)} e^{\ii k\cdot x}, \quad y>L, 
\end{aligned}
\end{equation}
where $\V_k=\V_k(z)$ and $\Q_k=\Q_k(z)$ are polynomials in $z$ with coefficients depending on $k$, such that 
\begin{equation}\label{est5.prop.bl.F}
\begin{aligned}
|\V_k(z)| 
&\le
b_k (1+|k||z|)^{n+2}, \\
|k|^{-1} (|\partial_z\V_k(z)| + |\Q_k(z)|)
&\le b_k (1+|k||z|)^{n+1}, \\
|k|^{-2} |\partial_z^2\V_k(z)|
+ |k|^{-3} |\partial_z^3\V_k(z)| 
&\le b_k (1+|k||z|)^{n+1}, \\
|k|^{-2} |\partial_z\Q_k(z)| 
+ |k|^{-3} |\partial_z^2\Q_k(z)| 
&\le b_k (1+|k||z|)^{n},
\end{aligned}
\end{equation}
with the sequence $\{b_k\}_{k\in\Z^{d-1}\setminus\{0\}}$ satisfying  
\begin{equation}\label{est6.prop.bl.F}
\begin{aligned}
\sum_{k\in\Z^{d-1}\setminus\{0\}}
|k| |b_k|^2 
\le C\bigg(\|\nabla V\|_{L^2(\Omega_{\rm p})}^2 
+ \sum_{k\in\Z^{d-1}\setminus\{0\}}
|k|^{-3} |a_k|^2 \bigg)
<\infty. 
\end{aligned}
\end{equation}

\end{enumerate}
All the constants $C$ above depend on $\Omega$, $L$ and $n$. 
\end{lemma}
%
%
\begin{proof}
Since $(V,Q)\in \widehat{H}^1(\Omega_{{\rm p}})^{d}\times L^2(\Omega_{\rm p})$, one has the Fourier series expansion in $y>0$: 
\begin{equation}\label{est1.proof.prop.bl.F}
\begin{aligned}
V(x,y)
= \sum_{k\in\Z^{d-1}} 
\widetilde{V}_k(y) e^{\ii k\cdot x}, \qquad 
Q(x,y)
= \sum_{k\in\Z^{d-1}} 
\widetilde{Q}_k(y) e^{\ii k\cdot x}. 
\end{aligned}
\end{equation}
However, to prove the claim, we need a more precise formula for $(V,Q)$ on $\{y>L\}$. Setting $b(x)=V(x,L)$, we see that the restriction $(V,Q)=(V|_{\{y>L\}},Q|_{\{y>L\}})$ solves 
\begin{equation}\label{eq1.proof.prop.bl.F}
\left\{
\begin{array}{ll}
-\Delta V + \nabla Q=F|_{\{y>L\}}, \quad y>L, \\
\nabla\cdot V=0, \quad y>L, \\
V(x,L)=b(x). 
\end{array}
\right.
\end{equation}
Note that $b(x)$ is periodic and thus expanded as 
\begin{equation*}
b(x) = \sum_{k\in\Z^{d-1}} \hat{V}_k e^{\ii k\cdot x}, \quad 
\hat{V}_k := \widetilde{V}_k(L) = \frac{1}{(2\pi)^{d-1}} \int_{(-\pi,\pi)^{d-1}} V(x,L) e^{\ii k\cdot x} \dd x. 
\end{equation*}
Moreover, from $V\in \widehat{H}^1(\Omega_{{\rm p}})^{d}$ and the Poincar\'{e} inequality, one has 
\begin{equation}\label{est2.proof.prop.bl.F}
|\hat{V}_0| + \bigg(\sum_{k\in\Z^{d-1}\setminus\{0\}} |k| |\hat{V}_k|^2\bigg)^{1/2} 
\le C\|\nabla V\|_{L^2(\Omega_{\rm p})}
\end{equation}
with a constant $C$ depending on $L$.

We will derive the explicit representation of $(V,Q)$ in \eqref{eq1.proof.prop.bl.F} using the assumption \eqref{assump1.prop.bl.F}. This is done by first solving $\Delta Q = \nabla\cdot F$ including an undetermined coefficient, and then by solving \eqref{eq1.proof.prop.bl.F} regarding $-\nabla Q + F$ as a given external force. The undetermined coefficients are found thanks to the solenoidal and boundary conditions in \eqref{eq1.proof.prop.bl.F}. Although the computation is not short, the argument itself is elementary, as all the equations are reduced to ODEs in the variable $y$ by the periodicity in $x$. For this reason, we avoid giving the details here, state only the resulting formula and provide the verification in Appendix. Define the (scalar) polynomial
\begin{equation*}
\begin{aligned}
\ul{\Q}_k(z) 
&= -\frac{1}{2|k|} \int_{0}^{z} 
\bigg\{
\begin{bmatrix} \ii k \\ -|k| \\ \end{bmatrix}\cdot \F_k(w)
+ (\partial_z\F_k)_d(w) 
\bigg\} \dd w \\
&\quad
-\frac{1}{2|k|} \int_{z}^{\infty} 
\bigg\{
\begin{bmatrix} \ii k \\ -|k| \\ \end{bmatrix}\cdot \F_k(w)
+ (\partial_z\F_k)_d(w) 
\bigg\}
e^{2|k|(z-w)} \dd w, 
\end{aligned}
\end{equation*}
the vector-valued polynomial
\begin{equation*}
\begin{aligned}
\ul{\V}_k(z) 
&= -\frac{1}{2|k|} \int_{0}^{z} 
\bigg\{-\F_k(w)
+ \begin{bmatrix} \ii k \\ -|k| \\ \end{bmatrix} \ul{\Q}_k(w)
+ \begin{bmatrix} 0 \\ \partial_z\ul{\Q}_k(w) \\ \end{bmatrix}\bigg\} \dd w \\
&\quad
-\frac{1}{2|k|} \int_{z}^{\infty} 
\bigg\{-\F_k(w)
+ \begin{bmatrix} \ii k \\ -|k| \\ \end{bmatrix} \ul{\Q}_k(w)
+ \begin{bmatrix} 0 \\ \partial_z\ul{\Q}_k(w) \\ \end{bmatrix}\bigg\} 
e^{2|k|(z-w)} \dd w, 
\end{aligned}
\end{equation*}
and the constant
\begin{equation}\label{def.ck}
\begin{aligned}
c_k=
\begin{bmatrix} \ii k \\ - |k| \\ \end{bmatrix}\cdot \hat{V}_k 
+ (\partial_z\underline{\V}_k)_d(0). 
\end{aligned}
\end{equation}
Then, $(V,Q)$ is represented as
\begin{equation}\label{est3.proof.prop.bl.F}
\begin{aligned}
V(x,y)
&= \hat{V}_0
+ \sum_{k\in\Z^{d-1}\setminus\{0\}} 
\V_k(y-L) e^{-|k|(y-L)} e^{\ii k\cdot x}, \quad y>L, \\
Q(x,y)
&= \sum_{k\in\Z^{d-1}\setminus\{0\}} 
\Q_k(y-L) e^{-|k|(y-L)} e^{\ii k\cdot x}, \quad y>L, 
\end{aligned}
\end{equation}
where the coefficient $(\V_k,\Q_k)=(\V_k(z),\Q_k(z))$ is explicitly given by 
\begin{equation}\label{est4.proof.prop.bl.F}
\begin{aligned}
\V_k(z) 
&=
\hat{V}_k 
+ \frac{c_k}{|k|} \begin{bmatrix} \ii k \\ -|k| \\ \end{bmatrix} z
+ \ul{\V}_k(z) - \ul{\V}_k(0), \\
\Q_k(z)
&=-2c_k + \ul{\Q}_k(z). 
\end{aligned}
\end{equation}

Let us verify \eqref{est1.prop.bl.F}--\eqref{est6.prop.bl.F}. The decomposition \eqref{est1.prop.bl.F} follows if we set 
\begin{equation*}
\begin{aligned}
V_{\rm poly} = \hat{V}_0, \qquad V_{\rm per}=V-V_{\rm poly}, \qquad Q_{\rm per}=Q. 
\end{aligned}
\end{equation*}
The bound \eqref{est2.prop.bl.F} is a consequence of \eqref{est2.proof.prop.bl.F}. Next we set, using $(\widetilde{V}_k(y),\widetilde{Q}_k(y))$ in \eqref{est1.proof.prop.bl.F}, 
\begin{equation*}
\begin{aligned}
V_0(y) = \widetilde{V}_0(y) - V_{\rm poly}, \qquad 
V_k(y) = \widetilde{V}_k(y) \quad \text{for $k\neq0$} 
\end{aligned}
\end{equation*}
and $Q(y)=\widetilde{Q}_k(y)$. Then, the formulas \eqref{est3.prop.bl.F} and \eqref{est4.prop.bl.F} are deduced thanks to \eqref{est3.proof.prop.bl.F} as well as to the absence of the zero-mode in $V-V_{\rm poly}$ for $y>L$.

It remains to prove \eqref{est5.prop.bl.F} and \eqref{est6.prop.bl.F}. Using the assumption \eqref{assump2.prop.bl.F}, we compute 
\begin{equation*}
\begin{aligned}
|k|^2 |\ul{\V}_k(z)| 
&\le C a_k (1+|k||z|)^{n+2}, \\
|k| (|\ul{\Q}_k(z)| + |\partial_z\ul{\V}_k(z)|)
&\le C a_k (1+|k||z|)^{n+1}, \\ 
|\partial_z^2\ul{\V}_k(z)|
+ |k|^{-1} |\partial_z^3\ul{\V}_k(z)| 
&\le C a_k (1+|k||z|)^{n+1}, \\ 
|\partial_z\ul{\Q}_k(z)|
+ |k|^{-1} |\partial_z^2\ul{\Q}_k(z)|
&\le C a_k (1+|k||z|)^{n}, \\
|k|^{-1} |c_k|
&\le C (|\hat{V}_k| + |k|^{-2} a_k).
\end{aligned}
\end{equation*}
Combined with \eqref{est4.proof.prop.bl.F}, this implies \eqref{est5.prop.bl.F} with the sequence $\{b_k\}_{k\in\Z^{d-1}\setminus\{0\}}$ defined by 
\begin{equation*}
\begin{aligned}
b_k = C(|\hat{V}_k| + |k|^{-2} a_k). 
\end{aligned}
\end{equation*}
The estimate \eqref{est6.prop.bl.F} for $b_k$ follows from \eqref{est2.proof.prop.bl.F} and \eqref{assump3.prop.bl.F}. This completes the proof. 
\end{proof}
%

Finally, we consider the Stokes equations with bounded boundary data: 
\begin{equation}\tag{BL$^{0}$}\label{eq.bl.zero}
\left\{
\begin{array}{ll}
-\Delta V + \nabla Q=0 &\text{in } \Omega \\
\nabla\cdot V=0 &\text{in } \Omega \\
V(x,y)+y^l{\bf e}_i=0 &\text{on } \Gamma.
\end{array}
\right.
\end{equation}

%
\begin{lemma}\label{prop.bl.zero}
Let $l\in\Z_{+}$ and $i\in\{1,\cdots,d\}$. There exists a unique weak solution $(V,Q)\in \widehat{H}^1(\Omega_{{\rm p}})^d\times L^2(\Omega_{\rm p})$ of \eqref{eq.bl.zero}. Moreover, $V$ and $Q$ can be decomposed into 
\begin{equation}\label{est1.prop.bl.zero}
\begin{aligned}
V(x,y) &= V_{\rm poly} + V_{\rm per}(x,y), \\
Q(x,y) &= Q_{\rm per}(x,y). 
\end{aligned}
\end{equation}
The terms on the right-hand sides are described as follows: 
\begin{enumerate}[(i)]

\item $V_{\rm poly}$ is a constant vector field, such that
\begin{equation}\label{est2.prop.bl.zero}
\begin{aligned}
|V_{\rm poly}| \le C. 
\end{aligned}
\end{equation}

\item $V_{\rm per}(x,y)$ and $Q_{\rm per}(x,y)$ are functions of $L^2(\Omega_{\rm p})$ decaying exponentially as $y\to\infty$, such that 
\begin{equation}\label{est3'.prop.bl.beta}
\begin{aligned}
\|V_{\rm per}\|_{H^1(\Omega_{\rm p})}
+ \|Q_{\rm per}\|_{L^2(\Omega_{\rm p})} \le C. 
\end{aligned}
\end{equation}
These can be expanded in the Fourier series as
\begin{equation}\label{est3.prop.bl.zero}
\begin{aligned}
V_{\rm per}(x,y)
&= \sum_{k\in\Z^{d-1}} 
V_k(y) e^{\ii k\cdot x}, \quad y>0, \\
Q_{\rm per}(x,y)
&= \sum_{k\in\Z^{d-1}} 
Q_k(y) e^{\ii k\cdot x}, \quad y>0, 
\end{aligned}
\end{equation}
with $V_{0}(y)$ and $Q_{0}(y)$ vanishing on $[3,\infty)$. Besides far away from the boundary, 
\begin{equation}\label{est4.prop.bl.zero}
\begin{aligned}
V_{\rm per}(x,y)
&= \sum_{k\in\Z^{d-1}\setminus\{0\}} 
\V_k(y-3) e^{-|k|(y-3)} e^{\ii k\cdot x}, \quad y>3, \\
Q_{\rm per}(x,y)
&= \sum_{k\in\Z^{d-1}\setminus\{0\}} 
\Q_k(y-3) e^{-|k|(y-3)} e^{\ii k\cdot x}, \quad y>3, 
\end{aligned}
\end{equation}
where $\V_k=\V_k(z)$ is a polynomial of degree $1$ in $z$ with coefficients depending on $k$, and $\Q_k$ is a constant depending on $k$, such that 
\begin{equation}\label{est5.prop.bl.zero}
\begin{aligned}
|\V_k(z)| 
&\le
b_k (1+|k||z|), \\
|k|^{-1} (|\partial_z\V_k(z)| + |\Q_k(z)|)
&\le b_k, 
\end{aligned}
\end{equation}
with the sequence $\{b_k\}_{k\in\Z^{d-1}\setminus\{0\}}$ satisfying  
\begin{equation}\label{est6.prop.bl.zero}
\begin{aligned}
\sum_{k\in\Z^{d-1}\setminus\{0\}}
|k| |b_k|^2
\le C. 
\end{aligned}
\end{equation}
All the constants $C$ above depend on $\Omega$ and $l$.  
\end{enumerate}

\end{lemma}
%
%
\begin{proof}
We use the cut-off $\eta_-$ in \eqref{def.eta.m} and the truncation $\eta_+$ in \eqref{def.eta.p}. Let us define the smooth vector field $\widetilde{U}$ and the constant $A$ depending on $i\in\{1,\cdots,d\}$ by 
\begin{equation}
\begin{aligned}
\widetilde{U}(y) =y^l \eta_-(y) {\bf e}_i, 
\qquad A=\frac{1}{(2\pi)^{d-1}}\int_{\Omega_{{\rm p},<2}} 
\nabla\cdot\widetilde{U}. 
\end{aligned}
\end{equation}

We focus only on the case $i=d$ here since the other cases are easier. Let us introduce the pair of functions $(U,P)\in \widehat{H}^1(\Omega_{{\rm p}})^d\times L^2(\Omega_{\rm p})$ defined by 
\begin{equation}
\begin{aligned}
U = \widetilde{U} - \B[\nabla\cdot \widetilde{U}-A\partial_y\eta_+(y)] - A \eta_+(y) {\bf e}_{d}, \qquad
P = 0. 
\end{aligned}
\end{equation}
Here $\B$ is the Bogovskii operator on $\Omega_{{\rm p},<2}$ and $\B[\nabla\cdot \widetilde{U}-A\partial_y\eta_+(y)]$ has been extended by zero from $\Omega_{p,<2}$ to $\Omega_p$. Note that the average of $\nabla\cdot \widetilde{U}-A\partial_y\eta_+(y)$ over $\Omega_{{\rm p},<2}$ is zero, and hence, that $\nabla\cdot U=0$. Then $(\widetilde{V},\widetilde{Q}):=(V-U,Q-P)$ is a solution to the Stokes problem \eqref{eq.bl.F} with the source $F:=-\Delta U$. By the Lax-Milgram lemma, there is a unique weak solution $(\widetilde{V},\widetilde{Q})$ to this problem, which implies the existence of solutions to \eqref{eq.bl.zero}. The uniqueness of the weak solution of \eqref{eq.bl.zero} is obvious by a simple energy estimate.

Following the proof of Lemma \ref{prop.bl.F} and using that $F$ is identically zero when $y>3$, one can obtain the Fourier series formula for $(\widetilde{V},\widetilde{Q})$.  Moreover, all the statements \eqref{est1.prop.bl.zero}--\eqref{est6.prop.bl.zero} can be provided in the same manner as in the proof of Lemma \ref{prop.bl.F}. We omit the details to avoid repetitive arguments. This concludes the proof. 
\end{proof}
%

%
\subsection{Proof of Proposition \ref{thm.bl.alpha}}\label{subsec.proof.thm.bl}
This subsection is devoted to the proof of Proposition \ref{thm.bl.alpha}. We divide it into two parts.

%
\begin{proofx}{Proposition \ref{thm.bl.alpha} (first half)}
We will find a solution of \eqref{eq.bl.alpha} under an ansatz 
\begin{equation}\label{ansatz}
\begin{aligned}
v^{\alpha}(x,y) &= 
\sum_{\beta\le\alpha} 
\dbinom{\alpha}{\beta} x^{\alpha-\beta} V^{\beta}(x,y), \qquad
q^{\alpha}(x,y) &= \sum_{\beta\le\alpha} 
\dbinom{\alpha}{\beta} x^{\alpha-\beta} Q^{\beta}(x,y),
\end{aligned}
\end{equation}
where $(V^{0},Q^{0})$ is the solution to \eqref{eq.bl.zero} in the previous subsection. If $(V^{\beta},Q^{\beta})$ with $\beta\neq0$ is constructed so that $V^{\beta}$ vanishes on the boundary $\partial\Omega$ for all $\beta\le\alpha$, the boundary condition in \eqref{eq.bl.alpha} for $v^{\alpha}$ is recovered from \eqref{ansatz}. Thus we will find such $(V^{\beta},Q^{\beta})$ making the pair $(v^{\alpha},q^{\alpha})$ in \eqref{ansatz} to be a weak solution of \eqref{eq.bl.alpha}.

Let us derive the equations which $(V^{\beta},Q^{\beta})$ must satisfy. In what follows, in order to avoid complexity, {\it all notations with multi-indices in the subscript are promised to be zero if there is a multi-index belonging to $\Z^d\setminus\Z_{\ge0}^d$.} For example, $V^{\gamma}=0$ if $\gamma\in\Z^d\setminus\Z_{\ge0}^d$. Then a direct computation yields that 
\begin{equation*}
\begin{aligned}
\Delta v^{\alpha}
&= 
\sum_{\beta\le\alpha}
x^{\alpha-\beta}
\Bigg(
\sum_{i=1}^{d-1}
\dbinom{\alpha}{\beta-2{\bf e}_i}
(\alpha_i-\beta_i+2)(\alpha_i-\beta_i+1) 
V^{\beta-2{\bf e}_i} \\
&\qquad\qquad\qquad\quad
+ \sum_{i=1}^{d-1}
2\dbinom{\alpha}{\beta-{\bf e}_i}
(\alpha_i-\beta_i+1)
\partial_i V^{\beta-{\bf e}_i} 
+ \dbinom{\alpha}{\beta} \Delta V^{\beta}
\Bigg) \\
&= 
\sum_{\beta\le\alpha}
\dbinom{\alpha}{\beta} x^{\alpha-\beta}
\Bigg(
\Delta V^{\beta}
+ \sum_{i=1}^{d-1}
\Big( 
2\dbinom{\beta_i}{2} 
V^{\beta-2{\bf e}_i}
+ 2\dbinom{\beta_i}{1} 
\partial_i V^{\beta-{\bf e}_i}
\Big)
\Bigg)
\end{aligned}
\end{equation*}
and
\begin{equation*}
\begin{aligned}
\nabla q^{\alpha} 
&= 
\sum_{\beta\le\alpha}
x^{\alpha-\beta}
\Bigg(
\sum_{i=1}^{d-1}
\dbinom{\alpha}{\beta-{\bf e}_i}
(\alpha_i-\beta_i+1) 
Q^{\beta-{\bf e}_i} {\bf e}_i
+ \dbinom{\alpha}{\beta} 
\nabla Q^{\beta}
\Bigg) \\
&= 
\sum_{\beta\le\alpha}
\dbinom{\alpha}{\beta} x^{\alpha-\beta}
\Bigg(
\nabla Q^{\beta}
+ \sum_{i=1}^{d-1}
\dbinom{\beta_i}{1} 
Q^{\beta-{\bf e}_i} {\bf e}_i
\Bigg), 
\end{aligned}
\end{equation*}
as well as
\begin{equation*}
\begin{aligned}
\nabla\cdot v^{\alpha}
&= 
\sum_{\beta\le\alpha}
x^{\alpha-\beta}
\Bigg(
\sum_{i=1}^{d-1}
\dbinom{\alpha}{\beta-{\bf e}_i}
(\alpha_i-\beta_i+1)
(V^{\beta-{\bf e}_i})_i
+ \dbinom{\alpha}{\beta} \nabla\cdot V^{\beta}
\Bigg) \\
&= 
\sum_{\beta\le\alpha}
\dbinom{\alpha}{\beta} x^{\alpha-\beta}
\Bigg(
\nabla\cdot V^{\beta}
+ \sum_{i=1}^{d-1}
\dbinom{\beta_i}{1} 
(V^{\beta-{\bf e}_i})_i
\Bigg). 
\end{aligned}
\end{equation*}
Since $(v^{\alpha},q^{\alpha})$ has to be a solution of \eqref{eq.bl.alpha}, collecting the terms by the powers of $x$, we find that $(V^{\beta},Q^{\beta})$ should satisfy the following Stokes system 
\begin{equation}\tag{BL$^{\beta}$}\label{eq.bl.beta}
\left\{
\begin{array}{ll}
-\Delta V^{\beta} + \nabla Q^{\beta}=F^{\beta} &\text{in } \Omega \\
\nabla\cdot V^{\beta}=G^{\beta} &\text{in } \Omega \\
V^{\beta}=0 &\text{on } \partial\Omega
\end{array}
\right.
\end{equation}
with the data $(F^{\beta},G^{\beta})$ recursively defined by 
\begin{equation}\label{def.FG.beta1}
\begin{aligned}
F^{\beta}
&=
\sum_{i=1}^{d-1}
\Big(
2\dbinom{\beta_i}{2} 
V^{\beta-2{\bf e}_i}
+ 2\dbinom{\beta_i}{1} 
\partial_i V^{\beta-{\bf e}_i}
- \dbinom{\beta_i}{1} 
Q^{\beta-{\bf e}_i} {\bf e}_i
\Big), \\
G^{\beta}
&=
-\sum_{i=1}^{d-1}
\dbinom{\beta_i}{1} 
(V^{\beta-{\bf e}_i})_i.
\end{aligned}
\end{equation}

One can find a solution to \eqref{eq.bl.beta} by an induction argument, as exhibited next.
\end{proofx}
%

Before continuing the proof of Proposition \ref{thm.bl.alpha}, we state the existence and estimates of solutions of \eqref{eq.bl.beta} above as an independent proposition.

%
\begin{proposition}\label{prop.bl.beta}
Let $\beta\in\Z_{\ge0}^{d-1}$. There exists a weak solution $(V^{\beta},Q^{\beta})\in H^1_{{\rm loc}}(\overline{\Omega})^{d}\times L^2_{{\rm loc}}(\overline{\Omega})$ of \eqref{eq.bl.beta} and of \eqref{eq.bl.zero} when $\beta=0$. Moreover, $(V^{\beta},Q^{\beta})$ can be decomposed into 
\begin{equation}\label{est3.prop.bl.beta}
\begin{aligned}
V^{\beta}(x,y) &= V^{\beta}_{\rm poly}(y) + V^{\beta}_{\rm per}(x,y), \\
Q^{\beta}(x,y) &= Q^{\beta}_{\rm poly}(y) + Q^{\beta}_{\rm per}(x,y). 
\end{aligned}
\end{equation}
The terms on the right-hand sides are described as follows: 
\begin{enumerate}[(i)]
\item $V^{\beta}_{\rm poly}(y)$ and $Q^{\beta}_{\rm poly}(y)$ are polynomials in $y$, such that
\begin{equation}\label{est4.prop.bl.beta}
\begin{aligned}
|V^{\beta}_{\rm poly}(y)| 
&\le C(1+|y|)^{|\beta|}, \\
|\partial_y V^{\beta}_{\rm poly}(y)| 
+ |Q^{\beta}_{\rm poly}(y)| 
&\le C(1+|y|)^{|\beta|-1}. 
\end{aligned}
\end{equation}

\item $V^{\beta}_{\rm per}(x,y)$ and $Q^{\beta}_{\rm per}(x,y)$ are functions of $L^2(\Omega_{\rm p})$ decaying exponentially as $y\to\infty$, such that 
\begin{equation}\label{est6.prop.bl.beta}
\begin{aligned}
\|V^{\beta}_{\rm per}\|_{H^1(\Omega_{\rm p})}
+ \|Q^{\beta}_{\rm per}\|_{L^2(\Omega_{\rm p})} \le C. 
\end{aligned}
\end{equation}
These can be expanded in the Fourier series as
\begin{equation}\label{est7.prop.bl.beta}
\begin{aligned}
V^{\beta}_{\rm per}(x,y)
&= \sum_{k\in\Z^{d-1}} 
V^{\beta}_k(y) e^{\ii k\cdot x}, \quad y>0, \\
Q^{\beta}_{\rm per}(x,y)
&= \sum_{k\in\Z^{d-1}} 
Q^{\beta}_k(y) e^{\ii k\cdot x}, \quad y>0, 
\end{aligned}
\end{equation}
with $V^{\beta}_{0}(y)$ and $Q^{\beta}_{0}(y)$ vanishing on $[3,\infty)$. Besides, far away from the the boundary, 
\begin{equation}\label{est8.prop.bl.beta}
\begin{aligned}
V^{\beta}_{\rm per}(x,y)
&= \sum_{k\in\Z^{d-1}\setminus\{0\}} 
\V^{\beta}_k(y-3) e^{-|k|(y-3)} e^{\ii k\cdot x}, \quad y>3, \\
Q^{\beta}_{\rm per}(x,y)
&= \sum_{k\in\Z^{d-1}\setminus\{0\}} 
\Q^{\beta}_k(y-3) e^{-|k|(y-3)} e^{\ii k\cdot x}, \quad y>3, 
\end{aligned}
\end{equation}
where $\V^{\beta}_k=\V^{\beta}_k(z)$ and $\Q^{\beta}_k=\Q^{\beta}_k(z)$ are polynomials in $z$ with coefficients depending on $k$, such that 
\begin{equation}
\begin{aligned}
|\V^{\beta}_k(z)|\label{est9.prop.bl.beta} 
&\le
b^{\beta}_k (1+|k||z|)^{2|\beta|+1}, \\
|k|^{-1} (|\partial_z\V^{\beta}_k(z)| + |\Q^{\beta}_k(z)|)
&\le b^{\beta}_k (1+|k||z|)^{2|\beta|}, \\
|k|^{-2} |\partial_z^2\V^{\beta}_k(z)|
+ |k|^{-3} |\partial_z^3\V^{\beta}_k(z)| 
&\le b^{\beta}_k (1+|k||z|)^{2|\beta|}, \\
|k|^{-2} |\partial_z\Q^{\beta}_k(z)| 
+ |k|^{-3} |\partial_z^2\Q^{\beta}_k(z)| 
&\le b^{\beta}_k (1+|k||z|)^{2|\beta|-1},
\end{aligned}
\end{equation}
with the sequence $\{b^{\beta}_k\}_{k\in\Z^{d-1}\setminus\{0\}}$ satisfying 
\begin{equation}\label{est10.prop.bl.beta}
\begin{aligned}
\sum_{k\in\Z^{d-1}\setminus\{0\}}
|k| |b^{\beta}_k|^2 
\le C. 
\end{aligned}
\end{equation}
\end{enumerate}
All the constants $C$ above depend on $\Omega$ and $\beta$ as well as $l$ in \eqref{eq.bl.alpha}.
\end{proposition}
%
%
\begin{proof}
We will show that a weak solution $(V^{\beta},Q^{\beta})$ of \eqref{eq.bl.beta} can be found under the ansatz \eqref{est3.prop.bl.beta}. We rely on an induction on the length $|\beta|$. When $|\beta|=0$, the statement is trivial as 
we chosen $(V^{0},Q^{0})$ to be the solution of \eqref{eq.bl.zero} in Lemma \ref{prop.bl.zero}. Thus we fix $\beta\in\Z_{\ge0}^{d-1}$ with $\beta\neq0$. We assume that a function $(V^{\gamma},Q^{\gamma})$ is constructed for all $\gamma$ with $|\gamma|\le |\beta|-1$ which satisfies \eqref{eq.bl.beta} and all the properties \eqref{est3.prop.bl.beta}--\eqref{est10.prop.bl.beta} with $\beta$ replaced by $\gamma$.

Throughout the proof, all the constants depending on $\beta$ are denoted by $C_\beta$.

Firstly we prove the existence of solutions. Using the induction assumption, we decompose $(F^{\beta},G^{\beta})$ in \eqref{def.FG.beta1} as follows: 
\begin{equation}\label{def.FG.beta2}
\begin{aligned}
F^{\beta}(x,y)
&=F^{\beta}_{\rm poly}(y) + F^{\beta}_{\rm per}(x,y), \\
G^{\beta}(x,y) 
&=G^{\beta}_{\rm poly}(y) + G^{\beta}_{\rm per}(x,y). 
\end{aligned}
\end{equation}
Here $F^{\beta}_{\rm poly}(y)$ and $G^{\beta}_{\rm poly}(y)$ are polynomials in $y$ given by
\begin{equation}\label{def.FG.beta3}
\begin{aligned}
F^{\beta}_{\rm poly}(y)
&=
\sum_{i=1}^{d-1}
\Big(
2\dbinom{\beta_i}{2} 
V^{\beta-2{\bf e}_i}_{\rm poly}(y)
- \dbinom{\beta_i}{1} 
Q^{\beta-{\bf e}_i}_{\rm poly}(y) {\bf e}_i
\Big), \\
G^{\beta}_{\rm poly}(y)
&=-\sum_{i=1}^{d-1}
\dbinom{\beta_i}{1} 
(V^{\beta-{\bf e}_i}_{\rm poly})_i(y), 
\end{aligned}
\end{equation}
while $F^{\beta}_{\rm per}$ and $G^{\beta}_{\rm per}$ are given by 
\begin{equation}\label{def.FG.beta4}
\begin{aligned}
F^{\beta}_{\rm per}
&=
\sum_{i=1}^{d-1}
\Big(
2\dbinom{\beta_i}{2} 
V^{\beta-2{\bf e}_i}_{\rm per}
+ 2\dbinom{\beta_i}{1} 
\partial_i V^{\beta-{\bf e}_i}_{\rm per}
- \dbinom{\beta_i}{1} 
Q^{\beta-{\bf e}_i}_{\rm per} {\bf e}_i
\Big), \\
G^{\beta}_{\rm per}
&=
-\sum_{i=1}^{d-1}
\dbinom{\beta_i}{1} 
(V^{\beta-{\bf e}_i}_{\rm per})_i.
\end{aligned}
\end{equation}

For later use, we summarize the properties of $(F^{\beta},G^{\beta})$. For $y>0$, 
\begin{equation}\label{def.FG.beta5}
\begin{aligned}
G^{\beta}_{\rm per}(x,y)
&= -\sum_{k\in\Z^{d-1}} 
\bigg(
\sum_{i=1}^{d-1} \dbinom{\beta_i}{1}(V^{\beta-{\bf e}_i}_k)_i(y) \bigg)
e^{\ii k\cdot x}, 
\end{aligned}
\end{equation}
and for $y>3$, 
\begin{equation}\label{def.FG.beta6}
\begin{aligned}
F^{\beta}_{\rm per}(x,y)
&= 
\sum_{k\in\Z^{d-1}\setminus\{0\}}
\F^{\beta}_k(y-3)
e^{-|k|(y-3)} 
e^{\ii k\cdot x}. 
\end{aligned}
\end{equation}
Here $\F^{\beta}_k=\F^{\beta}_k(z)$ is a polynomial in $z$ defined as 
\begin{equation*}
\begin{aligned}
\F^{\beta}_k(z)
&= 
\sum_{i=1}^{d-1}
\Big(2\dbinom{\beta_i}{2} 
\V^{\beta-2{\bf e}_i}_k(z) 
+ 2\dbinom{\beta_i}{1} 
\ii k_i \V^{\beta-{\bf e}_i}_k(z)
- \dbinom{\beta_i}{1} 
\Q^{\beta-{\bf e}_i}_k(z) {\bf e}_i\Big), 
\end{aligned}
\end{equation*}
which satisfies, by the induction assumption, 
\begin{equation}\label{def.FG.beta8}
\begin{aligned}
|\F^{\beta}_k(z)| + |k|^{-1} |\partial_z\F^{\beta}_k(z)|
&\le 
C_\beta |k| (b^{\beta-2{\bf e}_i}_k+b^{\beta-{\bf e}_i}_k) (1+|k||z|)^{2|\beta|-1}. 
\end{aligned}
\end{equation}

Clearly, $(F^{\beta},G^{\beta})$ cannot be treated in the $L^2$ spaces due to the polynomial growth. Our plan is to find the polynomial correctors to reduce it to a finite energy function. After obtaining such correctors, we will apply the Lax-Milgram lemma to the reduced problem.

The following proof of the existence is divided into four steps: Steps 1 finds a corrector $W^{\beta}$ for the inhomogeneous divergence $G^{\beta}$ and Step 2 finds a corrector for the new source $F^{\beta}+\Delta W^{\beta}$. Step 3 applies the Lax-Milgram lemma, and finally, the desired solution $(V^{\beta},Q^{\beta})$ will be given in Step 4. In these steps, we use a truncation function $\eta_+$ in \eqref{def.eta.p}.

\noindent {\bf Step 1: Corrector for the divergence.} In this step, we will find a solution $W^{\beta}$ to the following divergence problem for $G^{\beta}$:
\begin{equation}\label{eq1.step1.proof.prop.bl.beta}
\left\{
\begin{array}{ll}
\nabla\cdot W^{\beta}=G^{\beta} &\text{in}\ \Omega \\
W^{\beta}=0 &\text{on}\ \partial\Omega.
\end{array}
\right.
\end{equation}
As the support of $G^{\beta}$ is unbounded in $\Omega_{{\rm p}}$, one cannot use the Bogovskii operator to \eqref{eq1.step1.proof.prop.bl.beta} directly. Thus a correction term is introduced to localize it. Decompose $\Z^{d-1}\setminus\{0\}$ as
\begin{equation*}
    \Z^{d-1}\setminus\{0\} =\bigcup_{i=1}^{d-1} S_i,
\end{equation*}
where
\begin{equation*}
    S_i:=\{(0,\cdots,0,k_i,\cdots,k_{d-1})~|~k_i\neq0\}.
\end{equation*}
Then we define the corrector by 
\begin{equation*}
\begin{aligned}
\widetilde{W}^{\beta}(x,y) 
&= 
-\sum_{i=1}^{d-1}
\dbinom{\beta_i}{1}
\bigg(\int_0^y
\Big((V^{\beta-{\bf e}_i}_{\rm poly})_i(t) + (V^{\beta-{\bf e}_i}_{0})_i(t) \Big)
\dd t\bigg)
{\bf e}_{d} \eta_+(y) \\
&\quad
- \sum_{j=1}^{d-1}
\sum_{k\in S_j}
\bigg(
\sum_{i=1}^{d-1}
\dbinom{\beta_i}{1}
\frac{(V^{\beta-{\bf e}_i}_k)_i(y)}{\ii k_j} {\bf e}_j \bigg) 
e^{\ii k\cdot x} \eta_+(y) 
\end{aligned}
\end{equation*}
and the remainder by 
\begin{equation*}
\begin{aligned}
\widetilde{R}^{\beta}(x,y)
&= 
-G^{\beta}(x,y)(\eta_+(y)-1) \\
&\quad 
+ \sum_{i=1}^{d-1}
\dbinom{\beta_i}{1}
\bigg(\int_0^y 
\Big((V^{\beta-{\bf e}_i}_{\rm poly})_i(t) + (V^{\beta-{\bf e}_i}_{0})_i(t) \Big)
\dd t\bigg)
\partial_y \eta_+(y). 
\end{aligned}
\end{equation*}
Note that $\widetilde{W}^{\beta}=0$ on $\partial\Omega$ thanks to the truncation $\eta_+$. Moreover, $\widetilde{R}^{\beta}\in L^2(\Omega_{{\rm p}})$ is supported in $\Omega_{{\rm p},<2}$. Invoking the definition of $G^{\beta}$ in \eqref{def.FG.beta1} as well as \eqref{def.FG.beta5}, we compute
\begin{equation}\label{eq2.step1.proof.prop.bl.beta}
\begin{aligned}
&\nabla\cdot\widetilde{W}^{\beta}(x,y) \\
&= 
-\sum_{i=1}^{d-1}
\dbinom{\beta_i}{1}
\Big(
(V^{\beta-{\bf e}_i}_{\rm poly})_i(y) + (V^{\beta-{\bf e}_i}_{0})_i(y)
\Big) \eta_+(y) \\
&\quad
-\sum_{i=1}^{d-1}
\dbinom{\beta_i}{1}
\bigg(\int_0^y 
\Big((V^{\beta-{\bf e}_i}_{\rm poly})_i(t) + (V^{\beta-{\bf e}_i}_{0})_i(t) \Big)
\dd t\bigg)
\partial_y\eta_+(y) \\
&\quad
- \sum_{j=1}^{d-1}
\sum_{k\in S_j}
\bigg(
\sum_{i=1}^{d-1}
\dbinom{\beta_i}{1}
(V^{\beta-{\bf e}_i}_k)_i(y)\bigg) 
e^{\ii k\cdot x} \eta_+(y) \\
&= 
G^{\beta}(x,y)\eta_+(y) \\
&\quad
-\sum_{i=1}^{d-1}
\dbinom{\beta_i}{1}
\bigg(\int_0^y
\Big((V^{\beta-{\bf e}_i}_{\rm poly})_i(t) + (V^{\beta-{\bf e}_i}_{0})_i(t) \Big)
\dd t\bigg)
\partial_y \eta_+(y) \\
&= 
G^{\beta}(x,y) - \widetilde{R}^{\beta}(x,y).
\end{aligned}
\end{equation}
Hence $\widetilde{W}^{\beta}$ localizes $G^{\beta}$ in the sense that
the support of $\nabla\cdot\widetilde{W}^{\beta} - G^{\beta}$ is compact in $\Omega_{{\rm p}}$.

The equation for the difference $\widetilde{w}^{\beta}:=W^{\beta}-\widetilde{W}^{\beta}$ is written as 
\begin{equation*}
\left\{
\begin{array}{ll}
\nabla\cdot \widetilde{w}^{\beta}=\widetilde{R}^{\beta} &\text{in } \Omega \\
\widetilde{w}^{\beta}=0 &\text{on } \partial\Omega. 
\end{array}
\right.
\end{equation*}
Although $\widetilde{R}^{\beta}$ is supported in $\Omega_{{\rm p},<2}$, the Bogovskii operator is still not applicable as the average of $\widetilde{R}^{\beta}$ is not necessarily zero. However, this can be easily modified. Let us define the constant $A^{\beta}$ by
\begin{equation*}
A^{\beta}=\frac{1}{(2\pi)^{d-1}}\int_{\Omega_{{\rm p},<2}} \widetilde{R}^{\beta}
\end{equation*}
and the function $R^{\beta}$ by 
\begin{equation}\label{def0.step1.proof.prop.bl.beta}
\begin{aligned}
R^{\beta}(x,y)
&= 
\widetilde{R}^{\beta}(x,y) - A^{\beta}\partial_y\eta_+(y).
\end{aligned}
\end{equation}
Then we see that the average of $R^{\beta}$ over $\Omega_{{\rm p},<2}$ is zero. Thus one of the solutions to
\begin{equation*}
\left\{
\begin{array}{ll}
\nabla\cdot w^{\beta}=R^{\beta} &\text{in } \Omega \\
w^{\beta}=0 &\text{on } \partial\Omega 
\end{array}
\right.
\end{equation*}
is provided by $w^{\beta}=\B[R^{\beta}]\in H^1_0(\Omega_{{\rm p}})^{d}$ from Theorem \ref{thm.bog.}. We point out that $w^\beta$ has been zero-extended from $\Omega_{p,<2}$ to $\Omega_p$.

Now it is easy to verify that the vector field
\begin{equation}\label{def0'.step1.proof.prop.bl.beta}
\begin{aligned}
W^{\beta}(x,y)
:= 
\widetilde{W}^{\beta}(x,y)
+ \B[R^{\beta}](x,y) + A^{\beta} \eta_+(y) {\bf e}_{d}
\end{aligned}
\end{equation}
is a solution to the divergence problem \eqref{eq1.step1.proof.prop.bl.beta}. In fact, by the definition of $R^{\beta}$ in \eqref{def0.step1.proof.prop.bl.beta}, 
\begin{equation*}
\begin{aligned}
\nabla\cdot W^{\beta}(x,y)
&= 
\nabla\cdot \widetilde{W}^{\beta}(x,y)
+ R^{\beta}(x,y)
+ A^{\beta}\partial_y\eta_+(y) \\
&= 
\nabla\cdot \widetilde{W}^{\beta}(x,y)
+ \widetilde{R}^{\beta}(x,y) \\
&= 
G^{\beta}(x,y), 
\end{aligned}
\end{equation*}
where \eqref{eq2.step1.proof.prop.bl.beta} is used in the last line. Since $W^{\beta}=0$ on $\partial\Omega$, it is deduced that $W^{\beta}$ is a desired solution of the divergence problem \eqref{eq1.step1.proof.prop.bl.beta}.

Let us summarize the properties of $W^{\beta}$ to be used later. We decompose $W^{\beta}$ in \eqref{def0'.step1.proof.prop.bl.beta} as
\begin{equation*}
\begin{aligned}
W^{\beta}(x,y)
&=
W^{\beta}_{\rm poly}(y) + W^{\beta}_{\rm per}(x,y), 
\end{aligned}
\end{equation*}
where the polynomial part $W^{\beta}_{\rm poly}$ is defined by 
\begin{equation*}
\begin{aligned}
W^{\beta}_{\rm poly}(y) 
=
\bigg(A^{\beta} 
- \sum_{i=1}^{d-1}
\dbinom{\beta_i}{1}
\Big(
\int_0^y
(V^{\beta-{\bf e}_i}_{\rm poly})_i(t) 
\dd t
+ \int_0^3
(V^{\beta-{\bf e}_i}_{0})_i(t) \dd t
\Big)
\bigg)
{\bf e}_{d}. 
\end{aligned}
\end{equation*}
Here we have used the induction assumption that $V^{\beta-{\bf e}_i}_{0}$ vanishes on $[3,\infty)$. Then the periodic part $W^{\beta}_{\rm per}:=W^{\beta}-W^{\beta}_{\rm poly}$ has the representation
\begin{equation*}
\begin{aligned}
W^{\beta}_{\rm per}(x,y)
&= 
\sum_{k\in\Z^{d-1}\setminus\{0\}} 
\mcW^{\beta}_k(y-3) e^{-|k|(y-3)} e^{\ii k\cdot x}, \quad y>3, 
\end{aligned}
\end{equation*}
with a polynomial $\mcW^{\beta}_k=\mcW^{\beta}_k(z)$ given by, for $k\in\Z^{d-1}\setminus\{0\}$ such that $k \in S_j$ with some $j \in \{1,2,\cdots, d-1\}$, 
\begin{equation}\label{def1.est2.step1.proof.prop.bl.beta}
\begin{aligned}
\mcW^{\beta}_k(z) = -\sum_{i=1}^{d-1}
\dbinom{\beta_i}{1}
\frac{(\V^{\beta-{\bf e}_i}_k)_i(z)}{\ii k_j} {\bf e}_j. 
\end{aligned}
\end{equation}
By the induction assumption, $W^{\beta}_{\rm poly}$ and $W^{\beta}_{\rm per}$ satisfy
\begin{equation}\label{est1.step1.proof.prop.bl.beta}
\begin{aligned}
|W^{\beta}_{\rm poly}(y)| 
&\le C_\beta (1+|y|)^{|\beta|}, \\
|\partial_y W^{\beta}_{\rm poly}(y)|
&\le C_\beta (1+|y|)^{|\beta|-1}, \\
\|W^{\beta}_{\rm per}\|_{H^1(\Omega_{\rm p})}
&\le C_\beta. 
\end{aligned}
\end{equation}

Moreover, we decompose $\Delta W^{\beta}$ as 
\begin{equation*}
\begin{aligned}
\Delta W^{\beta}
&=
\Delta W^{\beta}_{\rm poly}
+ \Delta W^{\beta}_{\rm per}
=\partial_y^2 W^{\beta}_{\rm poly}
+ \Delta W^{\beta}_{\rm per}. 
\end{aligned}
\end{equation*}
Then we compute 
\begin{equation}\label{est2.step1.proof.prop.bl.beta}
\begin{aligned}
\partial_y^2 W^{\beta}_{\rm poly}(y) 
&=
-\bigg(\sum_{i=1}^{d-1}
\dbinom{\beta_i}{1}
(\partial_y V^{\beta-{\bf e}_i}_{\rm poly})_i(y) 
\bigg)
{\bf e}_{d}, \\
\Delta W^{\beta}_{\rm per}(x,y)
&= 
\sum_{k\in\Z^{d-1}\setminus\{0\}}
\mcG_k(y-3) e^{-|k|(y-3)} e^{\ii k\cdot x}, \quad y>3. 
\end{aligned}
\end{equation}
Here the polynomial $\mcG_k=\mcG_k(z)$ is given by, for $k\in\Z^{d-1}\setminus\{0\}$ such that $k \in S_j$, 
\begin{equation*}
\begin{aligned}
\mcG_k(z)
=-\sum_{i=1}^{d-1}
\dbinom{\beta_i}{1}
\frac{1}{\ii k_j}
\big((\partial_z^2 \V^{\beta-{\bf e}_i}_k)_i(z) 
-2|k|(\partial_z \V^{\beta-{\bf e}_i}_k)_i(z)\big)
{\bf e}_j. 
\end{aligned}
\end{equation*}
The cancellation $\Delta (e^{-|k|(y-3)} e^{\ii k\cdot x})=0$ has been used in the computation. Eventually, by the induction assumption again, we see that 
\begin{equation}\label{est4.step1.proof.prop.bl.beta}
\begin{aligned}
|\mcG_k(z)| + |k|^{-1} |\partial_z\mcG_k(z)|
\le C_\beta |k|^2 b^{\beta-{\bf e}_i}_k (1+|k||z|)^{2|\beta|-2}. 
\end{aligned}
\end{equation}

\noindent {\bf Step 2: Corrector for the source.} Our aim in this step is to find a corrector $(\Lambda^{\beta},\Pi^{\beta})$ for the polynomial part of $F^{\beta}+\Delta W^{\beta}$, namely, a pair of functions satisfying
\begin{equation}\label{eq1.step2.proof.prop.bl.beta}
\left\{
\begin{array}{ll}
-\Delta \Lambda^{\beta} + \nabla \Pi^{\beta} 
=  F^{\beta}_{\rm poly} + \partial_y^2 W^{\beta}_{\rm poly}
+ H^{\beta}_{\rm per} &\text{in } \Omega \\
\nabla\cdot\Lambda^{\beta} = 0 &\text{in } \Omega \\
\Lambda^{\beta}=0 &\text{on } \partial\Omega. 
\end{array}
\right.
\end{equation}
Here $H^{\beta}_{\rm per}\in L^2(\Omega_{\rm p})^{d}$ is the remainder in the correction assumed to be compactly supported in $\Omega_{\rm p}$. From \eqref{def.FG.beta3} and \eqref{est2.step1.proof.prop.bl.beta}, the function $F^{\beta}_{\rm poly} + \partial_y^2 W^{\beta}_{\rm poly}$ is explicitly given by 
\begin{equation*}
\begin{aligned}
&F^{\beta}_{\rm poly}(y) + \partial_y^2 W^{\beta}_{\rm poly}(y) \\
&= 
\sum_{j=1}^{d-1}
\sum_{i=1}^{d-1}
2\dbinom{\beta_i}{2} 
(V^{\beta-2{\bf e}_i}_{\rm poly})_j(y) {\bf e}_j
- \sum_{i=1}^{d-1}
\dbinom{\beta_i}{1} 
Q^{\beta-{\bf e}_i}_{\rm poly}(y) {\bf e}_i \\
&\quad
+ \sum_{i=1}^{d-1}
\bigg(
2\dbinom{\beta_i}{2}
(V^{\beta-2{\bf e}_i}_{\rm poly})_{d}(y) -\dbinom{\beta_i}{1}
(\partial_y V^{\beta-{\bf e}_i}_{\rm poly})_i(y) 
\bigg) {\bf e}_{d}. 
\end{aligned}
\end{equation*}
When finding $(\Lambda^{\beta},\Pi^{\beta})$, one needs to be careful about the solenoidal condition $\nabla\cdot\Lambda^{\beta}=0$ as the derivatives with respect to $y$ are involved. Thus the pair of functions $(\Lambda^{\beta}_{\rm poly}(y),\Pi^{\beta}_{\rm poly}(y))$ is introduced as follows. The vector-valued polynomial $\Lambda^{\beta}_{\rm poly}(y)$ is defined by
\begin{equation*}
\begin{aligned}
\Lambda^{\beta}_{\rm poly}(y)
&= 
-\sum_{j=1}^{d-1}
\sum_{i=1}^{d-1}
2\dbinom{\beta_i}{2} 
\Big(\int_0^y \int_0^t 
(V^{\beta-2{\bf e}_i}_{\rm poly})_j(s)
\dd s\dd t
\Big) 
{\bf e}_j \\
&\quad
+ \sum_{i=1}^{d-1}
\dbinom{\beta_i}{1}
\Big(
\int_0^y \int_0^t 
Q^{\beta-{\bf e}_i}_{\rm poly}(s) \dd s\dd t
\Big) 
{\bf e}_i, 
\end{aligned}
\end{equation*}
and the scalar polynomial $\Pi^{\beta}_{\rm poly}(y)$ is defined by
\begin{equation*}
\begin{aligned}
\Pi^{\beta}_{\rm poly}(y) 
&= \sum_{i=1}^{d-1}
\bigg(
2\dbinom{\beta_i}{2} 
\Big(
\int_0^y
(V^{\beta-2{\bf e}_i}_{\rm poly})_{d}(t) 
\dd t
\Big) 
- \dbinom{\beta_i}{1}
(V^{\beta-{\bf e}_i}_{\rm poly})_i(y) 
\bigg). 
\end{aligned}
\end{equation*}
Then, by a simple computation, one has
%
\begin{equation}\label{eq2.step2.proof.prop.bl.beta}
\begin{aligned}
-\Delta\Lambda^{\beta}_{\rm poly}+\nabla \Pi^{\beta}_{\rm poly}
= F^{\beta}_{\rm poly} + \partial_y^2 W^{\beta}_{\rm poly}. 
\end{aligned}
\end{equation}
Hence, by setting 
\begin{equation*}
\begin{aligned}
    &\Lambda^{\beta}(y) = \Lambda^{\beta}_{\rm poly}(y) \eta_+(y), &\qquad& \Pi^{\beta}(y) = \Pi^{\beta}_{\rm poly}(y) \eta_+(y),\\
    &\Lambda_{\rm per}^{\beta}(y) = \Lambda^{\beta}_{\rm poly}(y) (1-\eta_+(y)), &\qquad & \Pi^{\beta}_{\rm per}(y) = \Pi^{\beta}_{\rm poly}(y) (1-\eta_+(y)),
\end{aligned}
\end{equation*}
we see that $(\Lambda^{\beta},\Pi^{\beta})$ is the desired corrector. Indeed, $(\Lambda^{\beta},\Pi^{\beta})$ satisfies \eqref{eq1.step2.proof.prop.bl.beta} with 
\begin{equation}\label{eq3.step2.proof.prop.bl.beta}
\begin{aligned}
H^{\beta}_{\rm per}
&:=
-\Delta\Lambda^{\beta} + \nabla \Pi^{\beta} 
- F^{\beta}_{\rm poly} - \partial_y^2 W^{\beta}_{\rm poly} \\
&=
-\Delta\Lambda^{\beta}_{\rm per} + \nabla \Pi^{\beta}_{\rm per}, 
\end{aligned}
\end{equation}
where \eqref{eq2.step2.proof.prop.bl.beta} is used in the second line. By definition, $H^{\beta}_{\rm per}$ is supported in $\Omega_{{\rm p},<2}$.

We collect the estimates to be used later. By the induction assumption,  
\begin{equation}\label{est1.step2.proof.prop.bl.beta}
\begin{aligned}
|\Lambda^{\beta}_{\rm poly}(y)| 
&\le C_\beta(1+|y|)^{|\beta|}, \\
|\partial_y \Lambda^{\beta}_{\rm poly}(y)| 
+ |\Pi^{\beta}_{\rm poly}(y)| 
&\le C_\beta(1+|y|)^{|\beta|-1}, \\
\|H^{\beta}_{\rm per}\|_{L^2(\Omega_{{\rm p}})}
+ \|\Lambda^{\beta}_{\rm per}\|_{H^1(\Omega_{\rm p})}
+ \|\Pi^{\beta}_{\rm per}\|_{L^2(\Omega_{\rm p})} 
&\le C_\beta. 
\end{aligned}
\end{equation}

\noindent {\bf Step 3: Reduced equations.} We reduce \eqref{eq.bl.beta} using the correctors $W^{\beta}$ and $(\Lambda^{\beta},\Pi^{\beta})$ defined in Steps 1 and 2. Then we apply the Lax-Milgram lemma to the reduced problem.

Let us define the new function $(\widetilde{V}^{\beta},\widetilde{Q}^{\beta})$ by 
\begin{equation*}
\begin{aligned}
\widetilde{V}^{\beta}(x,y)
&=V^{\beta}(x,y)-W^{\beta}(x,y)-\Lambda^{\beta}(y), \\
\widetilde{Q}^{\beta}(x,y)
&=Q^{\beta}(x,y)-\Pi^{\beta}(y). 
\end{aligned}
\end{equation*}
Then we see that $(\widetilde{V}^{\beta},\widetilde{Q}^{\beta})$ must solve the Stokes system with source 
\begin{equation}\label{eq1.step3.proof.prop.bl.beta}
\left\{
\begin{array}{ll}
-\Delta \widetilde{V}^{\beta} 
+ \nabla \widetilde{Q}^{\beta}
= \widetilde{F}^{\beta} &\text{in } \Omega \\
\nabla\cdot \widetilde{V}^{\beta}=0 &\text{in } \Omega \\
\widetilde{V}^{\beta}=0 &\text{on } \partial\Omega. 
\end{array}
\right.
\end{equation}
Here the source $\widetilde{F}^{\beta}$ is given by, thanks to \eqref{eq2.step2.proof.prop.bl.beta} and \eqref{eq3.step2.proof.prop.bl.beta}, 
\begin{equation}\label{eq2.step3.proof.prop.bl.beta}
\begin{aligned}
\widetilde{F}^{\beta}
&=
F^{\beta} + \Delta W^{\beta}
-(-\Delta\Lambda^{\beta} + \nabla \Pi^{\beta}) \\
&=
F^{\beta}_{\rm per}
+ \Delta W^{\beta}_{\rm per} 
- H^{\beta}_{\rm per}. 
\end{aligned}
\end{equation}

We point out that $\widetilde{F}^{\beta}$ is bounded. Moreover, from \eqref{def.FG.beta6} and \eqref{est2.step1.proof.prop.bl.beta}, for $y>3$, 
\begin{equation*}
\begin{aligned}
\widetilde{F}^{\beta}(x,y) 
&=
\sum_{k\in\Z^{d-1}\setminus\{0\}} 
(\F^{\beta}_k(y-3) + \mcG^{\beta}_k(y-3)) e^{-|k|(y-3)} e^{\ii k\cdot x} \\
&=:
\sum_{k\in\Z^{d-1}\setminus\{0\}} 
\widetilde{\F}^{\beta}_k(y-3) e^{-|k|(y-3)} e^{\ii k\cdot x}. 
\end{aligned}
\end{equation*}
The polynomial $\widetilde{\F}^{\beta}_k=\widetilde{\F}^{\beta}_k(z)$ can be estimated as, by \eqref{def.FG.beta8} and \eqref{est4.step1.proof.prop.bl.beta}, 
\begin{equation}\label{est1.step3.proof.prop.bl.beta}
\begin{aligned}
|\widetilde{\F}^{\beta}_k(z)| + |k|^{-1} |\partial_z\widetilde{\F}^{\beta}_k(z)| 
\le 
a^\beta_k (1+|k||z|)^{2|\beta|-1}, 
\end{aligned}
\end{equation}
where the sequence $\{a^\beta_k\}_{k\in\Z^{d-1}\setminus\{0\}}$ is defined by 
\begin{equation}\label{def1.step3.proof.prop.bl.beta}
\begin{aligned}
a^\beta_k:= C_\beta |k|^2 (b^{\beta-2{\bf e}_i}_k+b^{\beta-{\bf e}_i}_k). 
\end{aligned}
\end{equation}
Using the induction assumption, we deduce that $\{a^\beta_k\}_{k\in\Z^{d-1}\setminus\{0\}}$ satisfies \eqref{assump3.prop.bl.F}. This implies, combined with \eqref{est1.step3.proof.prop.bl.beta}, that $\widetilde{F}^{\beta}$ meets all the assumptions in Lemma \ref{prop.bl.F} with $L=3$.

Before applying Lemma \ref{prop.bl.F} to \eqref{eq1.step3.proof.prop.bl.beta}, we verify the unique existence of solutions by the Lax-Milgram lemma. By \eqref{eq2.step3.proof.prop.bl.beta}, the weak formulation of \eqref{eq1.step3.proof.prop.bl.beta} is written as 
\begin{equation*}
\begin{aligned}
\langle \nabla\widetilde{V}^{\beta}, \nabla\varphi\rangle_{\Omega_{{\rm p}}}
&= \langle \widetilde{F}^{\beta}, \varphi\rangle_{\Omega_{{\rm p}}} \\
&= \langle F^{\beta}_{\rm per}- H^{\beta}_{\rm per}, \varphi\rangle_{\Omega_{{\rm p}}}
- \langle \nabla W^{\beta}_{\rm per},  \nabla\varphi\rangle_{\Omega_{{\rm p}}}, \quad \varphi\in \widehat{H}^1_{0,\sigma}(\Omega_{{\rm p}}).
\end{aligned}
\end{equation*}
We estimate all the terms on the right-hand side. Using the Poincar\'{e} inequality in $\Omega_{{\rm p},<3}$, 
\begin{equation*}
\begin{aligned}
|\langle F^{\beta}_{\rm per}- H^{\beta}_{\rm per}, \varphi\rangle_{\Omega_{{\rm p}}}|
&\le 
C(\|F^{\beta}_{\rm per}\|_{L^2(\Omega_{{\rm p},<3})} + \|H^{\beta}_{\rm per}\|_{L^2(\Omega_{{\rm p}})})
\|\nabla \varphi\|_{L^2(\Omega_{{\rm p},<3})} \\
&\quad
+ \bigg|\int_{3}^{\infty} \int_{(-\pi,\pi)^{d-1}} 
F^{\beta}_{\rm per}(x,y)
\varphi(x,y) \dd x \dd y\bigg|. 
\end{aligned}
\end{equation*}
Thanks to the exponential convergence of $F^{\beta}_{\rm per}$, one has
\begin{equation*}
\begin{aligned}
\bigg|\int_{3}^{\infty} \int_{(-\pi,\pi)^{d-1}} 
F^{\beta}_{\rm per}(x,y) \varphi(x,y) \dd x \dd y\bigg|
&\le C \|F^{\beta}_{\rm per}\|_{L^2(\Omega_{{\rm p}})} \|\nabla \varphi\|_{L^2(\Omega_{{\rm p}})}. 
\end{aligned}
\end{equation*}
Thus, by the induction assumption for $F^{\beta}_{\rm per}$ and \eqref{est1.step2.proof.prop.bl.beta}, it follows that 
\begin{equation*}
\begin{aligned}
|\langle F^{\beta}_{\rm per}- H^{\beta}_{\rm per}, \varphi\rangle_{\Omega_{{\rm p}}}|
\le C_\beta\|\nabla \varphi\|_{L^2(\Omega_{\rm p})}.
\end{aligned}
\end{equation*}
From \eqref{est1.step1.proof.prop.bl.beta}, we also have
\begin{equation*}
\begin{aligned}
|\langle \nabla W^{\beta}_{\rm per}, \nabla\varphi\rangle_{\Omega_{\rm p}}|
&\le 
C_\beta \|\nabla \varphi\|_{L^2(\Omega_{\rm p})}. 
\end{aligned}
\end{equation*}
From the two estimates above, we see that there is a unique weak solution $\widetilde{V}^{\beta}\in \widehat{H}^1_{0,\sigma}(\Omega_{{\rm p}})$ of \eqref{eq1.step3.proof.prop.bl.beta} by the Lax-Milgram lemma. Moreover, one can recovers the pressure $\widetilde{Q}^{\beta}\in L^2(\Omega_{\rm p})$ unique up to a constant by using the Bogovskii operator.

Now Lemma \ref{prop.bl.F} applying to \eqref{eq1.step3.proof.prop.bl.beta} yields that $(\widetilde{V}^{\beta},\widetilde{Q}^{\beta})$ can be decomposed into  
\begin{equation*}
\begin{aligned}
\widetilde{V}^{\beta}(x,y)
&= \widetilde{V}^{\beta}_{\rm poly} 
+ \widetilde{V}^{\beta}_{\rm per}(x,y), \\
\widetilde{Q}^{\beta}(x,y)
&= \widetilde{Q}^{\beta}_{\rm per}(x,y),
\end{aligned}
\end{equation*}
where $\widetilde{V}^{\beta}_{\rm poly}$ is a constant vector.
From \eqref{est1.step3.proof.prop.bl.beta} and \eqref{def1.step3.proof.prop.bl.beta} combined with the induction assumption, we see that 
\begin{equation}\label{est2.step3.proof.prop.bl.beta}
\begin{aligned}
|\widetilde{V}^{\beta}_{\rm poly}| 
+ \|\widetilde{V}_{\rm per}\|_{H^1(\Omega_{\rm p})}
+ \|\widetilde{Q}_{\rm per}\|_{L^2(\Omega_{\rm p})} 
&\le C_\beta
\end{aligned}
\end{equation}
and that, for $y>3$, 
\begin{equation*}
\begin{aligned}
\widetilde{V}^{\beta}_{\rm per}(x,y)
&= \sum_{k\in\Z^{d-1}\setminus\{0\}} 
\widetilde{\V}^{\beta}_k(y-3) e^{-|k|(y-3)} e^{\ii k\cdot x}, \\
\widetilde{Q}^{\beta}_{\rm per}(x,y)
&= \sum_{k\in\Z^{d-1}\setminus\{0\}} 
\widetilde{\Q}^{\beta}_k(y-3) e^{-|k|(y-3)} e^{\ii k\cdot x}, 
\end{aligned}
\end{equation*}
where the polynomials $\widetilde{\V}^{\beta}_k=\widetilde{\V}^{\beta}_k(z)$ and $\widetilde{\Q}^{\beta}_k=\widetilde{\Q}^{\beta}_k(z)$ are estimated as 
\begin{equation}\label{est3.step3.proof.prop.bl.beta}
\begin{aligned}
|\widetilde{\V}^{\beta}_k(z)| 
&\le
\tilde{b}^{\beta}_k (1+|k||z|)^{2|\beta|+1}, \\
|k|^{-1} (|\partial_z\widetilde{\V}^{\beta}_k(z)| + |\widetilde{\Q}^{\beta}_k(z)|)
&\le 
\tilde{b}^{\beta}_k (1+|k||z|)^{2|\beta|}, \\
|k|^{-2} |\partial_z^2\widetilde{\V}^{\beta}_k(z)|
+ |k|^{-3} |\partial_z^3\widetilde{\V}^{\beta}_k(z)| 
&\le 
\tilde{b}^{\beta}_k (1+|k||z|)^{2|\beta|}, \\
|k|^{-2} |\partial_z\widetilde{\Q}^{\beta}_k(z)| 
+ |k|^{-3} |\partial_z^2\widetilde{\Q}^{\beta}_k(z)| 
&\le 
\tilde{b}^{\beta}_k (1+|k||z|)^{2|\beta|-1}. 
\end{aligned}
\end{equation}
Moreover, the sequence $\{\tilde{b}^{\beta}_k\}_{k\in\Z^{d-1}\setminus\{0\}}$ satisfies 
\begin{equation}\label{est4.step3.proof.prop.bl.beta}
\begin{aligned}
\sum_{k\in\Z^{d-1}\setminus\{0\}}
|k| |\tilde{b}^{\beta}_k|^2 
\le C_\beta. 
\end{aligned}
\end{equation}

\noindent {\bf Step 4: Existence of a solution.} We finish the existence part of the proof. Let us define $(V^{\beta},Q^{\beta})\in H^1_{{\rm loc}}(\overline{\Omega})^{d}\times L^2_{{\rm loc}}(\overline{\Omega})$ by 
\begin{equation*}
\begin{aligned}
V^{\beta}(x,y)
&=\widetilde{V}^{\beta}(x,y)+W^{\beta}(x,y)+\Lambda^{\beta}(y), \\
Q^{\beta}(x,y)
&=\widetilde{Q}^{\beta}(x,y)+\Pi^{\beta}(y). 
\end{aligned}
\end{equation*}
Then $(V^{\beta},Q^{\beta})$ is a weak solution of \eqref{eq.bl.beta} by the definition of $(\widetilde{V}^{\beta},\widetilde{Q}^{\beta})$ in Step 3. Moreover, $(V^{\beta},Q^{\beta})$ can be decomposed into 
\begin{equation*}
\begin{aligned}
V^{\beta}(x,y) &= V^{\beta}_{\rm poly}(y) + V^{\beta}_{\rm per}(x,y), \\
Q^{\beta}(x,y) &= Q^{\beta}_{\rm poly}(y) + Q^{\beta}_{\rm per}(x,y), 
\end{aligned}
\end{equation*}
where the polynomial part $(V^{\beta}_{\rm poly}(y), Q^{\beta}_{\rm poly}(y))$ is defined by
\begin{equation*}
\begin{aligned}
V^{\beta}_{\rm poly}(y)
&= \widetilde{V}^{\beta}_{\rm poly} 
+ W^{\beta}_{\rm poly}(y) 
+ \Lambda^{\beta}_{\rm poly}(y), \\
Q^{\beta}_{\rm poly}(y)
&= \Pi^{\beta}_{\rm poly}(y), 
\end{aligned}
\end{equation*}
while the periodic part $(V^{\beta}_{\rm per}, Q^{\beta}_{\rm per})$ is defined by
\begin{equation*}
\begin{aligned}
V^{\beta}_{\rm per}(x,y)
&= \widetilde{V}^{\beta}_{\rm per}(x,y)+W^{\beta}_{\rm per}(x,y)
+ \Lambda^{\beta}_{\rm per}(y), \\
Q^{\beta}_{\rm per}(x,y)
&= \widetilde{Q}^{\beta}_{\rm per}(x,y)
+ \Pi^{\beta}_{\rm per}(y). 
\end{aligned}
\end{equation*}
Notice that $(V^{\beta}_{\rm per},Q^{\beta}_{\rm per})$ has the Fourier series expansion, when $y>3$, 
\begin{equation*}
\begin{aligned}
V^{\beta}_{\rm per}(x,y)
&= \sum_{k\in\Z^{d-1}\setminus\{0\}} 
\V^{\beta}_k(y-3) e^{-|k|(y-3)} e^{\ii k\cdot x}, \\
Q^{\beta}_{\rm per}(x,y)
&= \sum_{k\in\Z^{d-1}\setminus\{0\}} 
\Q^{\beta}_k(y-3) e^{-|k|(y-3)} e^{\ii k\cdot x}, 
\end{aligned}
\end{equation*}
with the polynomials 
\begin{equation*}
\begin{aligned}
\V^{\beta}_k(z) := \widetilde{\V}^{\beta}_k(z) + \mcW^{\beta}_k(z), \quad 
\Q^{\beta}_k(z) := \widetilde{\Q}^{\beta}_k(z). 
\end{aligned}
\end{equation*}
Therefore, we find the desired solution $(V^{\beta},Q^{\beta})$ which satisfies \eqref{est3.prop.bl.beta}, \eqref{est7.prop.bl.beta} and \eqref{est8.prop.bl.beta}.

To conclude, we check \eqref{est4.prop.bl.beta}, \eqref{est6.prop.bl.beta}, \eqref{est9.prop.bl.beta} and \eqref{est10.prop.bl.beta}. The estimates in \eqref{est4.prop.bl.beta} and \eqref{est6.prop.bl.beta} follow from \eqref{est1.step1.proof.prop.bl.beta}, \eqref{est1.step2.proof.prop.bl.beta} and \eqref{est2.step3.proof.prop.bl.beta}. The ones in \eqref{est9.prop.bl.beta} and \eqref{est10.prop.bl.beta} are the consequence 
of \eqref{def1.est2.step1.proof.prop.bl.beta}, \eqref{est3.step3.proof.prop.bl.beta}, \eqref{est4.step3.proof.prop.bl.beta}, and the induction assumption. This completes the proof of Proposition \ref{prop.bl.beta}. 
\end{proof}
%

We finish the proof of Proposition \ref{thm.bl.alpha} using Proposition \ref{prop.bl.beta}.

%
\begin{proofx}{Proposition \ref{thm.bl.alpha} (second half)}
Take $(V^{\beta},Q^{\beta})$ in Proposition \ref{prop.bl.beta} for all $\beta\le \alpha$ and define $(v^{\alpha},q^{\alpha})$ by \eqref{ansatz}. Then, by the definition of $(V^{\beta},Q^{\beta})$, we see that $(v^{\alpha},q^{\alpha})$ is a weak solution of \eqref{eq.bl.alpha}. Moreover, it has the decomposition as displayed in \eqref{est2.thm.bl.alpha}, if we set
\begin{equation}\label{eq.vpoly.vhet}
\begin{aligned}
&v^{\alpha}_{\rm poly}(x,y) =
\sum_{\beta\le\alpha} \dbinom{\alpha}{\beta} x^{\alpha-\beta} V^{\beta}_{\rm poly}(y) , \quad
v^{\alpha}_{\rm het}(x,y) = \sum_{\beta\le\alpha} \dbinom{\alpha}{\beta} x^{\alpha-\beta} V^{\beta}_{\rm per}(x,y), \\
&q^{\alpha}_{\rm poly}(x,y) =
\sum_{\beta\le\alpha} \dbinom{\alpha}{\beta} x^{\alpha-\beta} Q^{\beta}_{\rm poly}(y) , \quad
q^{\alpha}_{\rm het}(x,y) = \sum_{\beta\le\alpha} \dbinom{\alpha}{\beta} x^{\alpha-\beta} Q^{\beta}_{\rm per}(x,y). 
\end{aligned}
\end{equation}

Now it suffices to prove the estimates in \eqref{est3.thm.bl.alpha}, \eqref{est4.thm.bl.alpha} and \eqref{est5.thm.bl.alpha} since then \eqref{est1.thm.bl.alpha} follows. From the estimate of $V^{\beta}_{\rm poly}(y)$ in Proposition \ref{prop.bl.beta}, we obtain the first line of \eqref{est3.thm.bl.alpha} by
\begin{equation}
\begin{aligned}
|v^{\alpha}_{\rm poly}(x,y)|
&\le 
\sum_{\beta\le\alpha} 
C|x|^{|\alpha|-|\beta|} (1+|y|)^{|\beta|} \\
&\le
C(1+|x|+|y|)^{|\alpha|}. 
\end{aligned}
\end{equation}
The other estimates can be proved by easy computation using Proposition \ref{prop.bl.beta}. The proof is complete.
\end{proofx}
%

\subsection{Heterogeneous Stokes polynomials}\label{subsec.het.stokespoly}

In this subsection, we would like to prove Theorem \ref{intro.thm.bl} and describe the heterogeneous Stokes polynomials in $\Omega$ satisfying the no-slip boundary condition.

\begin{proofx}{Theorem \ref{intro.thm.bl}}
For every triplet $(\alpha,l,i)\in\Z_{\ge0}^{d-1}\times \Z_{+} \times \{1,\cdots,d\}$, we pick a solution $(v^{\alpha,l}_{(i)}, q^{\alpha,l}_{(i)})$ of \eqref{eq.bl.alpha} in Proposition \ref{thm.bl.alpha}. Fix $m\ge 1$. Then a linear mapping $\mathscr{S}$ on 
\begin{equation*}
\begin{aligned}
\mathcal{P}_{m,0}(\R^d_+)^d
= \spn \{x^\alpha y^l {\bf e}_i~|~ 
\alpha\in\Z_{\ge0}^{d-1}, \ \, 
l\in\Z_{+}, \ \, 
i\in\{1,\cdots,d\}\} 
\end{aligned}
\end{equation*}
to $H^1_{{\rm loc}}(\overline{\Omega})^{d}\times L^2_{{\rm loc}}(\overline{\Omega})$ is induced which maps $x^\alpha y^l {\bf e}_i$ to the chosen $(v^{\alpha,l}_{(i)}, q^{\alpha,l}_{(i)})$. Moreover, for every $P\in\mathcal{P}_{m,0}(\R^d_+)^d$, the image $(v_{P},q_{P}):=\mathscr{S}[P]$ is a weak solution to
\begin{equation}\label{eq.BLp}
\left\{
\begin{array}{ll}
-\Delta v_P + \nabla q_P=0 &\text{in } \Omega \\
\nabla\cdot v_P=0 &\text{in } \Omega \\
v_P + P =0 &\text{on } \partial\Omega.
\end{array}
\right.
\end{equation}

It is remaining to show that $v_P$ has a strictly lower growth rate than $P$ at infinity. By the definition of $\mathcal{P}_{m,0}(\R^d_+)^d$, $P=P(x,y)$ has a form of
\begin{equation}\label{eq.p.StokesP}
    P(x,y) = \sum_{\substack{|\alpha|+l\le m \\ l\ge 1}}\sum_{i=1}^{d} A^{\alpha,l}_i x^\alpha y^l \mathbf{e}_i,
\end{equation}
where $A^{\alpha,l}_i$ are constant coefficients. 
By 
the definition of $\mathscr{S}$, we have 
\begin{equation}\label{eq.vp.qp}
    v_P = \sum_{\substack{|\alpha|+l\le m \\ l\ge 1}}\sum_{i=1}^{d} A^{\alpha,l}_i v^{\alpha,l}_{(i)}, \qquad q_P = \sum_{\substack{|\alpha|+l\le m \\ l\ge 1}} \sum_{i=1}^{d} A^{\alpha,l}_i q^{\alpha,l}_{(i)}. 
\end{equation}
Suppose that $P$ grows like $O(|(x,y)|^m)$, which implies that $|\alpha| + l\le m$ for any nonzero $A^{\alpha,l}_i$. From \eqref{eq.p.StokesP}, we see that $l \ge 1$ and therefore $|\alpha|\le m-1$. On the other hand, the pointwise estimates \eqref{est3.thm.bl.alpha} and \eqref{est5.thm.bl.alpha} imply that $v_P$ grows at most like $O(|(x,y)|^{m-1})$ which is strictly less than the growth rate of $P$ at infinity. This completes the proof.
\end{proofx}

The next proposition gives an intrinsic formula for $\mathscr{S}$ in terms of the building blocks 
$$\{(V^\beta,Q^\beta)\}_{\beta\ge0} = \{(V^{\beta,l}_{(i)},Q^{\beta,l}_{(i)}\}_{\beta\ge0}$$
constructed in Proposition \ref{prop.bl.beta}. Here we recall the dependence on $(l,i)$, which has been omitted for simplicity. We define the matrix $V^{\beta,l}$ and the vector $Q^{\beta,l}$ respectively by
$$V^{\beta,l} = (V^{\beta,l}_{(1)}, V^{\beta,l}_{(2)},\cdots, V^{\beta,l}_{(d)})$$ 
and 
$$Q^{\beta,l} = (Q^{\beta,l}_{(1)}, Q^{\beta,l}_{(2)},\cdots, Q^{\beta,l}_{(d)}).$$

%
\begin{proposition}\label{prop.S.formula}
For each $P \in \mathcal{P}_{m,0}(\R^d_+)^d$, 
the image $(v_P,q_P):=\mathscr{S}[P]$ is represented as 
%
\begin{equation}\label{eq.wp.formula}
v_P(x,y) 
= \sum_{\substack{|\beta|+k \le m \\ \beta \ge 0, k\ge 1}}
\frac{1}{\beta! k!} V^{\beta,k}(x,y) 
\partial_x^\beta \partial_y^k P(x,0)
\end{equation}
and
\begin{equation}\label{eq.qp.formula}
q_P(x,y)  
= \sum_{\substack{|\beta|+k \le m \\ \beta \ge 0, k\ge 1}}
Q^{\beta,k}(x,y) 
\cdot \partial_x^\beta \partial_y^k P(x,0). 
\end{equation}
\end{proposition}
%
%
\begin{proof}
By the linearity of $\mathscr{S}$, it suffices to show \eqref{eq.wp.formula} and \eqref{eq.qp.formula} for monomials. We first consider $v^{\alpha,l}_{(i)}:=v_{P_0}$ for the monomial $P_0(x,y) = x^\alpha y^l \mathbf{e}_i$. By the representation \eqref{ansatz},
\begin{equation}\label{eq1.proof.prop.S.formula}
v^{\alpha,l}_{(i)}(x,y) 
= \sum_{\beta \le \alpha} \dbinom{\alpha}{\beta} x^{\alpha-\beta} V^{\beta,l}_{(i)}(x,y).
\end{equation}
Using the formula
\begin{equation*}
x^{\alpha-\beta} 
= \frac{(\alpha-\beta)!}{\alpha! l!} \partial_x^\beta \partial_y^l (x^\alpha y^l) 
= \frac{(\alpha-\beta)!}{\alpha! l!} \partial_x^\beta \partial_y^l P_0(x,0),
\end{equation*}
we rewrite \eqref{eq1.proof.prop.S.formula} as
\begin{equation}
\begin{aligned}
    v^{\alpha,l}_{(i)}(x,y) & = \sum_{\beta \le \alpha} \dbinom{\alpha}{\beta} \frac{(\alpha-\beta)!}{\alpha! l!}V^{\beta,l}_{(i)}(x,y) 
    \partial_x^\beta \partial_y^l P_0(x,0) \\
    & = \sum_{\substack{|\beta|+k \le m \\ \beta \ge 0, k\ge 1}} 
    \frac{1}{\beta! k!}V^{\beta,k}(x,y) \partial_x^\beta \partial_y^k P_0(x,0),
\end{aligned}   
\end{equation}
where the matrix-vector multiplication has been applied. Hence we have \eqref{eq.wp.formula} for $P_0$. The formula \eqref{eq.qp.formula} for $P_0$ follows from a similar argument. This completes the proof. 
\end{proof}

Now, recall that $\mathcal{S}_m(\R_+^{d})$ consists of all the Stokes polynomials in the upper half space up to degree $m$; see \eqref{def.stokespoly} in Section \ref{sec.stokespoly} for the definition. We define the space of heterogeneous Stokes polynomials up to degree $m$ by 
\begin{equation}
\mathcal{S}_m(\Omega) 
= \{ (w_P,\pi_P) = (P,
Q) + \mathscr{S}[P]~|~ (P,Q) \in\mathcal{S}_m(\R^d_+) \}.
\end{equation}
Note that, for a given $(P,Q) \in \mathcal{S}_m(\R^d_+)$ with $P$ expressed as in \eqref{eq.p.StokesP}, we have
\begin{equation}\label{eq.wp.StokesP}
\begin{aligned}
    w_P = \sum_{\substack{|\alpha|+l\le m \\ l\ge 1}}\sum_{i=1}^{d} A^{\alpha,l}_i (x^\alpha y^l \mathbf{e}_i + v^{\alpha,l}_{(i)}), \qquad
    \pi_P = Q + \sum_{\substack{|\alpha|+l\le m \\ l\ge 1}}\sum_{i=1}^{d} A^{\alpha,l}_i q^{\alpha,l}_{(i)},
\end{aligned}
\end{equation}
where $Q$ can be determined by $P$ up to a constant.
Observe that the mapping $(P,Q) \mapsto (w_P, \pi_P)$ is a linear bijection from $\mathcal{S}_m(\R_+^d)$ to $\mathcal{S}_m(\Omega)$, since $\mathscr{S}[P]$ can be viewed as a lower order perturbation added to $(P,Q)$. 
This combined with the results in Section \ref{sec.stokespoly} means that
\begin{equation*}
 \dim \mathcal{S}_m(\Omega) = \dim \mathcal{S}_m(\R^{d}_+) = d \dbinom{m+d-2}{d-1}.
\end{equation*}

Finally, the following proposition is useful in proving the higher-order regularity.
%
\begin{proposition}\label{prop.Dwp-Dp}
Let $P$ and $w_P$ be given by \eqref{eq.p.StokesP} and \eqref{eq.wp.StokesP}. Then for $r\ge 1$,
\begin{equation}
    \bigg( \dashint_{B_{r,+}} |\nabla w_P - \nabla P|^2 \bigg)^{1/2} \le \sum_{\substack{|\alpha|+l\le m \\ l\ge 1}}\sum_{i=1}^{d} C_{\alpha,l} |A^{\alpha,l}_i|\ r^{|\alpha|-1/2}.
\end{equation}
The constant $C_{\alpha,l}$ depends on $\Omega,\alpha$ and $l$. 
\end{proposition}
%
\begin{proof}
This is a corollary of \eqref{est1.thm.bl.alpha} in Proposition \ref{thm.bl.alpha} and the fact that $v_P = w_P - P$, where $v_P$ is given by \eqref{eq.vp.qp}.
\end{proof}
%

\section{Large-scale regularity}\label{sec.largescale}
In this section, we would like to prove Theorem \ref{thm.main}. The proof uses the boundary layers in the previous section and will be provided at the end of Subsection \ref{subsec.pressure}.

Let $(u, p)$ be a weak solution of 
\begin{equation}\tag{S}\label{eq.S}
\left\{
\begin{array}{ll}
-\Delta u+\nabla p=0&\mbox{in}\ B_{R,+}\\
\nabla\cdot u=0&\mbox{in}\ B_{R,+}\\
u=0&\mbox{on}\ \Gamma_R. 
\end{array}
\right.
\end{equation}
In \cite{HPZ}, we have proved the large-scale Lipschitz estimate for Stokes system without periodicity assumption on the boundary. That is, for any $r\in (1, R/4)$,
\begin{equation}\label{est.StokesLip}
    \bigg( \dashint_{B_{r,+}} |\nabla u|^2 \bigg)^{1/2} + \bigg( \dashint_{B_{r,+}} |p - \dashint_{B_{R/2,+}} p|^2 \bigg)^{1/2} \le C\bigg( \dashint_{B_{R,+}} |\nabla u|^2 \bigg)^{1/2},
\end{equation}
where $C$ is independent of $r$ and $R$. Note that this estimate is scaling and translation invariant. Also, by a standard argument in John domains with the Bogovski lemma, the region $B_{R/2,+}$ for the averaged pressure can be replaced by any $B_{cR,+}(x,0)$ with $c\in (0,1/4)$ and $|x|<R/4$. In this section, we will take advantage of this a priori estimate.

Throughout this section, for $r>1$, we define
\begin{equation*}
    Q_{r,+} = Q_r(0) \cap \{ y > -1 \}.
\end{equation*}
Observe that $B_{r,+} \subset Q_{r,+}$ and $|Q_{r,+} \setminus B_{r,+}| \le r^{-1} |B_{r,+}|$. Moreover, all the constants depending on the order $m$ are denoted by $C_m$.

\subsection{Velocity estimate}\label{subsec.velocityest}
To study the higher-order regularity, we first define the higher-order excess:
\begin{equation}
    H_{m}(u; r) : = \inf_{(w_m,q_m) \in \mathcal{S}_m(\Omega) } \bigg( \dashint_{B_{r,+}} |\nabla u - \nabla w_m|^2  \bigg)^{1/2}.
\end{equation}
Since the Stokes system is linear, unlike in \cite{HPZ}, we do not include the pressure part in $H_m$, whose estimate could be recovered from the velocity easily; see Subsection \ref{subsec.pressure}.

%
\begin{lemma}
Let $(u, p)$ be the weak solution of \eqref{eq.S} in $B_{2r,+}$. Then for every $r>1$, there exists $(v_r, q_r)$ satisfying
\begin{equation}\label{eq.vr.appro}
\left\{
\begin{array}{ll}
-\Delta v_r +\nabla q_r =0&\mbox{in}\ Q_{r,+}\\
\nabla\cdot v_r =0&\mbox{in}\ Q_{r,+}\\
v_r = 0 &\mbox{on}\ \partial Q_{r,+} \cap \{ y = -1 \}
\end{array}
\right.
\end{equation}
such that
\begin{equation}\label{est.Due-Dvr}
    \bigg( \dashint_{Q_{r,+} } |\nabla u - \nabla v_r|^2 \bigg)^{1/2} \le Cr^{-1/2} \bigg( \dashint_{Q_{2r,+}} |\nabla u|^2 \bigg)^{1/2},
\end{equation}
where $C$ depends only on the dimension $d$. 
\end{lemma}
%

%
\begin{proof}
Let $(v_r, q_r)$ be the weak solution of \eqref{eq.vr.appro} subject to the Dirichlet boundary condition $v_r = u$ on $\partial Q_{r,+}$ in the sense of trace in $H^{1/2}(\partial Q_{r,+})^d$. Since $u = 0$ on $\partial Q_{r,+}\cap \{y=-1\}$, we also have $v_r = 0$ on $\partial Q_{r,+}\cap \{y=-1\}$. Note that the solvability of \eqref{eq.vr.appro} is guaranteed due to the compatibility condition
\begin{equation}
    \int_{\partial Q_{r,+}} v_r\cdot n d\sigma = \int_{\partial Q_{r,+}} u\cdot n d\sigma = \int_{Q_{r,+}} \nabla\cdot u =0.
\end{equation}
Also note that the domains for $(u,p)$ and $(v_r,q_r)$ are different. We have to use the variational equations for them separately. Recall that $u$ is automatically zero-extended. Then $u-v_r \in H_0^1(Q_{r,+})$ and $\nabla\cdot (u - v_r) = 0$ in $Q_{r,+}$. Testing the system \eqref{eq.vr.appro} against $u -v_r$, we have
\begin{equation}\label{eq.vr.var}
    \int_{Q_{r,+}} \nabla v_r \cdot \nabla(u -v_r) = 0.
\end{equation}
To use the system for $(u,p)$, we introduce a cut-off function $\eta = \eta(y)$ such that $\eta(y) = 1$ for $y\ge 1$, $\eta(y) = 0$ for $y<0$ and $|\eta'(y)|\le 2$. Testing the system \eqref{eq.S} against $(u-v_r)\eta^2$, we have
\begin{equation}\label{eq.u.var}
    \int_{B_{r,+}} \nabla u \cdot \nabla ((u-v_r) \eta^2) 
    = \int_{B_{r,+}} 
    (p - L) (u - v_r) \cdot (2\eta \nabla \eta),
\end{equation}
where $L\in \R$ is an arbitrary constant to be specified later. Combining \eqref{eq.vr.var} and \eqref{eq.u.var}, we have
\begin{equation}\label{eq.u-vr.var}
\begin{aligned}
    \int_{Q_{r,+}} (\nabla u- \nabla v_r) \cdot (\nabla u - \nabla v_r) & = \int_{B_{r,+}} \nabla u \cdot \nabla ((u-v_r) (1-\eta^2)) \\
    &\qquad +\int_{B_{r,+}} 
    (p - L) (u - v_r) \cdot (2\eta \nabla \eta). 
\end{aligned}
\end{equation}

Now we estimate the integrals on the right-hand side of \eqref{eq.u-vr.var}. Since $1-\eta^2$ is supported in $\{ y\le 1\}$, by the Poincar\'{e} inequality on the thin layer $Q_{r,+}\cap \{ y<1  \}$, we have
\begin{equation*}
\begin{aligned}
    & \bigg| \int_{B_{r,+}} \nabla u \cdot \nabla ((u-v_r) (1-\eta^2)) \bigg| \\
    &\qquad \le C\bigg( \int_{B_{r,+} \cap \{ y<1 \}} |\nabla u|^2 \bigg)^{1/2} \bigg( \int_{Q_{r,+} \cap \{y<1 \} } |\nabla u - \nabla v_r|^2 \bigg)^{1/2} \\
    & \qquad \le Cr^{(d-1)/2} \bigg( \dashint_{B_{2r,+} } |\nabla u|^2 \bigg)^{1/2} \bigg( \int_{Q_{r,+}} |\nabla u - \nabla v_r|^2 \bigg)^{1/2},
\end{aligned}
\end{equation*}
In the last inequality, after decomposing $B_{r,+} \cap \{ y<1 \}$ into finitely many sets whose volume are comparable to $B_{1,+}$, we used \eqref{est.StokesLip} on each of them.

To estimate the second integral on the right-hand side of \eqref{eq.u-vr.var}, we choose
\begin{equation*}
    L = \dashint_{B_{r,+}} p.
\end{equation*}
Since $\nabla \eta$ is supported in $\{y\le 1\}$, 
it follows from \eqref{est.StokesLip} and the Poincar\'{e} inequality that
\begin{equation*}
\begin{aligned}
    & \bigg| \int_{B_{r,+}} (p - L) (u - v_r) \cdot (2\eta \nabla \eta) \bigg| \\
    & \qquad \le C \bigg( \int_{B_{r,+} \cap \{ y<1\} } |p-L|^2 \bigg)^{1/2} \bigg( \int_{Q_{r,+} \cap \{y<1 \} } |\nabla u - \nabla v_r|^2 \bigg)^{1/2} \\
    & \qquad \le Cr^{(d-1)/2} \bigg( \dashint_{B_{2r,+} } |\nabla u|^2 \bigg)^{1/2} \bigg( \int_{Q_{r,+}} |\nabla u - \nabla v_r|^2 \bigg)^{1/2}. 
\end{aligned}
\end{equation*}
As a result, \eqref{eq.u-vr.var} gives
\begin{equation}
    \bigg( \int_{Q_{r,+}} |\nabla u - \nabla v_r|^2 \bigg)^{1/2} \le Cr^{(d-1)/2} \bigg( \dashint_{B_{2r,+} } |\nabla u|^2 \bigg)^{1/2}, 
\end{equation}
which is the desired estimate \eqref{est.Due-Dvr}. The proof is complete.
\end{proof}
%

Note that $(v_r, q_r)$ in the above lemma is smooth near the lower boundary. 
In particular, we can find a polynomial $P_r(x,y)$ such that
\begin{equation}\label{est.Dvr-DPr}
    |\nabla v_r(x,y) - \nabla P_r(x,y)| \le C_m (|x|+|y| + 1)^{m} \| \nabla^{m+1} v_r\|_{L^\infty(Q_{r/2,+})},
\end{equation}
where $P_r$ is in fact the Taylor polynomial of $v_r$ of degree $m$ centered at $(0,-1)$, i.e.,
\begin{equation}
    P_r(x,y) = \sum_{k=0}^m \frac{1}{k!} \nabla^k v_r(0,-1): (x,y+1)^{\otimes k}.
\end{equation}
We point out that $P_r(x,y)$ actually is a polynomial solution (velocity component) to the system \eqref{eq.vr.appro}. This can be verified by a simple computation.

A standard regularity result shows that
\begin{equation}\label{est. Dkvr}
    \| \nabla^{k} v_r\|_{L^\infty(Q_{r/2,+})} \le \frac{C_k}{r^{k-1}} \bigg( \dashint_{Q_{r,+} } |\nabla v_r|^2 \bigg)^{1/2} \le \frac{C_k}{r^{k-1}} \bigg( \dashint_{Q_{2r,+}} |\nabla u|^2 \bigg)^{1/2},
\end{equation}
where the second inequality follows from \eqref{est.Due-Dvr}.
For convenience, we will shift the polynomial $P_r(x)$ so that it is centered at the origin. Define
\begin{equation*}
    P_r^*(x,y) = P_r(x,y-1) = \sum_{k=0}^m \frac{1}{k!} \nabla^k v_r(0,-1): (x,y)^{\otimes k}.
\end{equation*}
By the mean value theorem and \eqref{est. Dkvr}, for $|(x,y)| \in [0, r)$, 
\begin{equation}\label{est.DPr-DPr*}
\begin{aligned}
    |\nabla P_r(x,y) - \nabla P_r^*(x,y)| & \le \frac{1}{r} \Big( \sum_{k=2}^m C_k \frac{(|x|+|y|+1)^{k-2}}{r^{k-2}} \Big) \bigg( \dashint_{Q_{2r,+}} |\nabla u|^2 \bigg)^{1/2} \\
    & \le  \frac{C_m}{r} \bigg( \dashint_{Q_{2r,+} } |\nabla u|^2 \bigg)^{1/2}.
\end{aligned}
\end{equation}
Consequently, for any $\theta \in (0,1/4)$, from \eqref{est.Dvr-DPr}, \eqref{est. Dkvr} and \eqref{est.DPr-DPr*}, we have
\begin{equation}\label{est.Dvr-DPr*}
\begin{aligned}
    \bigg( \dashint_{B_{\theta r,+}} |\nabla v_r - \nabla P_r^*|^2 \bigg)^{1/2} & \le C_m \Big( \frac{(\theta r + 1)^m}{r^m} + \frac{1}{r} \Big) \bigg( \dashint_{Q_{2r,+} } |\nabla u|^2 \bigg)^{1/2} \\
    & \le C_m \Big(\theta^m +  \frac{1}{r}\Big) \bigg( \dashint_{Q_{2r,+} } |\nabla u|^2 \bigg)^{1/2},
\end{aligned}
\end{equation}
provided that $\theta r > 1$.

On the other hand, from Proposition \ref{prop.Dwp-Dp} and \eqref{est. Dkvr}, there exists $(w_m, \pi_m) \in \mathcal{S}_m(\Omega)$ such that for any $t\in (1,r)$, 
\begin{equation}\label{est.DWr-DPr*}
    \bigg( \dashint_{Q_{t,+} } |\nabla w_m - \nabla P_r^*|^2 \bigg)^{1/2} \le C_m t^{-1/2} \bigg( \dashint_{Q_{2r,+} } |\nabla u|^2 \bigg)^{1/2}. 
\end{equation}

Combining these estimates, we have
%
\begin{corollary}
For any $\theta \in (0,1/4)$ and $\theta r > 1$, 
\begin{equation}\label{eq.Hm.thetar-2r}
    H_m(u; \theta r) 
    \le  
    (C_m \theta^m + C_{m,\theta} r^{-1/2}) H_m (u; 2r),
\end{equation}
where $C_m$ depends on $\Omega$ and $m$, and $C_{m,\theta}$ depends on $\Omega,m$ and $\theta$.
\end{corollary}
%
\begin{proof}
By the triangle inequality, we have
\begin{equation*}
    \begin{aligned}
    H_m(u; \theta r) & = \inf_{(w_m,q_m) \in \mathcal{S}_m(\Omega)  } \bigg( \dashint_{B_{\theta r,+}} |\nabla u - \nabla w_m|^2  \bigg)^{1/2} \\
    & \le \bigg( \dashint_{B_{\theta r,+} } |\nabla u - \nabla v_r|^2 \bigg)^{1/2} + \bigg( \dashint_{B_{\theta r,+} } |\nabla P_r^* - \nabla v_r|^2 \bigg)^{1/2} \\
    & \qquad + \inf_{(w_m,q_m) \in \mathcal{S}_m(\Omega)  } \bigg( \dashint_{B_{\theta r,+}} |\nabla P_r^* - \nabla w_m |^2  \bigg)^{1/2} \\
    & \le ( C_{m,\theta} r^{-1/2} + C_m \theta^m ) 
    \bigg( \dashint_{B_{2r,+}} 
    |\nabla u|^2 \bigg)^{1/2},
    \end{aligned}
\end{equation*}
where we have used \eqref{est.Due-Dvr}, \eqref{est.Dvr-DPr*} and \eqref{est.DWr-DPr*} in the last inequality.

Finally, observe that $( u - w_m,\ p - \pi_m)$
is also a weak solution of \eqref{eq.S} for any $(w_m, \pi_m) \in \mathcal{S}_m(\Omega)$. Applying the above estimate to all such weak solutions and taking the infimum over all possible $(w_m, \pi_m) \in \mathcal{S}_m(\Omega)$, we obtain the assertion.
\end{proof}
%

%
\begin{proposition}\label{thm.Hr}
For any $r\in (1,R)$ and $\mu\in (0,1)$,
\begin{equation}\label{eq.Hm.m-mu}
    H_m(u; r) 
    \le C_{m,\mu} \Big( \frac{r}{R} \Big)^{m-\mu} H_m (u;R),
\end{equation}
where $C_{m,\mu}$ depends on $\Omega,m$ and $\mu$.
\end{proposition}
%
\begin{proof}
First of all, for any $\mu\in (0,1)$, we can find sufficiently small $\theta \in (0,1/4)$ such that $C_m \theta^m \le (1/2)^{m+1-\mu} \theta^{m-\mu}$. Fix such $\theta$. Set 
$$\rho_0 = \Big(C_{m,\theta} \Big(\frac12\Big)^{-m-1+\mu} \theta^{-m+\mu}\Big)^2.$$
Then, for any $r\ge \rho_0$, we have $C_{m,\theta} r^{-1/2} \le (1/2)^{m+1-\mu}\theta^{m-\mu}$. Consequently,
it follows from \eqref{eq.Hm.thetar-2r} that
\begin{equation}\label{eq.Hm.iteration}
    H_m(u;\theta r) \le (\theta/2)^{m-\mu} H_m(u; 2r).
\end{equation}
Now let $\theta_1 = \theta/2$. By an iteration, we have
\begin{equation}
    H(u;\theta_1^k R) \le \theta_1^{k(m-\mu)} H(u; R),
\end{equation}
provided $\theta_1^{k} R \ge \rho_0$. This implies \eqref{eq.Hm.m-mu} for any $r\ge \rho_0$. Finally, the case $r\in (1,\rho_0)$ follows from the case $r = \rho_0$ by enlarging the constant $C_{m,\mu}$. The proof is complete. 
\end{proof}

\subsection{Improvements}
Let $(W_r^{*}, \Pi_r^{*})$ denote the minimizer for $H_m (u; r)$, namely,
\begin{equation*}
    H_m (u;r) = \bigg( \dashint_{B_{r,+}} |\nabla u - \nabla W_r^{*}|^2  \bigg)^{1/2}.
\end{equation*}
Clearly, the minimizer $W_r^*$ in $H_m(u;r)$, as an approximation of $u$, depends essentially on the radius $r$. In the following, we will show that we can choose a fixed approximation independent of the radius $r$.
To this end, we have to take advantage of the polynomial-like properties of the heterogeneous Stokes polynomials in $\mathcal{S}_m(\Omega)$.

The next lemma describes a polynomial-like growth property for the heterogeneous polynomials.

%
\begin{lemma}[Growth estimate]\label{lem.growth}
Let $m\in \Z_{+}$. Then there exist $r_0> 1$ and $C_m>c_m>0$ depending on $m$ and $\Omega$ such that if $r\ge r_0$, then for any $(w,\pi)=(w_P,\pi_P) \in 
\mathcal{S}_m(\Omega)$ with  
\begin{equation*}
    P(x,y) = \sum_{ \substack{ |\alpha|+l \le m\\ l \ge 1}} \sum_{i=1}^d A^{\alpha,l}_i x^\alpha y^l \mathbf{e}_i, 
\end{equation*}
we have
\begin{equation}\label{est.0Dw.Growth}
\begin{aligned}
    c_m \sum_{ \substack{ |\alpha|+l \le m\\ l \ge 1}} \sum_{i=1}^d  |A^{\alpha,l}_i| \ r^{|\alpha|+l-1}
    & \le \bigg( \dashint_{B_{r,+}} |\nabla w|^2 \bigg)^{1/2} \\
    & \le C_m \sum_{ \substack{ |\alpha|+l \le m\\ l \ge 1}} \sum_{i=1}^d |A^{\alpha,l}_i| \ r^{|\alpha|+l-1}. 
\end{aligned}
\end{equation}
In particular, for any $s\ge r\ge r_0$,
\begin{equation}\label{est.Dw.Growth}
\bigg( \dashint_{B_{s,+}} |\nabla w|^2 \bigg)^{1/2} \le C_{m} \Big( \frac{s}{r} \Big)^{m-1} \bigg( \dashint_{B_{r,+}} |\nabla w|^2 \bigg)^{1/2}.
\end{equation}
\end{lemma}
%
%
\begin{proof}
From a well-known property of polynomials, for any $r\ge1$,
\begin{equation}\label{est.0PpartOfW.Br}
\begin{aligned}
    c_m \sum_{ \substack{ |\alpha|+l \le m\\ l \ge 1}} 
    \sum_{i=1}^d |A^{\alpha,l}_i| \ r^{|\alpha|+l-1}
    &\le \bigg( \dashint_{B_{r,+}} |\nabla P|^2 \bigg)^{1/2} \\
    &\le C_m \sum_{ \substack{ |\alpha|+l \le m\\ l \ge 1}} 
    \sum_{i=1}^d |A^{\alpha,l}_i| \ r^{|\alpha|+l-1}, 
\end{aligned}
\end{equation}
with some $C_m>c_m>0$ depending on $m$. On the other hand, Proposition \ref{prop.Dwp-Dp} implies that if $r\ge r_0$, for some sufficiently large $r_0> 1$ depending on $m$ and $\Omega$, then
\begin{equation}\label{est.0BLPartOfW.Br}
    \bigg( \dashint_{B_{r,+}} |\nabla w_P - \nabla P|^2 \bigg)^{1/2} 
    \le \frac{c_m}{2} \sum_{\substack{|\alpha|+l\le m \\ l\ge 1}}\sum_{i=1}^{d} |A^{\alpha,l}_i|\ r^{|\alpha|+l-1}.
\end{equation}
Combining \eqref{est.0PpartOfW.Br} and \eqref{est.0BLPartOfW.Br}, we obtain
\begin{equation}\label{est.Dwp.2Sides}
\begin{aligned}
    \frac{c_m}{2} \sum_{ \substack{ |\alpha|+l \le m\\ l \ge 1}} \sum_{i=1}^d  |A^{\alpha,l}_i| \ r^{|\alpha|+l-1}
    & \le \bigg( \dashint_{B_{r,+}} |\nabla w_P|^2 \bigg)^{1/2} \\
    & \le \big(C_m+\frac{c_m}{2}\big) \sum_{ \substack{ |\alpha|+l \le m\\ l \ge 1}} \sum_{i=1}^d |A^{\alpha,l}_i| \ r^{|\alpha|+l-1},
\end{aligned}
\end{equation}
for all $r\ge r_0$. This implies \eqref{est.0Dw.Growth} if the constants are rewritten. The estimate \eqref{est.Dw.Growth} directly follows from \eqref{est.0Dw.Growth}. This completes the proof. 
\end{proof}

\begin{proposition}\label{thm.Due-Dwr0}
There exists $r_0 >1$ depending on $m$ and $\Omega$ such that for any $r\in (1, R)$,
\begin{equation*}
    \bigg( \dashint_{B_{r,+}} |\nabla u - \nabla W_{r_0}^{*}|^2  \bigg)^{1/2} \le C_{m,\mu} \Big( \frac{r}{R} \Big)^{m-\mu} H_m (u;R).
\end{equation*}
\end{proposition}

\begin{proof}
Let $r_0$ be given in Lemma \ref{lem.growth} and $r_k = 2^k r_0$ for $k\in\Z_{+}$. For each $j\in\Z_{+}$, the triangle inequality implies
\begin{equation*}
\begin{aligned}
   \bigg( \dashint_{B_{r_{j-1},+}} |\nabla W_{r_j}^{*} - \nabla W_{r_{j-1}}^{*}|^2  \bigg)^{1/2} & \le H_m (u; r_{j-1}) + CH_m(u; r_j) \\
   & \le C_{m,\mu} \Big( \frac{r_{j-1}}{R} \Big)^{m-\mu} H_m(u;R),
\end{aligned}
\end{equation*}
where we have used Proposition \ref{thm.Hr} in the last inequality. Then, by the triangle inequality and Lemma \ref{lem.growth}, we have
\begin{equation*}
    \begin{aligned}
    \bigg( \dashint_{B_{r_k,+}} |\nabla W_{r_k}^{*} - \nabla W_{r_0}^{*}|^2  \bigg)^{1/2} &
    \le \sum_{j=1}^k \bigg( \dashint_{B_{r_k,+}} |\nabla W_{r_j}^{*} - \nabla W_{r_{j-1}}^{*}|^2  \bigg)^{1/2} \\
    & \le \sum_{j=1}^k C_m \Big( \frac{r_k}{r_{j-1}} \Big)^{m-1} \bigg( \dashint_{B_{r_{j-1},+}} |\nabla W_{r_j}^{*} - \nabla W_{r_{j-1}}^{*}|^2  \bigg)^{1/2} \\
    & \le \sum_{j=1}^k C_{m,\mu} \Big( \frac{r_k}{r_{j-1}} \Big)^{m-1} \Big( \frac{r_{j-1}}{R} \Big)^{m-\mu}  H_m (u;R) \\
    & \le C_{m,\mu} \Big( \frac{r_{k}}{R} \Big)^{m-\mu} H_m (u;R).
    \end{aligned}
\end{equation*}
Consequently, for any $k\in\Z_{+}$ with $r_k<R$,
\begin{equation*}
\begin{aligned}
    & \bigg( \dashint_{B_{r_k,+}} |\nabla u - \nabla W_{r_0}^{*}|^2  \bigg)^{1/2} \\
    & \le \bigg( \dashint_{B_{r_{k},+}} |\nabla u - \nabla W_{r_k}^{*}|^2  \bigg)^{1/2} + \bigg( \dashint_{B_{r_{k},+}} |\nabla W_{r_{k}}^{*} - \nabla W_{r_0}^{*}|^2  \bigg)^{1/2} \\
    & \le C_{m,\mu} \Big( \frac{r_{k}}{R} \Big)^{m-\mu} H_m (u;R).
 \end{aligned}   
\end{equation*}
This implies the assertion for each $r\in (r_0,R)$. Finally, as before, the case $r\in (1,r_0)$ can be obtained by enlarging the constant. The proof is complete. 
\end{proof}

Next, we will improve the exponent from $m-\mu$ in Proposition \ref{thm.Hr} to $m$. First, by Proposition \ref{thm.Due-Dwr0} (applied with $m$ replaced by $m+1$), there exists $(\widetilde{W}_{r_0}^{*},\widetilde{\Pi}_{r_0}^{*}) \in \mathcal{S}_{m+1}(\Omega)$ such that for $r\in (1,R)$,
\begin{equation}\label{est.Due-DWr0.Br}
    \bigg( \dashint_{B_{r,+}} |\nabla u - \nabla \widetilde{W}_{r_0}^{*}|^2  \bigg)^{1/2} \le C_{m+1,\mu} \Big( \frac{r}{R} \Big)^{m+1-\mu} H_{m+1}(u;R).
\end{equation}
Now let $(\overline{W}_{r_0}^{*},\overline{\Pi}_{r_0}^{*})$ be the projection of $(\widetilde{W}_{r_0}^{*},\widetilde{\Pi}_{r_0}^{*})$ into $\mathcal{S}_m(\Omega)$. This means that if
\begin{equation}\label{eq.wt.Wr0}
    \widetilde{W}_{r_0}^{*}(x,y) = \sum_{\substack{|\alpha|+l\le m+1 \\ l\ge 1}}\sum_{i=1}^{d} \widetilde{A}^{\alpha,l}_i (x^\alpha y^l \mathbf{e}_i + v^{\alpha,l}_{(i)}),
\end{equation}
then
\begin{equation}\label{eq.ol.Wr0}
    \overline{W}_{r_0}^{*}(x,y) = \sum_{\substack{|\alpha|+l\le m \\ l\ge 1}}\sum_{i=1}^{d} \widetilde{A}^{\alpha,l}_i (x^\alpha y^l \mathbf{e}_i + v^{\alpha,l}_{(i)}).
\end{equation}

The next proposition proves the velocity estimate in Theorem \ref{thm.main}.

\begin{proposition}\label{thm.main.velocity}
For any $r\in (1,R)$, 
\begin{equation}\label{est.Du-DbarW}
    \bigg( \dashint_{B_{r,+}} |\nabla u - \nabla \overline{W}_{r_0}^{*}|^2  \bigg)^{1/2} \le C_{m} \Big( \frac{r}{R} \Big)^{m} \bigg( \dashint_{B_{R,+}} |\nabla u |^2  \bigg)^{1/2}.
\end{equation}
\end{proposition}

\begin{proof}
Let $\tilde{r}_0>1$ be the constant in Lemma \ref{lem.growth} corresponding to $m+1$. It suffices to prove the assertion \eqref{est.Du-DbarW} for $r\in (\tilde{r}_0,R)$. Suppose $\widetilde{W}_{r_0}^{*}$ and $\overline{W}_{r_0}^{*}$ are given by \eqref{eq.wt.Wr0} and \eqref{eq.ol.Wr0}, respectively. Let $E_{r_0}^{*} := \widetilde{W}_{r_0}^{*} - \overline{W}_{r_0}^{*}$. By the structure of $E_{r_0}^{*}$, we know that
\begin{equation*}
    E_{r_0}^{*}(x,y) = \sum_{\substack{|\alpha|+l= m+1 \\ l\ge 1}}\sum_{i=1}^{d} \widetilde{A}^{\alpha,l}_i (x^\alpha y^l \mathbf{e}_i + v^{\alpha,l}_{(i)}) 
\end{equation*}
and that $E_{r_0}^{*}\in \mathcal{S}_{m+1}(\Omega)$. 
Thus Lemma \ref{lem.growth} applies to $E_{r_0}^{*}$. For $r\ge \tilde{r}_0$, we have 
\begin{equation*}
\begin{aligned}
\bigg( \dashint_{B_{r,+}} |\nabla E_{r_0}^{*}|^2 \bigg)^{1/2} 
    & \le C_{m+1} r^{m}
    \sum_{ \substack{ |\alpha|+l = m+1\\ l \ge 1}} \sum_{i=1}^d |A^{\alpha,l}_i|. 
\end{aligned}
\end{equation*}
On the other hand, by Lemma \ref{lem.growth} again, we see that 
\begin{equation*}
\begin{aligned}
    c_{m+1} \sum_{ \substack{ |\alpha|+l \le m+1\\ l \ge 1}} \sum_{i=1}^d  |A^{\alpha,l}_i| \ r^{|\alpha|+l-1}
    & \le \bigg( \dashint_{B_{r,+}} |\nabla \widetilde{W}_{r_0}^{*} |^2 \bigg)^{1/2}.
\end{aligned}
\end{equation*}
Combining the above two estimates, we have
\begin{equation*}
\begin{aligned}
    \bigg( \dashint_{B_{r,+}} |\nabla E_{r_0}^{*} |^2 \bigg)^{1/2} 
    &\le C_m \Big( \frac{r}{R} \Big)^{m} \bigg( \dashint_{B_{R,+}} |\nabla \widetilde{W}_{r_0}^{*} |^2 \bigg)^{1/2} \\
    &\le C_m \Big( \frac{r}{R} \Big)^{m}\bigg( \dashint_{B_{R,+}} |\nabla u |^2 \bigg)^{1/2},
\end{aligned}
\end{equation*}
where we have used \eqref{est.Due-DWr0.Br} in the last inequality.

Finally, another triangle inequality yields
\begin{equation*}
\begin{aligned}
    \bigg( \dashint_{B_{r,+}} |\nabla u - \nabla \overline{W}_{r_0}^{*}|^2  \bigg)^{1/2} & \le \bigg( \dashint_{B_{r,+}} |\nabla u - \nabla \widetilde{W}_{r_0}^{*}|^2  \bigg)^{1/2} + \bigg( \dashint_{B_{r,+}} |\nabla E_{r_0}^{*} |^2 \bigg)^{1/2} \\
    & \le C_m \Big( \frac{r}{R} \Big)^{m} \bigg( \dashint_{B_{R,+}} |\nabla u |^2 \bigg)^{1/2}
\end{aligned}
\end{equation*}
for $r\ge \tilde{r}_0$. This completes the proof.
\end{proof}

\subsection{Pressure estimate}\label{subsec.pressure}
Previously, we have constructed $\overline{W}_{r_0}^{*}$ as a fixed good approximation of $u$ for all $r\in (1,R)$ (where $r_0$ is given in Lemma \ref{lem.growth}). Note that there is also an associated pressure counterpart denoted by $\overline{\Pi}_{r_0}^{*}$. Moreover, it satisfies
\begin{equation}\label{eq.u-W*}
\left\{
\begin{array}{ll}
-\Delta (u - \overline{W}_{r_0}^{*})+\nabla (p - \overline{\Pi}_{r_0}^{*}) =0&\mbox{in}\ B_{R,+}\\
\nabla\cdot (u - \overline{W}_{r_0}^{*})=0&\mbox{in}\ B_{R,+} \\
u - \overline{W}_{r_0}^{*}=0&\mbox{on}\ \Gamma_{R} . 
\end{array}
\right.
\end{equation}

\begin{lemma}\label{lem.pe.Br}
For any $r\in (1, R/2)$,
\begin{equation*}
    \bigg( \dashint_{B_{r,+}} |p - \overline{\Pi}_{r_0}^{*} -  \dashint_{B_{r,+}} (p - \overline{\Pi}_{r_0}^{*})|^2 \bigg)^{1/2} \le C_m \Big( \frac{r}{R} \Big)^{m} \bigg( \dashint_{B_{R,+}} |\nabla u|^2 \bigg)^{1/2}.
\end{equation*}
\end{lemma}
\begin{proof}
This follows the estimate of the velocity and a routine argument in John domains. By Definition \ref{def.John2}, there is a bounded John domain $\Omega_{r}$ such that
%
\begin{equation}
    B_{r,+} \subset \Omega_{r} \subset B_{2r,+} \subset B_{R,+}.
\end{equation}
By \eqref{eq.u-W*} and the Bogovski lemma,
\begin{equation}\label{est.p-Pi*.Or}
\begin{aligned}
    \bigg( \dashint_{\Omega_{r}} |p - \overline{\Pi}_{r_0}^{*} -  \dashint_{\Omega_{r}} (p - \overline{\Pi}_{r_0}^{*})|^2 \bigg)^{1/2} & \le C\bigg( \dashint_{\Omega_{r}} |\nabla u - \nabla \overline{W}_{r_0}^{*}|^2 \bigg)^{1/2} \\
    & \le C\bigg( \dashint_{B_{2r,+}} |\nabla u - \nabla \overline{W}_{r_0}^{*}|^2 \bigg)^{1/2} \\
    & \le C_m \Big( \frac{r}{R} \Big)^{m} \bigg( \dashint_{B_{R,+}} |\nabla u|^2 \bigg)^{1/2},
\end{aligned}
\end{equation}
where we have used \eqref{est.Du-DbarW} in the last inequality. Finally, using the fact that
\begin{equation}
    \inf_{L\in \R} \bigg( \dashint_{S} |f - L|^2 \bigg)^{1/2} = \bigg( \dashint_{S} |f - \dashint_{S} f|^2 \bigg)^{1/2}
\end{equation}
for any bounded measurable set $S$ and scalar function $f\in L^2(S)$, we have
\begin{equation}
\begin{aligned}
    \bigg( \dashint_{B_{r,+}} |p - \overline{\Pi}_{r_0}^{*} -  \dashint_{B_{r,+}} (p - \overline{\Pi}_{r_0}^{*})|^2 \bigg)^{1/2} & = \inf_{L\in R} \bigg( \dashint_{B_{r,+}} |p - \overline{\Pi}_{r_0}^{*} -  L|^2 \bigg)^{1/2} \\
    & \le C \inf_{L\in R} \bigg( \dashint_{\Omega_{r}} |p - \overline{\Pi}_{r_0}^{*} -  L|^2 \bigg)^{1/2} \\
    & \le C_m \Big( \frac{r}{R} \Big)^{m} \bigg( \dashint_{B_{R,+}} |\nabla u|^2 \bigg)^{1/2},
\end{aligned}
\end{equation}
where we have used \eqref{est.p-Pi*.Or} in the last inequality. This ends the proof.
\end{proof}

\begin{proposition}\label{thm.main.pressue}
For any $r\in (1, R/2)$, 
\begin{equation*}
    \bigg( \dashint_{B_{r,+}} |p - \overline{\Pi}_{r_0}^{*} -  \dashint_{B_{1,+}} (p - \overline{\Pi}_{r_0}^{*})|^2 \bigg)^{1/2} \le C_m \Big( \frac{r}{R} \Big)^{m} \bigg( \dashint_{B_{R,+}} |\nabla u|^2 \bigg)^{1/2}. 
\end{equation*}
\end{proposition}

\begin{proof}
Lemma \ref{lem.pe.Br} implies
\begin{equation*}
    \sup_{r\le s,t\le 2r} \bigg| \dashint_{B_{s,+}} (p - \overline{\Pi}_{r_0}^{*}) - \dashint_{B_{t,+}} (p - \overline{\Pi}_{r_0}^{*}) \bigg| \le C_m \Big( \frac{r}{R} \Big)^{m} \bigg( \dashint_{B_{R,+}} |\nabla u|^2 \bigg)^{1/2}.
\end{equation*}
By a dyadic decomposition, this yields
\begin{equation*}
    \bigg|\dashint_{B_{r,+}} (p - \overline{\Pi}_{r_0}^{*}) 
    - \dashint_{B_{1,+}} (p - \overline{\Pi}_{r_0}^{*}) \bigg| 
    \le C_m \Big( \frac{r}{R} \Big)^{m} \bigg( \dashint_{B_{R,+}} |\nabla u|^2 \bigg)^{1/2}.
\end{equation*}
This, combined with Lemma \ref{lem.pe.Br} again and the triangle inequality, leads to the assertion.
\end{proof}

\begin{proofx}{Theorem \ref{thm.main}}
Let us set
\begin{equation*}
    c_p=\dashint_{B_{1,+}} (p - \overline{\Pi}_{r_0}^{*}). 
\end{equation*}
Then the estimate \eqref{est.mainRegularity} is simply a combination of Proposition \ref{thm.main.velocity} and Proposition \ref{thm.main.pressue}.
\end{proofx}

%
\subsection{Liouville theorems}\label{subsec.liouville}
%
As a application of the large-scale estimates, we have the Liouville theorem for
\begin{equation}\label{eq.Stokes.Omega}
\left\{
\begin{array}{ll}
-\Delta u+\nabla p=0&\mbox{in}\ \Omega\\
\nabla\cdot u=0&\mbox{in}\ \Omega\\
u=0&\mbox{on}\ \Gamma. 
\end{array}
\right.
\end{equation}
\begin{proposition}\label{thm.Liouville}
Let $m\in \Z_{+}$ and $(u,p)$ be a weak solution to \eqref{eq.Stokes.Omega} satisfying
\begin{equation}\label{eq.u.grow}
    \lim_{R\to \infty} R^{-m} \| \nabla u\|_{\ul{L}^2(B_{R,+})} = 0.
\end{equation}
Then $(u,p) \in \mathcal{S}_m(\Omega)$. 
\end{proposition}

\begin{proof}
Let $r>1$ be arbitrary and $R>r$. By Theorem \ref{thm.main} and the Poincar\'{e} inequality,
\begin{equation*}
    \inf_{\substack{(w,\pi)\in \mathcal{S}_m(\Omega)\\ L\in \R}} \Big\{ r^{-1} \|u- w\|_{\ul{L}^2(B_{r,+})}
+ \|p-\pi-L\|_{\ul{L}^2(B_{r,+})} \Big\}
\le 
C\Big(\frac{r}{R}\Big)^{m}
\|\nabla u\|_{\ul{L}^2(B_{R,+})}. 
\end{equation*}
Observe that the left-hand side is independent of $R$. Letting $R\to \infty$ and in view of \eqref{eq.u.grow}, we see that $(u,p) = (w,\pi + L)$ in $B_{r,+}$ for some $(w,\pi) \in \mathcal{S}_m(\Omega)$ and $L\in \R$. This implies the desired result by the uniqueness continuation for Stokes system. 
\end{proof}

Thanks to the theorem, we have the following intrinsic characterization of $\mathcal{S}_m(\Omega)$:
\begin{equation}
\begin{aligned}
\mathcal{S}_m(\Omega)= \{(u,p)
~|~&(u,p) \ \, \text{is a weak solution of \eqref{eq.Stokes.Omega} such that} \\
& \quad \| \nabla u\|_{\ul{L}^2(B_{R,+})}=o(R^m), \quad R\to\infty
\}.
\end{aligned}
\end{equation}
We note that the growth condition of $u$ is optimal. The result improves the Liouville theorem obtained in the previous work \cite{HPZ}, which was suboptimal about the growth condition and also restricted only to the orders $m=1,2$. 

%
\section{Effective Stokes polynomials}\label{sec.effective}
In this section, we will introduce the space of effective Stokes polynomials. As applications, we will prove Theorem \ref{thm.pointwise} in Subsection \ref{subsec.effective} and investigate the relation to the wall laws in Subsection \ref{subsec.wall.law}. A concrete example of the wall law will be given in Subsection \ref{subsec.Example}.

As in the previous section,  all the constants depending on $m\in\Z_{+}$ are denoted by $C_m$.

%
\subsection{Definition and proof of Theorem \ref{thm.pointwise}}\label{subsec.effective}
%
We first describe the effective polynomials for monomials. Let us take $(v^{\alpha}, q^{\alpha}) = (v^{\alpha,l}_{(i)}, q^{\alpha,l}_{(i)})$ given in Proposition \ref{thm.bl.alpha}, which is the boundary layer corrector for the monomial $x^\alpha y^l{\bf e}_i$. Set $(w,\pi)=(x^\alpha y^l{\bf e}_i,0) + (v^{\alpha,l}_{(i)}, q^{\alpha,l}_{(i)})$. Then by the decomposition \eqref{est2.thm.bl.alpha}, we have
\begin{equation*}
\begin{aligned}
    (w,\pi) &= (w_{{\rm poly}},\pi_{{\rm poly}}) + (w_{{\rm het}},\pi_{{\rm het}}),
\end{aligned}
\end{equation*}
where 
\begin{equation*}
\begin{aligned}
    (w_{{\rm poly}},\pi_{{\rm poly}}) 
    &:= (x^\alpha y^l{\bf e}_i + (v^{\alpha,l}_{(i)})_{{\rm poly}}, (q^{\alpha,l}_{(i)})_{{\rm poly}}), \\
    (w_{{\rm het}},\pi_{{\rm het}})
    &:=((v^{\alpha,l}_{(i)})_{{\rm het}}, (q^{\alpha,l}_{(i)})_{{\rm het}}).
\end{aligned}
\end{equation*}
Due to \eqref{est5.thm.bl.alpha}, the heterogeneous part $(w_{{\rm het}},\pi_{{\rm het}})$ decays exponentially as $y\to \infty$. Thus the polynomial part $(w_{\rm poly}, \pi_{\rm poly})$ approximates $(w,\pi)$ extremely well far away from the boundary. Based on this observation, we regard $(w_{{\rm poly}},\pi_{{\rm poly}})$ being effective when $y\gg1$ and will call it the effective part or the effective polynomial of $(w,\pi)$.

This notion is easily extended to the heterogeneous Stoke polynomials. Let $(w, \pi) = (w_P, \pi_P) \in \mathcal{S}_m(\Omega)$ be given as in \eqref{eq.wp.StokesP}. Then we define
\begin{equation}\label{eq.wp.poly}
\begin{aligned}
    w_{\rm poly} & = \sum_{\substack{|\alpha|+l\le m \\ l\ge 1}}\sum_{i=1}^{d} A^{\alpha,l}_i (x^\alpha y^l \mathbf{e}_i + (v^{\alpha,l}_{(i)})_{{\rm poly}}), \\
    \pi_{\rm poly} &= Q + \sum_{\substack{|\alpha|+l\le m \\ l\ge 1}}\sum_{i=1}^{d} A^{\alpha,l}_i (q^{\alpha,l}_{(i)})_{{\rm poly}}, 
\end{aligned}
\end{equation}
and
\begin{equation}\label{eq.wpi.het}
    \begin{aligned}
    w_{\rm het} = \sum_{\substack{|\alpha|+l\le m \\ l\ge 1}}\sum_{i=1}^{d} A^{\alpha,l}_i (v^{\alpha,l}_{(i)})_{{\rm het}}, \qquad
    \pi_{\rm het} = \sum_{\substack{|\alpha|+l\le m \\ l\ge 1}}\sum_{i=1}^{d} A^{\alpha,l}_i (q^{\alpha,l}_{(i)})_{{\rm het}}. 
\end{aligned}
\end{equation}
From the consideration above, one can regard $(w_{{\rm poly}},\pi_{{\rm poly}})$ as the effective part of $(w,\pi)$. Therefore, the space of effective Stokes polynomials is naturally defined by
\begin{equation}
    \mathcal{S}^{\rm eff}_m(\Omega) = \{(w_{\rm poly}, \pi_{\rm poly})~|~ (w, \pi) \in \mathcal{S}_m(\Omega) \}. 
\end{equation}

\begin{proposition}\label{prop.wpi.ptws}
Let $m\in\Z_{+}$ and let $r_0>1$ be the constant in Lemma \ref{lem.growth}. Then, for any $(w,\pi) \in \mathcal{S}_m(\Omega)$ and for any $(x,y) \in \Omega$ with $y\ge 4$,
\begin{equation}\label{est.w-wpoly.ptws}
\begin{aligned}
    & \sum_{j=0,1} | \nabla^j w(x,y) - \nabla^j w_{\rm poly}(x,y)| + |\pi(x,y) - \pi_{\rm poly}(x,y)|  \\
    & \qquad \le C_m \| \nabla w \|_{\ul{L}^2(B_{r_0,+})} (1+|x|)^{m-1} e^{-y/2}.
\end{aligned}
\end{equation}
\end{proposition}
\begin{proof}
Assume that $(w,\pi) = (w_P, \pi_P)$ takes the form of \eqref{eq.wp.StokesP}. Then, it follows from the definitions \eqref{eq.wp.poly} and \eqref{eq.wpi.het} as well as the estimate \eqref{est5.thm.bl.alpha} that, when $y\ge4$, 
\begin{equation}
\begin{aligned}
    & \sum_{j=0,1} | \nabla^j w(x,y) - \nabla^j w_{\rm poly}(x,y)| + |\pi(x,y) - \pi_{\rm poly}(x,y)| \\
    & \qquad \le C_m \sum_{\substack{|\alpha|+l\le m \\ l\ge 1}}\sum_{i=1}^{d} |A^{\alpha,l}_i| (1+|x|)^{|\alpha|} e^{-y/2} \\
    & \qquad \le C_m (1+|x|)^{m-1} e^{-y/2} \sum_{\substack{|\alpha|+l\le m \\ l\ge 1}}\sum_{i=1}^{d} |A^{\alpha,l}_i| \\
    & \qquad \le C_m (1+|x|)^{m-1} e^{-y/2}  \| \nabla w \|_{\ul{L}^2(B_{r_0,+})},
\end{aligned}
\end{equation}
where the last inequality follows from Lemma \ref{lem.growth}. The proof is complete. 
\end{proof}

Now we prove Theorem \ref{thm.pointwise} using Proposition \ref{prop.wpi.ptws}.

\begin{proofx}{Theorem \ref{thm.pointwise}}
Let $r_0>1$ be the constant in Lemma \ref{lem.growth} and let $R>2r_0$.
First of all, by Theorem \ref{thm.main}, there exist $(w,\pi) \in \mathcal{S}_m(\Omega)$ and a constant $c_p$ such that
\begin{equation}\label{est.u-w.p-pi}
    \|\nabla u- \nabla w\|_{\ul{L}^2(B_{r,+})}
+ \|p-\pi-c_p\|_{\ul{L}^2(B_{r,+})}
\le 
C\Big(\frac{r}{R}\Big)^{m}
\|\nabla u\|_{\ul{L}^2(B_{R,+})}.
\end{equation}
Let $(x,y) \in B_{R/2,+}$ and $y\ge4$. Put $r := |(x,y)|$. Without loss of generality, assume $1<r<R/5$. Then the interior and boundary Lipschitz estimates (see, e.g., \eqref{est.StokesLip}) for the solution $(u-w, p-\pi-c_p)$ implies
\begin{equation}\label{est.Du-Dw.pt}
\begin{aligned}
    |\nabla u(x,y) - \nabla w(x,y)| 
    & \le C \bigg( \dashint_{Q_{y/2}(x,y)} |\nabla (u-w)|^2 \bigg)^{1/2} \\
    & \le C \bigg( \dashint_{B_{2y,+}(x,0)} |\nabla (u-w)|^2 \bigg)^{1/2} \\
    & \le C \bigg( \dashint_{B_{2r,+}(x,0)} |\nabla (u-w)|^2 \bigg)^{1/2} \\
    & \le C \bigg( \dashint_{B_{3r,+}} |\nabla (u-w)|^2 \bigg)^{1/2} \\
    & \le C_m \Big(\frac{r}{R}\Big)^{m}
\|\nabla u\|_{\ul{L}^2(B_{R,+})}.
\end{aligned}
\end{equation}

To estimate the pressure, since $p-\pi - c_p$ is harmonic in $\Omega$, the interior estimate (such as the mean value property) implies that
\begin{equation*}
\begin{aligned}
    |p(x,y) - \pi(x,y) - c_p| & \le C \bigg( \dashint_{Q_{y/2}(x,y)} |p-\pi - c_p|^2 \bigg)^{1/2} \\
    & \le C \bigg( \dashint_{B_{2y,+}(x,0)} |p-\pi - c_p|^2 \bigg)^{1/2} \\
    & \le C\bigg( \dashint_{B_{2y,+}(x,0)} |p-\pi - c_p - \dashint_{B_{2r,+}(x,0)} (p-\pi-c_p) |^2 \bigg)^{1/2} \\
    & \qquad + C\bigg| \dashint_{B_{2r,+}(x,0)} (p-\pi-c_p) \bigg| \\
    & \le C\bigg( \dashint_{B_{4r,+}(x,0)} |\nabla (u-w)|^2 \bigg)^{1/2} + C\|p-\pi-c_p\|_{\ul{L}^2(B_{4r,+})} \\
    & \le C\Big(\frac{r}{R}\Big)^{m}
\|\nabla u\|_{\ul{L}^2(B_{R,+})},
\end{aligned}
\end{equation*}
where we have used \eqref{est.StokesLip} and \eqref{est.u-w.p-pi} in the last two inequalities.

On the other hand, let $(w_{\rm poly}, \pi_{\rm poly})$ be the effective Stokes polynomial of $(w,\pi)$. Then Proposition \ref{prop.wpi.ptws} implies
\begin{equation*}
\begin{aligned}
    & |\nabla w(x,y) - \nabla w_{\rm poly}(x,y)| + |\pi(x,y) - \pi_{\rm poly}(x,y) | \\
    & \qquad \le C_m \| \nabla w \|_{\ul{L}^2(B_{r_0,+})} (1+|x|)^{m-1} e^{-y/2} \\
    & \qquad \le C_m \| \nabla w \|_{\ul{L}^2(B_{R/2,+})} (1+|x|)^{m-1} e^{-y/2} \\
    & \qquad \le C_m \| \nabla u \|_{\ul{L}^2(B_{R,+})} (1+|x|)^{m-1} e^{-y/2},
\end{aligned}
\end{equation*}
where we have used \eqref{est.u-w.p-pi} in the last inequality. Combining the above three estimates, we obtain the desired estimate \eqref{est.AporByEff} by the triangle inequality.

Finally, using the interior and boundary Lipschitz estimates, as well as the Poincar\'{e} inequality, we have
\begin{equation*}
    \begin{aligned}
    | u(x,y) - w(x,y)| & \le C \bigg( \dashint_{B_{y/2}(x,y)} | u-w|^2 \bigg)^{1/2} \\
    & \le C \bigg( \dashint_{B_{2y,+}(x,0)} | u-w|^2 \bigg)^{1/2} \\
    & \le Cy \bigg( \dashint_{B_{2y,+}(x,0)} |\nabla (u-w)|^2 \bigg)^{1/2} \\
    & \le  C_m y \Big(\frac{r}{R}\Big)^{m}
\|\nabla u\|_{\ul{L}^2(B_{R,+})},
    \end{aligned}
\end{equation*}
where we have used the similar argument as \eqref{est.Du-Dw.pt} in the last inequality. This, combined with \eqref{est.w-wpoly.ptws} again, leads to  \eqref{est.AporByEff.1}. Thus, the proof is complete.
\end{proofx}

\subsection{Higher-order wall laws}\label{subsec.wall.law}
Our aim here is to identify the (local) wall laws described in the introduction. As seen from Proposition \ref{prop.HigherWallLaw} below, these depend only on the structure of $\partial \Omega$, not on the specific elements in $\mathcal{S}_m^{\rm eff}(\Omega)$. This indicates that the wall laws are intrinsic to the domain $\Omega$.

Firstly we describe how the Navier wall law will be derived in our context. Let us take $(v^{\alpha}, q^{\alpha}) = (v^{\alpha,l}_{(i)}, q^{\alpha,l}_{(i)})$ in Proposition \ref{thm.bl.alpha}. The general element in $\mathcal{S}_1^{\rm eff}(\Omega)$ takes a form of
\begin{equation*}
w_{\rm poly}(x,y) = \sum_{i=1}^{d-1} c_i (y \mathbf{e}_i + T_{(i)}),
\qquad T_{(i)}:=\lim_{y\to0} v^{0,1}_{(i)}(x,y), 
\end{equation*}
where $T_{(i)}$ is the constant vector known as the first-order boundary layer tail. As a side note, we mention that $T_{(i)}\cdot{\bf e}_d=0$ by the divergence-free condition. Now if we set 
\begin{equation}\label{def.Phi.01}
    \Phi^{0,1} = (T_{(1)}, T_{(2)},\cdots, T_{(d-1)}, 0) \in \R^{d\times d}, 
\end{equation}
then we recover the classical Navier wall law, or the Navier slip condition, as
\begin{equation}\label{eq.walllaw.1st}
    w_{\rm poly}(x,0) = \Phi^{0,1} \partial_y w_{\rm poly}(x,0)
    \quad \text{for any }
    (w_{\rm poly}, \pi_{\rm poly}) \in \mathcal{S}_1^{\rm eff}(\Omega).
\end{equation}

Next let $W^{\alpha,l}_{(i)}$ denote the effective polynomial of $x^\alpha y^l \mathbf{e}_i + v^{\alpha,l}_{(i)}$ in the meaning of Subsection \ref{subsec.effective}. Then we define the matrix-valued polynomial $W^{\alpha,l}=W^{\alpha,l}(x,y)$ by
$$W^{\alpha,l} = (W^{\alpha,l}_{(1)}, W^{\alpha,l}_{(2)},\cdots, W^{\alpha,l}_{(d)}).$$ 
We also define the matrix-valued polynomial $\Phi^{\alpha,l}=\Phi^{\alpha,l}(x)$ recursively. Starting from $\Phi^{0,1}$ in \eqref{def.Phi.01}, for $(\alpha,l)$ with $|\alpha|+l=m$, we define
\begin{equation}\label{eq.Phi.betak}
    \Phi^{\alpha,l}(x) 
    = \frac{1}{\alpha! l!} \bigg(W^{\alpha,l}(x,0) 
    - \sum_{\substack{|\beta|+k\le m-1 \\ \beta\ge 0, k\ge1}} 
    \Phi^{\beta,k}(x) 
    \partial_x^\beta \partial_y^k W^{\alpha,l}(x,0)
    \bigg). 
\end{equation}

The higher-order wall laws are described by $\{\Phi^{\alpha,l}\}_{\alpha\ge0,l\ge1}$. Indeed, we have

\begin{proposition}\label{prop.HigherWallLaw}
Let $m\in\Z_{+}$. Then, for any $(w_{\rm poly}, \pi_{\rm poly}) \in \mathcal{S}_m^{\rm eff}(\Omega)$, 
\begin{equation}\label{eq.walllaw}
\begin{aligned}
    w_{\rm poly}(x,0) 
    &= \sum_{\substack{|\alpha|+l \le m \\ \alpha\ge 0, l\ge 1}} 
    \Phi^{\alpha,l}(x) 
    \partial_x^\alpha \partial_y^l 
    w_{\rm poly}(x,0). 
\end{aligned}
\end{equation}
\end{proposition}
\begin{proof}
The proof uses an induction on the order $m$. We already proved the case when $m = 1$ in \eqref{def.Phi.01} and \eqref{eq.walllaw.1st}. Let us assume that \eqref{eq.walllaw} holds for $m-1$. Let $(w_{\rm poly}, \pi_{\rm poly}) \in \mathcal{S}_m^{\rm eff}(\Omega)$. Since the highest degree of $w_{\rm poly}$ is $m$, we can easily find its leading term as 
\begin{equation*}
    \frac{1}{\alpha! l!}
    \sum_{\substack{|\alpha|+l = m \\ \alpha\ge 0, l\ge 1}} 
    W^{\alpha,l}(x,y) \partial_x^\alpha \partial_y^l w_{\rm poly}(x,0).
\end{equation*}
Define
\begin{equation}\label{eq.wpoly.m-1}
    \widetilde{w}_{\rm poly}(x,y) 
    = w_{\rm poly}(x,y) 
    - \frac{1}{\alpha! l!}
    \sum_{\substack{|\alpha|+l = m \\ \alpha\ge 0, l\ge 1}} 
    W^{\alpha,l}(x,y) \partial_x^\alpha \partial_y^l w_{\rm poly}(x,0).
\end{equation}
Similarly, we can define the corresponding pressure part $\widetilde{\pi}_{\rm poly}(x,y)$. Since 
$$(\widetilde{w}_{\rm poly}(x,y), \widetilde{\pi}_{\rm poly}(x,y)) \in \mathcal{S}_{m-1}^{\rm eff}(\Omega),$$ 
by the inductive assumption, one has 
\begin{equation*}
\begin{aligned}
    \widetilde{w}_{\rm poly}(x,0)
    &= \sum_{\substack{|\beta|+k \le m-1 \\ \beta\ge 0, k\ge 1}} 
    \Phi^{\beta,k}(x) 
    \partial_x^\beta \partial_y^k 
    \widetilde{w}_{\rm poly}(x,0). 
\end{aligned}
\end{equation*}
Combined with \eqref{eq.wpoly.m-1}, we have
\begin{equation}\label{eq.wpoly.m-2}
\begin{aligned}
    \widetilde{w}_{\rm poly}(x,0)
    &= \sum_{\substack{|\beta|+k \le m-1 \\ \beta\ge 0, k\ge 1}} 
    \Phi^{\beta,k}(x) 
    \partial_x^\beta \partial_y^k 
    w_{\rm poly}(x,0) \\
    &\quad
    - \frac{1}{\alpha! l!}
    \sum_{\substack{|\alpha|+l = m \\ \alpha\ge 0, l\ge 1}} 
    \sum_{\substack{|\beta|+k \le m-1 \\ \beta\ge 0, k\ge 1}}
    \Phi^{\beta,k}(x) 
    \partial_x^\beta \partial_y^k W^{\alpha,l}(x,0) 
    \partial_x^\alpha \partial_y^l w_{\rm poly}(x,0). 
\end{aligned}
\end{equation}
Again by \eqref{eq.wpoly.m-1}, \eqref{eq.wpoly.m-2}, and the definition of $\Phi^{\alpha,l}$ with $|\alpha|+l=m$ in \eqref{eq.Phi.betak}, we obtain 
\begin{equation}
\begin{aligned}
    &w_{\rm poly}(x,0) \\
    &= 
    \widetilde{w}_{\rm poly}(x,0)
    + \frac{1}{\alpha! l!}
    \sum_{\substack{|\alpha|+l = m \\ \alpha\ge 0, l\ge 1}} 
    W^{\alpha,l}(x,0) \partial_x^\alpha \partial_y^l w_{\rm poly}(x,0)  \\
    &= 
    \sum_{\substack{|\beta|+k \le m-1 \\ \beta\ge 0, k\ge 1}} 
    \Phi^{\beta,k}(x) 
    \partial_x^\beta \partial_y^k w_{\rm poly}(x,0) \\
    &\quad 
    + \frac{1}{\alpha! l!}
    \sum_{\substack{|\alpha|+l = m \\ \alpha\ge 0, l\ge 1}} 
    \bigg(
    W^{\alpha,l}(x,0) 
    - \sum_{\substack{|\beta|+k \le m-1 \\ \beta\ge 0, k\ge 1}}
    \Phi^{\beta,k}(x)
    \partial_x^\beta \partial_y^k W^{\alpha,l}(x,0) 
    \bigg)
    \partial_x^\alpha \partial_y^l w_{\rm poly}(x,0) \\
    &= \sum_{\substack{|\alpha|+l \le m \\ \alpha\ge 0, l\ge 1}} 
    \Phi^{\alpha,l}(x) 
    \partial_x^\alpha \partial_y^l 
    w_{\rm poly}(x,0). 
\end{aligned}
\end{equation}
This is the desired relation \eqref{eq.walllaw}. The proof is complete. 
\end{proof}

\subsection{An example for $d= m=2$}\label{subsec.Example}
To provide a concrete example, we describe the second-order wall law for two-dimensional flows in $\Omega$. A basis of the space $\mathcal{S}_2(\R^{2}_+)$ can be found as  
\begin{equation*}
\begin{aligned}
(0,1), \quad (y {\bf e}_1, 0), \quad (y^2 {\bf e}_1, 2x), \quad (-2xy{\bf e_1} + y^2{\bf e_2}, 2y). 
\end{aligned}
\end{equation*}

We would like to make explicit the elements in $\mathcal{S}_2^{\rm eff}(\Omega)$, which is nothing more than determining the effective polynomial corresponding to the following velocity field 
\begin{equation}\label{eq0.subsec.Example}
\begin{aligned}
u(x,y):=c_1 y {\bf e}_1 + c_2 y^2 {\bf e}_1 + c_3(-2xy{\bf e_1} + y^2{\bf e_2}),
\quad c_i\in\R.
\end{aligned}
\end{equation}

Let $\alpha\in\Z_{\ge0}$, $l\in\Z_{+}$ and $i\in\{1,2\}$. By Proposition \ref{thm.bl.alpha}, the polynomial part of the boundary layer $v^{\alpha,l}_{(i)}$ correcting $x^\alpha y^l {\bf e}_i$ is given by 
\begin{equation*}
\begin{aligned}
&(v^{\alpha,l}_{(i)})_{\rm poly}(x,y) =
\sum_{\beta=0}^{\alpha} \dbinom{\alpha}{\beta} x^{\alpha-\beta} 
(V^{\beta,l}_{(i)})_{\rm poly}(y)
\end{aligned}
\end{equation*}
with $(V^{\beta,l}_{(i)})_{\rm poly}(y)=V^{\beta}_{\rm poly}(y)$ determined recursively by
\begin{equation}\label{eq1.subsec.Example}
V^{0}_{\rm poly}(y) 
= V^{0}_{\rm poly}
=
\left\{ 
\begin{array}{ll}
(V^{0}_{\rm const})_1 {\bf e}_1 \quad &\text{if } \ i=1, \\
V^{0}_{\rm const} \quad &\text{if } \ i=2,
\end{array} \right.
\quad \text{and} \quad
Q^{0}_{\rm poly}(y) = 0,
\end{equation}
and from the proof of Proposition \ref{prop.bl.beta}, 
\begin{equation}\label{eq2.subsec.Example}
\begin{aligned}
V^{\beta}_{\rm poly}(y) 
&= V^{\beta}_{\rm const} 
- \bigg(\int_{0}^{y} \dbinom{\beta}{1} (V^{\beta-1}_{\rm poly})_1(t)\dd t \bigg) {\bf e}_2 \\
&\quad 
+ \bigg(\int_{0}^{y}\int_{0}^{t} 
\bigg(-2\dbinom{\beta}{2} (V^{\beta-2}_{\rm poly})_1(s)
+ \dbinom{\beta}{1} Q^{\beta-1}_{\rm poly}(s) \bigg) \dd s\dd t\bigg) {\bf e}_1, \\
Q^{\beta}_{\rm poly}(y) 
&= \bigg(\int_{0}^{y} 2\dbinom{\beta}{2} (V^{\beta-2}_{\rm poly})_2(t)\dd t\bigg)
-\dbinom{\beta}{1} (V^{\beta-1}_{\rm poly})_1(y), 
\end{aligned}
\end{equation}
where all the constants are gathered together and denoted by $V^{\beta}_{\rm const}$ in the first line of \eqref{eq2.subsec.Example}. In \eqref{eq1.subsec.Example}, the second component of $V^{0}_{\rm poly}$ is zero when $i=1$, because of the divergence-free condition and the boundary condition. The proof for the 3D case can be found in \cite[Proposition 3]{HP20}. The argument for the case of general dimension is the same.

The effective polynomials corresponding to $y {\bf e}_1$, $y^2 {\bf e}_1$, $y^2{\bf e_2}$ in \eqref{eq0.subsec.Example} are, respectively,
\begin{equation*}
\begin{aligned}
\big(y + ((V^{0,1}_{(1)})_{\rm const})_1\big) {\bf e}_1, \qquad
\big(y^2 + ((V^{0,2}_{(1)})_{\rm const})_1\big) {\bf e}_1, \qquad
y^2{\bf e_2} + (V^{0,2}_{(2)})_{\rm const}.
\end{aligned}
\end{equation*}
Thus we focus on the effective polynomial corresponding to $-2xy{\bf e_1}$ in \eqref{eq0.subsec.Example}, namely $-2(xy{\bf e_1}+(v^{1,1}_{(1)})_{\rm poly})$. By a direct computation, one has 
\begin{equation*}
\begin{aligned}
(v^{1,1}_{(1)})_{\rm poly}(x,y) 
&= x (V^{0,1}_{(1)})_{\rm poly}(y) 
+ (V^{1,1}_{(1)})_{\rm poly}(y) \\
&= x ((V^{0,1}_{(1)})_{\rm const})_1 {\bf e}_1 
+ (V^{1,1}_{(1)})_{\rm const}
- y ((V^{0,1}_{(1)})_{\rm const})_1{\bf e}_2. 
\end{aligned}
\end{equation*}
Therefore, we have 
\begin{equation*}
\begin{aligned}
&-2\big(xy{\bf e_1} 
+ (v^{1,1}_{(1)})_{\rm poly}(x,y)\big) \\
&= 
-2\big(xy 
+ x ((V^{0,1}_{(1)})_{\rm const})_1\big){\bf e_1}
+ 2y ((V^{0,1}_{(1)})_{\rm const})_1 {\bf e}_2 
- 2 (V^{1,1}_{(1)})_{\rm const}. 
\end{aligned}
\end{equation*}
Combining the above, we see that the general element in $\mathcal{S}_2^{\rm eff}(\Omega)$ is written as 
\begin{equation*}
\begin{aligned}
w(x,y) 
&=
c_1\big(y + ((V^{0,1}_{(1)})_{\rm const})_1\big) {\bf e}_1
+ c_2\big(y^2 + ((V^{0,2}_{(1)})_{\rm const})_1\big) {\bf e}_1 \\
&\quad
- 2c_3\big(xy 
+ x ((V^{0,1}_{(1)})_{\rm const})_1\big){\bf e_1}
+ 2c_3y ((V^{0,1}_{(1)})_{\rm const})_1 {\bf e}_2 
+ c_3 y^2{\bf e_2} \\
&\quad
- 2c_3 (V^{1,1}_{(1)})_{\rm const}
+ c_3(V^{0,2}_{(2)})_{\rm const}, \quad c_i\in\R.
\end{aligned}
\end{equation*}

Next we derive the second-order wall law described by $w$. Taking the trace of $w(x,y)$ on $\{y=0\}$, we find 
\begin{equation}\label{eq3.subsec.Example}
\begin{aligned}
w(x,0) 
&=
c_1 ((V^{0,1}_{(1)})_{\rm const})_1 {\bf e_1}
+ c_2 ((V^{0,2}_{(1)})_{\rm const})_1 {\bf e_1} \\
&\quad 
+ c_3 \big(
- 2 x ((V^{0,1}_{(1)})_{\rm const})_1 {\bf e_1} 
- 2 (V^{1,1}_{(1)})_{\rm const}
+  (V^{0,2}_{(2)})_{\rm const} 
\big).
\end{aligned}
\end{equation}
Moreover, computing the derivatives, 
\begin{equation*}
\begin{aligned}
\partial_y \partial_x w(x,0) &= -2c_3 {\bf e_1}, \\
\partial_y^2 w(x,0) &= 2c_2 {\bf e_1} + 2c_3 {\bf e_2}, \\
\partial_y w(x,0) &= (c_1 - 2c_3 x) {\bf e_1} 
+ 2c_3 
((V^{0,1}_{(1)})_{\rm const})_1 {\bf e_2}.
\end{aligned}
\end{equation*}
The latter system can be solved for $(c_1,c_2,c_3)$:
\begin{equation*}
\begin{aligned}
c_1 &= \partial_y w_1(x,0) + 2c_3x = \partial_y w_1(x,0) - x \partial_y \partial_x w_1(x,0), \\
c_2 &= \frac{1}{2} \partial_y^2 w_1(x,0), \\
c_3 &= -\frac{1}{2} \partial_y \partial_x w_1(x,0). 
\end{aligned}
\end{equation*}
Inserting this into
\eqref{eq3.subsec.Example}, we obtain the following second-order wall law: 
\begin{equation*}
\begin{aligned}
w(x,0) 
&= 
\big(\partial_y w_1(x,0) - x \partial_y \partial_x w_1(x,0)\big) 
((V^{0,1}_{(1)})_{\rm const})_1 {\bf e_1} \\
&\quad
+ \frac{1}{2} \partial_y^2 w_1(x,0) 
((V^{0,2}_{(1)})_{\rm const})_1 {\bf e_1} \\
&\quad
-\frac{1}{2} \partial_y \partial_x w_1(x,0) 
\big(
- 2 x ((V^{0,1}_{(1)})_{\rm const})_1 {\bf e_1} 
- 2 (V^{1,1}_{(1)})_{\rm const}
+  (V^{0,2}_{(2)})_{\rm const} 
\big) \\
&= 
\Big(
((V^{0,1}_{(1)})_{\rm const})_1
\partial_y w_1(x,0)
+ \frac{1}{2} 
((V^{0,2}_{(1)})_{\rm const})_1
\partial_y^2 w_1(x,0) 
\Big) {\bf e_1} \\
&\quad
-\frac{1}{2} 
\big(
- 2 (V^{1,1}_{(1)})_{\rm const}
+  (V^{0,2}_{(2)})_{\rm const} 
\big)
\partial_y \partial_x w_1(x,0). 
\end{aligned}
\end{equation*}

In applications, the roughness of the boundary, i.e., amplitude and wavelength, are assumed to be small. Denoting it by $\ep$, after rescaling, the second-order wall law reads as 
\begin{equation}\label{eq4.subsec.Example}
\begin{aligned}
w(x,0) 
&= 
\Big(
\ep ((V^{0,1}_{(1)})_{\rm const})_1
\partial_y w_1(x,0)
+ \frac{\ep^2}{2} 
((V^{0,2}_{(1)})_{\rm const})_1
\partial_y^2 w_1(x,0) 
\Big)
{\bf e_1} \\
&\quad
-\frac{\ep^2}{2} 
\big(
- 2 (V^{1,1}_{(1)})_{\rm const}
+  (V^{0,2}_{(2)})_{\rm const} 
\big)
\partial_y \partial_x w_1(x,0).
\end{aligned}
\end{equation}
If we set $\lambda=((V^{0,1}_{(1)})_{\rm const})_1$ and neglect the terms of order $\ep^2$, then \eqref{eq4.subsec.Example} becomes 
\begin{equation*}
\begin{aligned}
w_1(x,0) = \ep \lambda \partial_y w_1(x,0), \qquad
w_2(x,0) = 0. 
\end{aligned}
\end{equation*}
It implies that the principal part of the the second-order wall law is the Navier wall. It is worth noting that $\lambda$ is known to be positive. This is called the slip length and plays an important role especially in the energy computation; see \cite{BDG12, H16} for details.

In contrast, the sign of $((V^{0,2}_{(1)})_{\rm const})_1$ does not seem to be known, which is important when one studies the second-order wall law in the global stationary/nonstationary settings. We refer to the numerical analysis in \cite[Section 6]{BM10} for the case of the Laplace equation. The results suggest that the sign of the corresponding constant is negative.

\section*{Appendix}
In this appendix, we verify that the pair $(V,Q)$ given by \eqref{est3.proof.prop.bl.F} is a solution of \eqref{eq1.proof.prop.bl.F}. First, a straightforward computation shows
\begin{equation}\label{apx.DeltaV}
    \Delta V = \sum_{k\in \mathbb{Z}^{d-1}\setminus \{0\} } \big\{ (\partial_z^2 \mathcal{V}_k)(y-L) -2|k|(\partial_z \mathcal{V}_k)(y-L)  \big\} e^{-|k|(y-L)} e^{\ii k\cdot x},
\end{equation}
and
\begin{equation}\label{apx.DQ}
    \nabla Q = \sum_{k\in \mathbb{Z}^{d-1}\setminus \{0\} } \bigg\{ \begin{bmatrix} \ii k \\ - |k| \\ \end{bmatrix} \mathcal{Q}_k(y-L) + \begin{bmatrix} 0 \\ (\partial_z \mathcal{Q}_k)(y-L) \\ \end{bmatrix}  \bigg\} e^{-|k|(y-L)} e^{\ii k\cdot x}.
\end{equation}
Then, by \eqref{est4.proof.prop.bl.F}, we compute
\begin{equation}
    (\partial_z \mathcal{V}_k)(y-L) = \frac{c_k}{|k|} \begin{bmatrix} \ii k \\ - |k| \\ \end{bmatrix} + (\partial_z \ul{\mathcal{V}}_k)(y-L),
\end{equation}
\begin{equation}
    (\partial_z^2 \mathcal{V}_k)(y-L) = (\partial_z^2 \ul{\mathcal{V}}_k)(y-L).
\end{equation}
Consequently,
\begin{equation}\label{apx.dVk-2}
\begin{aligned}
    & (\partial_z^2 \mathcal{V}_k)(y-L) - 2|k| (\partial_z \mathcal{V}_k)(y-L) \\
    & \qquad = (\partial_z^2 \ul{\mathcal{V}}_k)(y-L) - 2|k| (\partial_z \ul{\mathcal{V}}_k)(y-L) - 2c_k \begin{bmatrix} \ii k \\ - |k| \\ \end{bmatrix}.
\end{aligned}
\end{equation}
Next, we compute
\begin{equation}\label{apx.dVk-3}
    \partial_z \ul{\mathcal{V}}_k(z) = -\int_z^\infty \bigg\{ -\mathcal{F}_k(w) + \begin{bmatrix} \ii k \\ - |k| \\ \end{bmatrix} \ul{\mathcal{Q}}_k(w) +  \begin{bmatrix} 0 \\ (\partial_z \mathcal{Q}_k)(w) \\ \end{bmatrix}  \bigg\} e^{2|k|(z-w)} \dd w,
\end{equation}
and thus
\begin{equation}\label{apx.dVk-4}
    \partial_z^2 \ul{\mathcal{V}}_k(z) = \bigg\{ -\mathcal{F}_k(z) + \begin{bmatrix} \ii k \\ - |k| \\ \end{bmatrix} \ul{\mathcal{Q}}_k(z) +  \begin{bmatrix} 0 \\ (\partial_z \mathcal{Q}_k)(z) \\ \end{bmatrix}  \bigg\} + 2|k| \partial_z \ul{\mathcal{V}}_k(z).
\end{equation}
From \eqref{apx.dVk-2} and \eqref{apx.dVk-4}, we have
\begin{equation}
    \begin{aligned}
       & (\partial_z^2 \mathcal{V}_k)(y-L) - 2|k| (\partial_z \mathcal{V}_k)(y-L) \\
       & \quad = -\mathcal{F}_k(y-L) + \begin{bmatrix} \ii k \\ - |k| \\ \end{bmatrix} (-2c_k + \ul{\mathcal{Q}}_k(y-L) ) + \begin{bmatrix} 0 \\ (\partial_z \mathcal{Q}_k)(y-L) \\ \end{bmatrix}  \\
       & \quad = -\mathcal{F}_k(y-L) + \begin{bmatrix} \ii k \\ - |k| \\ \end{bmatrix}  {\mathcal{Q}}_k(y-L) + \begin{bmatrix} 0 \\ (\partial_z \mathcal{Q}_k)(y-L) \\ \end{bmatrix}
    \end{aligned}
\end{equation}
Substituting this into \eqref{apx.DeltaV}, we obtain
\begin{equation}
\begin{aligned}
    &-\Delta V \\
    & = \sum_{k\in \mathbb{Z}^{d-1}\setminus \{0\} } \bigg\{ \mathcal{F}_k(y-L) - \begin{bmatrix} \ii k \\ - |k| \\ \end{bmatrix}  {\mathcal{Q}}_k(y-L) - \begin{bmatrix} 0 \\ (\partial_z \mathcal{Q}_k)(y-L) \\ \end{bmatrix} \bigg\} e^{-|k|(y-L)} e^{\ii k\cdot x}.
\end{aligned}
\end{equation}
This, combined with \eqref{apx.DQ} and \eqref{assump1.prop.bl.F}, leads to
\begin{equation}\label{apx.VQ=F}
    -\Delta V + \nabla Q = \sum_{k\in \mathbb{Z}^{d-1}\setminus \{0\} } \mathcal{F}_k(y-L) e^{-|k|(y-L)} e^{\ii k\cdot x} = F.
\end{equation}
This proves the first equation in \eqref{eq1.proof.prop.bl.F}.

Now, we show $\nabla\cdot V = 0$. To this end, we first calculate
\begin{equation}\label{apx.DeltaQ}
    \Delta Q = \sum_{k\in \mathbb{Z}^{d-1}\setminus \{0\} } \big\{ (\partial_z^2 \mathcal{Q}_k)(y-L) - 2|k|(\partial_z \mathcal{Q}_k)(y-L) \big\} e^{-|k|(y-L)} e^{\ii k\cdot x}.
\end{equation}
Moreover,
\begin{equation}
    \partial_z \mathcal{Q}_k(z) = -\int_z^\infty \bigg\{ \begin{bmatrix} \ii k \\ - |k| \\ \end{bmatrix}\cdot \mathcal{F}_k(w) + (\partial_z \mathcal{F}_k)_d(w) \bigg\} e^{2|k|(z-w)} \dd w,
\end{equation}
\begin{equation}\label{apx.dz2Qk}
    \partial_z^2 \mathcal{Q}_k(z) = \begin{bmatrix} \ii k \\ - |k| \\ \end{bmatrix} \cdot \mathcal{F}_k(z) + (\partial_z \mathcal{F}_k)_d(z) + 2|k| (\partial_z \mathcal{Q}_k)(z).
\end{equation}
Combining \eqref{apx.DeltaQ} and \eqref{apx.dz2Qk}, we have
\begin{equation}\label{apx.DeltaQ=DF}
\begin{aligned}
    \Delta Q & = \sum_{k\in \mathbb{Z}^{d-1}\setminus \{0\} } \bigg\{ \begin{bmatrix} \ii k \\ - |k| \\ \end{bmatrix} \cdot \mathcal{F}_k(y-L) + (\partial_z \mathcal{F}_k)_d(y-L) \bigg\} e^{-|k|(y-L)} e^{\ii k\cdot x} \\
    & =
    \nabla \cdot F.
\end{aligned}
\end{equation}
Hence, \eqref{apx.VQ=F} and \eqref{apx.DeltaQ=DF} together yield that $\nabla\cdot V$ is harmonic, namely,
\begin{equation}\label{apx.DeltaDV}
    \Delta (\nabla\cdot V) = 0.
\end{equation}
To show $\nabla\cdot V = 0$, it suffices to show $\nabla\cdot V(\cdot, L) = 0$ and $\lim_{y\to \infty}\nabla\cdot V(\cdot,y) = 0$. In order to see this, we compute
\begin{equation}\label{apx.DV}
    \nabla\cdot V = \sum_{k\in \mathbb{Z}^{d-1}\setminus \{0\} } \bigg\{ \begin{bmatrix} \ii k \\ - |k| \\ \end{bmatrix}\cdot \mathcal{V}_k(y-L) + \begin{bmatrix} 0 \\ 1 \\ \end{bmatrix}\cdot (\partial_z \mathcal{V}_k)(y-L) \bigg\} e^{-|k|(y-L)} e^{\ii k\cdot x}.
\end{equation}
Using \eqref{est4.proof.prop.bl.F} and \eqref{def.ck}, we have
\begin{equation}\label{apx.DV-1}
\begin{aligned}
    & \begin{bmatrix} \ii k \\ - |k| \\ \end{bmatrix}\cdot \mathcal{V}_k(z) + \begin{bmatrix} 0 \\ 1 \\ \end{bmatrix}\cdot (\partial_z \mathcal{V}_k)(z) \\
    & = \begin{bmatrix} \ii k \\ - |k| \\ \end{bmatrix}\cdot \hat{V}_k - c_k + \begin{bmatrix} \ii k \\ - |k| \\ \end{bmatrix}\cdot (\ul{\mathcal{V}}_k(z) -  \ul{\mathcal{V}}_k(0) ) + \begin{bmatrix} 0 \\ 1 \\ \end{bmatrix} \cdot \partial_z \ul{\mathcal{V}}_k(z) \\
    & = \begin{bmatrix} \ii k \\ - |k| \\ \end{bmatrix}\cdot (\ul{\mathcal{V}}_k(z) - \ul{\mathcal{V}}_k(0)) + \begin{bmatrix} 0 \\ 1 \\ \end{bmatrix} \cdot ( \partial_z \ul{\mathcal{V}}_k(z) - \partial_z \ul{\mathcal{V}}_k(0) ).
\end{aligned}
\end{equation}
Clearly, \eqref{apx.DV} and \eqref{apx.DV-1} imply $\nabla\cdot V(\cdot, L) = 0$. On the other hand, notice that $\ul{\mathcal{V}}_k(z)$ is a polynomial of $z$, which means that \eqref{apx.DV-1} has at most polynomial growth as $z\to \infty$. This implies that $\nabla\cdot V(\cdot, y) \to 0$ as $y\to \infty$, due to the exponential decay factor $e^{-|k|(y-L)}$ in \eqref{apx.DV}. As a result, we derive $\nabla\cdot V = 0$ from \eqref{apx.DeltaDV}.

Finally, it is easy to see that $V(\cdot,L) = b$. Therefore, we have verified that $(V,Q)$ satisfies all the equations in \eqref{eq1.proof.prop.bl.F}.

\bibliography{Ref}
\bibliographystyle{plain}

\medskip

	\begin{flushleft}
	M. Higaki\\
	Department of Mathematics, 
	Kobe University, 1-1 Rokkodai, Nada-ku, Kobe 657-8501, 
	Japan.
	
	Email: higaki@math.kobe-u.ac.jp
	\end{flushleft}

\begin{flushleft}
J. Zhuge\\
Department of Mathematics,
University of Chicago,
Chicago, IL 60637,
USA.

Email: jpzhuge@uchicago.edu
\end{flushleft}

\medskip

\noindent \today

\end{document}